\documentclass[reqno,a4paper,12pt]{amsart}
\usepackage[utf8]{inputenc}
\usepackage[T1]{fontenc}
\usepackage{amsmath,amsfonts,amssymb,amsthm,mathtools}
\usepackage[dvipsnames]{xcolor}

\usepackage{tikz}
\usetikzlibrary{arrows.meta}
\definecolor{pbblue}{HTML}{000000}
\definecolor{pbteal}{HTML}{000000}
\definecolor{pbrust}{HTML}{000000}
\usepackage{enumitem}
\usepackage{booktabs}
\usepackage{array}
\usepackage{placeins}
\usepackage{aliascnt}
\usepackage{microtype}
\usepackage{hyperref}
\usepackage{cleveref}
\hypersetup{hidelinks}

\numberwithin{equation}{section}

\theoremstyle{plain}
\newtheorem{theorem}{Theorem}[section]
\newaliascnt{proposition}{theorem}
\newtheorem{proposition}[proposition]{Proposition}
\aliascntresetthe{proposition}
\newaliascnt{corollary}{theorem}
\newtheorem{corollary}[corollary]{Corollary}
\aliascntresetthe{corollary}
\newaliascnt{lemma}{theorem}
\newtheorem{lemma}[lemma]{Lemma}
\aliascntresetthe{lemma}

\theoremstyle{definition}
\newaliascnt{remark}{theorem}
\newtheorem{remark}[remark]{Remark}
\aliascntresetthe{remark}
\newaliascnt{definition}{theorem}

\aliascntresetthe{definition}

\crefname{section}{Section}{Sections}
\Crefname{section}{Section}{Sections}
\crefname{subsection}{Section}{Sections}
\Crefname{subsection}{Section}{Sections}
\crefname{theorem}{Theorem}{Theorems}
\Crefname{theorem}{Theorem}{Theorems}
\crefname{proposition}{Proposition}{Propositions}
\Crefname{proposition}{Proposition}{Propositions}
\crefname{lemma}{Lemma}{Lemmas}
\Crefname{lemma}{Lemma}{Lemmas}
\crefname{corollary}{Corollary}{Corollaries}
\Crefname{corollary}{Corollary}{Corollaries}
\crefname{remark}{Remark}{Remarks}
\Crefname{remark}{Remark}{Remarks}
\crefname{table}{Table}{Tables}
\Crefname{table}{Table}{Tables}

\newcommand{\R}{\mathbb R}

\newcommand{\C}{\mathbb C}
\newcommand{\Z}{\mathbb Z}
\newcommand{\T}{\mathbb T}
\newcommand{\ip}[2]{\left(#1,#2\right)}
\DeclareMathOperator{\dist}{dist}
\DeclareMathOperator{\supp}{supp} %

\DeclareMathOperator{\diag}{diag}
\DeclareMathOperator{\spanop}{span}
\DeclareMathOperator{\E}{\mathbb E}
\newcommand{\F}{\mathbb F}
\newcommand{\1}{\mathbf 1}
\newcommand{\norm}[1]{\left\lVert#1\right\rVert}
\newcommand{\pair}[2]{\left\langle#1,#2\right\rangle}

\title[Structured coordinate systems in $L^p(\mathbb{R})$]{Structured coordinate systems in $L^p(\R)$}

\author{Daniel~Freeman}
\address{Department of Mathematics and Statistics, St Louis University, St Louis, MO, USA.}
\email{daniel.freeman@slu.edu}
\author{Lukas~Liehr}
\address{Department of Mathematics, Bar-Ilan University, Ramat-Gan 5290002, Israel.}
\email{lukas.liehr@biu.ac.il}
\author{Mitchell~A.\ Taylor}

\address{Mathematical Institute, University of Oxford, Andrew Wiles Building, Radcliffe Observatory Quarter, Woodstock Road, Oxford, OX2 6GG, United Kingdom.}
\email{mitchell.taylor@maths.ox.ac.uk}


\begin{document}

\begin{abstract}
We construct Schauder bases of $L^p(\mathbb{R})$ with significant structural constraints on their elements. First, we construct non-negative Schauder bases of $L^p(\mathbb{R})$ for every $1 \leq p < \infty$. Although such bases cannot be unconditional, their basis constants can be made arbitrarily close to one. By duality, we also obtain bases with positive biorthogonal functionals and the same bounds for all $1\leq p<\infty$. Finally, we prove that $L^p(\mathbb{R})$ admits a Schauder basis consisting of translates of a single function for every $2\le p<\infty$. The generating function $f$ can be chosen to be the restriction to $\mathbb{R}$ of an entire function. For $p>2$, the translation parameters can be chosen from any prescribed unbounded subset of $\R$, a conclusion that fails for $p=2$.
\end{abstract}

\maketitle

\section{Introduction}\label{sec:intro}
In this article, we construct Schauder bases of $L^p(\R)$ whose
elements are either non-negative functions or translates of a single
function.
These problems are part of a broader search for useful coordinate systems
whose vectors have a prescribed structure. Wavelet bases, for example,
are built from translations and dilations of a single generating function,
so that their coefficients record behavior at different locations
and scales~\cite{Dau,Mallat}. Similarly, Gabor systems are built from time-frequency shifts~\cite{BHW,LW} of a single function. In both cases, the rigid structure
of the basis elements allows one to study properties of the
whole system via its generator.

Here, we go one step further and analyze coordinate systems for $L^p(\mathbb{R})$ consisting only of translates of a single
function. The question of whether a Schauder basis of translates exists was posed in the Hilbert space
setting by Olson and Zalik~\cite{OZ}, and was also recorded as an open problem by
Bruna, Olevskii, and Ulanovskii~\cite[p.~3]{BOU}. Such a basis must be \emph{conditional}~\cite{FOSZ,OSSZ,OZ}, i.e., there must exist rearrangements of the sequence which do not form a basis. Nevertheless, depending on the choice of $p$, there may exist unconditional
Schauder frames of translates.
Indeed, the first author, Odell, Schlumprecht, and Zs\'ak~\cite[Theorem~3.2]{FOSZ} constructed
unconditional Schauder frames of translates for $p>2$, using any
prescribed unbounded sequence of translation parameters.
Berasategui and Carando~\cite[Theorem~4.1]{BC} showed that, after passing to a
subsequence, the coefficient functionals can also be chosen with
norms bounded above and away from zero. On the other hand, Lev and
Tselishchev~\cite{LT-no} proved that unconditional Schauder frames
of translates do not exist for $1\le p\le2$.

 In \cite{NO}, S.~Nitzan and A.~Olevskii (see also the earlier paper \cite{NO07}) constructed sparse,
uniformly discrete systems of translates in $L^2(\R)$ with the following
approximation property: for every $s>2$, every function
$f\in L^2(\R)$ can be approximated arbitrarily well by finite
linear combinations with $\ell^s$-norms of the coefficients at most
$C(s)\norm{f}_2$. Powell and the first and third
authors~\cite[Corollary~4.3]{FPT} later constructed Schauder frames of translates
of a non-negative function for every $1\le p<\infty$, although their
construction allows the translation parameters to repeat.
Lev and Tselishchev~\cite{LT2,LTp} constructed Schauder frames of
distinct, uniformly separated translates in $L^2(\R)$ and, more
generally, for $p>(1+\sqrt5)/2$. In this range, the generator can
also be chosen to be non-negative~\cite[Theorem~5.1]{LTp}.
Related results on Schauder frames of translates in Orlicz spaces
and modulation spaces were obtained by Dharra and
Sivananthan~\cite{DS} and Yu~\cite{Yu}, respectively.
In the present article, we construct Schauder bases of translates for every
$2\le p<\infty$. For $p>2$, we further show that one can select the parameters from any
prescribed unbounded set.

Historically, developments on systems of translates have run in parallel with developments on non-negative coordinate systems. A non-negative basis $(f_n)_{n=1}^\infty$ allows one to infer that a function $f:=\sum_{n=1}^\infty a_nf_n$ is pointwise a.e.~non-negative whenever the coefficients $(a_n)_{n=1}^\infty$ are non-negative numbers. In applications, it is often the case that only non-negative functions have physical meaning, so the ability to read off the admissibility of a signal directly from its coefficients is desirable.

However, constructing non-negative bases is not easy, as non-negativity significantly alters the role of cancellation. Note that all non-constant Haar
functions have mean zero, and in fact all integrable orthonormal wavelets, as well as the non-constant elements of the standard trigonometric and Walsh bases, have this property. Thus, significant cancellation occurs within the basis functions themselves. On the other hand, for a non-negative basis, all sign changes in the represented function
must come from the coefficients rather than the basis functions themselves. This, in particular, quickly implies that such a basis cannot be unconditional, i.e., any non-negative basis of $L^p(\mathbb{R})$ can be rearranged so that it fails to be a basis \cite{NV,PS}. Further, every Schauder frame of non-negative functions must be conditional \cite{PS}. Thus, both non-negative bases and bases of translates are necessarily conditional.

In 2015, Johnson and Schechtman \cite{JS} constructed a non-negative Schauder basis of $L^1(\R)$.
Powell and the first and third authors \cite{FPT} subsequently constructed a non-negative Schauder basis of $L^2(\R)$. For general $1<p<\infty$,
they also obtained a non-negative \emph{basic sequence} whose closed span contains an isomorphic
copy of $L^p(\R)$. However, the existence of a non-negative basis remained open for $p\ne1,2$.  Indeed, the proof for $L^1(\mathbb{R})$ made crucial use of the fact that non-negative bumps translated to infinity do not converge weakly to zero, whereas the proof for $L^2(\R)$ crucially used that every closed subspace of a Hilbert space is $1$-complemented. Our first main theorem constructs a non-negative Schauder basis for $L^p(\R)$ for every $1\leq p<\infty$.
We remark that the choice of $\mathbb{R}$ with Lebesgue measure is not important, as once a non-negative basis of $L^p(\mathbb{R})$ is constructed, the classification of abstract $L^p$-spaces implies that every separable $L^p(\mu)$ has a non-negative basis.

\subsection{Main results}

In order to state the main results of the article, let us introduce some basic notions from Banach space theory and basis theory. For $1\le p\le\infty$, we write $q$ for the conjugate exponent,
defined by $1/p+1/q=1$. The scalar field $\F$ may be $\R$ or $\mathbb C$.
We use the bilinear pairing
$$
\pair{f}{g}=\int_{\R}fg,\quad f\in L^p(\R),\quad g\in L^q(\R).
$$
Over $\mathbb C$, this duality pairing contains no conjugation, and
adjoints in the non-negative basis construction are taken with respect
to this pairing. For the construction of a basis of translates in $L^2(\R)$, we will
use a Hermitian inner product, as specified in \Cref{h2:sec:weighted}.

For $1\le p<\infty$, a sequence $(f_n)$ is a \emph{Schauder basis}
of $L^p(\R)$ if every
$f\in L^p(\R)$ has a unique norm-convergent expansion
$f=\sum_n a_n(f)f_n$. Here, the coefficient maps $a_n$ are bounded
linear functionals, called the \emph{biorthogonal functionals}.
Since $p<\infty$, they are represented by functions $\gamma_n\in L^q(\R)$
through $a_n(f)=\pair{f}{\gamma_n}$.
We call the $\gamma_n$ the \emph{coefficient functions}.
For a basis $\mathcal B=(f_n)$, we define its \emph{partial sum projections}
and \emph{basis constant} by
\begin{equation}\label{eq:K}
P_nf=\sum_{j=1}^n a_j(f)f_j,\quad
K(\mathcal B)=\sup_n\norm{P_n}.
\end{equation}
Every nonzero projection has norm at least one, so $K(\mathcal B)\ge1$. A basis is called \emph{monotone} if $K(\mathcal B)=1$.

We say that a basis $(f_n)$ is \emph{non-negative} if every $f_n$ is non-negative
almost everywhere. A basis has \emph{positive biorthogonal functionals}
if every $\gamma_n$ is non-negative almost everywhere.
Such bases are related by duality.

%\subsection{Schauder basis of non-negative functions}

Our first main result constructs non-negative bases of $L^p(\mathbb{R})$ with basis constants arbitrarily close to the optimal lower bound of one for every $1\leq p<\infty$.

\begin{theorem}\label{thm:positive}
Let $1\leq p<\infty$ and $\varepsilon>0$. There is a non-negative Schauder
basis $(f_n)$ of $L^p(\R)$ with basis constant $K((f_n))<1+\varepsilon$.
Moreover, no non-negative Schauder basis $(f_n)$ of $L^p(\R)$ has $K((f_n))=1$.
\end{theorem}
For $1<p<\infty$, the coefficient functionals of any Schauder basis of
$L^p(\R)$ form a Schauder basis of $L^q(\R)$ with the same basis
constant.
Applying this to the non-negative basis in \Cref{thm:positive}, with
$q$ in place of $p$, gives \Cref{thm:functionals} as an immediate consequence. When $p=1$, duality alone does not yield a basis of $L^1(\R)$ with positive biorthogonal functionals, but our argument will nevertheless establish this.
\begin{corollary}\label{thm:functionals}\label{prop:endpoint}
Let $1\leq p<\infty$ and $\varepsilon>0$. There is a Schauder basis
$(f_n)$ of $L^p(\R)$ with $K((f_n))<1+\varepsilon$ whose coefficient
functions $\gamma_n$ are non-negative.
Moreover, no basis with positive
biorthogonal functionals has basis constant one.
\end{corollary}

We remark that a basis with positive biorthogonal functionals has the equally useful property that if $f=\sum_n a_nf_n\geq 0$ then $a_n\geq 0$ for all $n$, i.e., the only way a function can be pointwise a.e.~non-negative is if all of the coefficients satisfy $a_n\geq 0$. It is easy to see that -- in a general Banach lattice -- the equivalence $f=\sum_n a_nf_n\geq 0$ if and only if $a_n\geq 0$ for all $n$ implies that the basis vectors $f_n$ are pairwise disjoint. Obviously, $L^p(\mathbb{R})$ does not have a basis with this property.

%\subsection{Schauder bases of translates}

Our next main result concerns bases of translates. For $\lambda\in\R$, we
write $T_\lambda f=f(\,\cdot-\lambda)$ for the translation of $f$ by $\lambda$.

\begin{theorem}\label{thm:translates}
For every $2\le p<\infty$, there is a nonzero real-valued function
$f\in L^p(\R)$ and an ordered sequence of real numbers
$(\lambda_n)$ such that $(T_{\lambda_n}f)$ is a Schauder basis of
$L^p(\R)$ over either scalar field.
% If $p>2$, the parameters $\lambda_n$ can be chosen from any prescribed
% unbounded subset of $\R$. At $p=2$, the basis constant can be taken
% to be at most $13$.
\end{theorem}
Our construction of a basis of translates for $p=2$ is related to the weighted exponential approach of Olevskii~\cite{Ol97} and Lev and Tselishchev~\cite{LT2}. The additional estimates needed to obtain a Schauder basis control the coordinate projections of the individual exponentials throughout the construction, see \Cref{h2:sec:weighted}.

 In the case of non-negative bases, the principle of small perturbations
can be used to convert any non-negative basis into a non-negative basis with
whatever smoothness or decay properties one wishes for. Thus, the only interesting quantity to optimize is the
basis constant. However, in the case of bases of translates, there is much more freedom on what to optimize.
Some interesting ones are listed below.

\begin{theorem} \label{thm:translate-properties}
 The Schauder basis of translates $(T_{\lambda_n} f)$ in Theorem \ref{thm:translates} can be chosen with the following additional regularity and structural properties:
\begin{enumerate}
\item If $p=2$ then for every $\varepsilon>0$ one can choose $f$ which extends to an entire function and $(\lambda_n)$ so that the basis constant satisfies $K((T_{\lambda_n}f)) < 1+\varepsilon$.
\item If $p>2$ and $S\subset\R$ is unbounded, there are distinct parameters $(\lambda_n)\subset S$ and a non-negative function $f\in L^p(\R)$ extending to an entire function such that $(T_{\lambda_n}f)$ is a Schauder basis of $L^p(\R)$.
\end{enumerate}
\end{theorem}

\begin{remark} \label{rem:translation-restrictions}
There are several known restrictions on any basis of translates.
As mentioned above, such a basis must be
conditional. Another interesting fact from \cite{DH,OSSZ} is that the generator cannot
belong to $L^1(\R)$. In particular,
there is no basis of translates of $L^1(\R)$.

There are also several Fourier restrictions. For example, when $p=2$, completeness
requires $\widehat f\ne0$ almost everywhere~\cite[Theorem~1.1]{OSSZ}.
Moreover, for every finite $p$, the Fourier transform of $f$, interpreted as
a distribution when necessary, must have support equal to $\R$.
Interestingly, for $1<p<2$, Lev and Olevskii~\cite[Corollary~2]{LO} constructed two
functions in $L^1(\R)\cap L^p(\R)$ whose Fourier transforms have the same
zero set, but whose families of all real translates have different
completeness properties in $L^p(\R)$. Thus, in this range, the zero set
alone does not determine completeness of the family of all real translates.
\end{remark}

\begin{remark}\label{rem:translation-sets}
For a basis $(T_{\lambda_n}f)$, it is evident that the translation parameters $\lambda_n$ must be \emph{uniformly discrete}, meaning that
$\inf_{n\ne m}|\lambda_n-\lambda_m|>0$.
For $p>2$, \Cref{thm:translates} applies with the parameters $\lambda_n$ chosen from any prescribed
unbounded subset of $\R$. However, even for $p>2$, the whole
lattice $a\mathbb Z+b$ cannot be ordered as a Schauder basis of
translates~\cite[Proposition~3.2]{OSSZ}.

When $p=2$, our construction allows for rational parameters whose pairwise
separations tend to infinity;
see \Cref{h2:ss:transfer}. In particular, the translation set can have zero
upper uniform density, that is,
$$
\limsup_{R\to\infty}\sup_{x\in\R}
\frac{\#\{n:\lambda_n\in[x,x+R]\}}{R}=0.
$$
However, one cannot prescribe an arbitrary unbounded set at the endpoint $p=2$. Indeed,
it is a classical result of Atzmon and Olevskii~\cite{AO} that
no subset of $a\mathbb Z+b$ gives a complete system of
translates in $L^2(\R)$. Nonetheless, small nonzero perturbations of the integers can support
complete systems; see \cite{LT-sparse,Ol97,OU} for additional information.
\end{remark}

\subsection{Notation}

For $s\in L^p(\R)$ and $d\in L^q(\R)$, we define
$s\otimes d\colon L^p(\R)\to L^p(\R)$ by
$$
(s\otimes d)f=\pair{f}{d}s,\quad
\norm{s\otimes d}=\norm{s}_p\norm{d}_q,
$$
so that the first coordinate projection of a basis is
$f_1\otimes\gamma_1$.
We use the same notation in finite-dimensional sequence spaces.
Two families $(s_j)$ and $(d_j)$ are \emph{biorthogonal} if
$\pair{s_i}{d_j}=\delta_{ij}$ and an operator is \emph{positive} if it maps non-negative functions
to non-negative functions.

We write $\1_G$ for the indicator of a measurable set $G$ and $|G|$ for its Lebesgue measure. A \emph{step function} is a finite linear combination of indicator functions of bounded intervals.
A partition of a set into intervals is a collection of pairwise disjoint intervals whose union is that set.
In the non-negative construction, all partitions consist of
half-open intervals of positive length, and every unit interval belongs to
$\{[l,l+1):l\in\mathbb Z\}$.

A \emph{basic sequence} is a Schauder basis of its closed linear span, and its basis constant is the largest norm of its partial sum projections on its span. We use the classical fact that the normalized Haar system is an unconditional basis of $L^p(0,1)$ for $1<p<\infty$~\cite{LTclassic}.

\subsection{Formal verification in Lean and usage of LLMs}

Theorems~\ref{thm:positive}, \ref{thm:translates},
and~\ref{thm:translate-properties} have been formally verified in Lean~4. The corresponding Lean code is available in the accompanying GitHub repository:
\begin{center}
\url{https://github.com/lukasliehr/Structured-Coordinate-Systems}
\end{center}
The file \texttt{Showcase.lean} contains the formal statements of these theorems. The Lean verification has been conducted with Codex (GPT-6 Astra) and Claude Opus 5.5. The authors also used large language models as research
tools in developing the mathematical results presented in this article. This work involved sustained dialogue with GPT-5.4, GPT-5.5, GPT-5.6, and GPT-6 during the development of the results, combining input from these models with the authors’ previous research on these topics and their earlier attempts to prove the theorems presented in this paper.

\section{Outline of the non-negative construction}\label{sec:positive-outline}

We construct the basis inductively by adding finitely many functions at a time,
starting with $f_1=\1_{[0,1)}$. At each stage, the functions
already constructed form a non-negative basis of a finite-dimensional
space of step functions. More precisely, at each stage we have a finite union $A$ of unit
intervals and a partition $\mathcal P$ of $A$ into smaller intervals.
The functions constructed in the previous step span exactly the space of functions
that vanish outside of $A$ and are constant on each interval of
$\mathcal P$ (the constant values can be chosen independently
on the different intervals). We extend this space
by refining the partition and including new unit intervals, while
keeping all previously constructed basis functions unchanged.

There are two requirements for this procedure to give a Schauder basis of $L^p(\R)$. The finite spans must have dense union, and their partial sum projections must remain uniformly bounded. The first requirement is achieved by eventually including every interval between consecutive integers and refining the partitions so that the longest subinterval of each fixed unit interval has length tending to zero. The main work is to control the partial sums while keeping all basis functions non-negative. We explain one stage for $1<p<\infty$ below.

The reason that we were not able to execute the construction in \cite{FPT} for $p\neq 2$ is because of the lack of contractive projections onto closed subspaces of $L^p(\mathbb{R})$. We could not simply pass to the orthogonal complement of the first part of the construction and continue onward. However, now that we are working with entire spaces of step functions at each stage of the construction, we do have a contractive projection onto this space. Indeed, we may simply replace a function by its average on each interval of $\mathcal P$ and set it equal to zero outside of $A$. Thus, at every stage we have a projection of norm one onto the old span, even though the basis functions spanning it may have a
complicated form.

To refine the partition, we split an interval $D_j$ into a short left interval $B_j$
and a right remainder $C_j$, with $|B_j|=\theta|D_j|$. Here,
we choose an integer $L\ge2$ and set $\theta=1/(2L)$.
We then add the function
$$
\varphi_j=\1_{B_j}-\theta\1_{D_j},\quad
D_j=B_j\sqcup C_j.
$$
This function has value $1-\theta$ on $B_j$ and $-\theta$ on $C_j$, and
has integral zero. Together with the old functions, it allows us to
prescribe different constant values on the two pieces -- it plays
the role of a Haar function for this unequal splitting.
In each interval of $\mathcal P$, we make $L$ successive cuts.
The first cut divides the interval into a left piece of relative
length $\theta$ and a right piece of relative length $1-\theta$.
We then split the right piece in the same proportions and repeat
this procedure on each new right-hand piece. All left pieces remain
unchanged. After $L$ cuts, the original interval is divided into
$L+1$ pieces. Since we make $L$ cuts in each of the $R$ intervals
of $\mathcal P$, the total number of cuts is $N=RL$.
The functions $\varphi_1,\ldots,\varphi_N$ therefore provide the
additional functions needed to span the refined step-function space.
Each takes the value $-\theta$ on part of its support, so we must
modify the construction to obtain non-negative basis functions.

We will recover each $\varphi_j$ as a difference of two
non-negative functions. To do this while controlling the partial
sums, we use $2N$ new unit intervals disjoint from $A$, and denote
their union by $F$. In \Cref{sec:finite}, we construct a non-negative
basis $(s_j)_{j=1}^{2N}$ of the functions constant on these intervals
and zero outside of $F$. The norms $\norm{s_j}_p$ are bounded
independently of $N$, and the basis constant is at most $1+\delta$,
where $\delta>0$ can be arbitrarily small. For a parameter
$\kappa>0$, we extend the ordered list of functions already
constructed by adding
\begin{equation}\label{eq:outline-pairs}
s_1,\ \kappa^{-1}s_1+\varphi_1,\quad
s_2,\ \kappa^{-1}s_2+\varphi_2,\quad
\ldots,\quad
s_N,\ \kappa^{-1}s_N+\varphi_N,
\end{equation}
followed by $s_{N+1},\ldots,s_{2N}$. Thus, each $s_j$ is immediately
followed by $\kappa^{-1}s_j+\varphi_j$ for $1\le j\le N$.
Each pair allows us to recover the function $\varphi_j$, since
$$
\varphi_j
=\bigl(\kappa^{-1}s_j+\varphi_j\bigr)-\kappa^{-1}s_j.
$$
We will choose the scaling parameter $\kappa>0$ below to control
the partial sum projections that end between the two functions
of a pair.
The list in \eqref{eq:outline-pairs}, together with the remaining
$s_j$ and the old basis functions, spans all of the functions which are
constant on the refined partition of $A$ and on the new unit intervals
in $F$, and are zero outside of $A\cup F$. The second member of
each pair still takes negative values, which we will remove shortly.

Our next objective is to estimate the partial sum projections for this ordering.
For a function $h$, we define $E_jh$ by replacing $h$ on each
interval of the partition of $A$ after $j$ splittings by its
average on that interval, and setting it equal to zero outside of $A$.
To define $Q_jh$, we first average $h$ on each unit interval in $F$
and set it equal to zero outside of $F$. We then expand this step
function in the basis $(s_l)_{l=1}^{2N}$ and retain its first $j$ terms.

After adding the first $j$ complete pairs in \eqref{eq:outline-pairs},
the partial sum projection is $E_j+Q_j$. The operator $E_j$ uses
only the values of $h$ on $A$ and produces a function supported
in $A$. Similarly, $Q_j$ uses only the values on $F$ and produces
a function supported in $F$. Since $A$ and $F$ are disjoint, we have
$$
\norm{E_j+Q_j}
=\max\{\norm{E_j},\norm{Q_j}\}
\le1+\delta.
$$
Here, $\norm{E_j}=1$ because averaging is contractive, and
$\norm{Q_j}\le1+\delta$ due to the basis constant of $(s_l)$.
A partial sum ending between the members of the $j$-th pair differs
from $E_{j-1}+Q_j$ by the operator
$$
h\longmapsto-\kappa^{-1}s_j
\left(\frac{1}{|B_j|}\int_{B_j}h
 -\frac{1}{|C_j|}\int_{C_j}h\right).
$$
The expression in parentheses is the difference between the two
averages created by the splitting. We arrange the splittings so that
each $D_j$ has at least half the length of the shortest interval
in $\mathcal P$. For $h\in L^p(\R)$, H\"older's inequality gives
$$
\left|\frac{1}{|B_j|}\int_{B_j}h
 -\frac{1}{|C_j|}\int_{C_j}h\right|
\le \left(|B_j|^{-1/p}+|C_j|^{-1/p}\right)\norm{h}_p.
$$
Here, $|B_j|=\theta|D_j|$ and $|C_j|=(1-\theta)|D_j|$.
Using $\theta\le1/2$ and the fact that each $D_j$ has at least half
the length of the shortest interval in $\mathcal P$, we conclude that
$$
\left|\frac{1}{|B_j|}\int_{B_j}h
 -\frac{1}{|C_j|}\int_{C_j}h\right|
\le C\theta^{-1/p}\norm{h}_p,
$$
where $C$ depends only on $p$ and the shortest interval length
in $\mathcal P$, and is independent of $\theta$ and $j$. We may therefore take $\kappa$
proportional to $\theta^{-1/p}$, with a sufficiently large constant,
so that this additional operator has norm at most a prescribed
$\eta>0$. All partial sums ending among the new vectors then have
norm at most $1+\delta+\eta$. The calculation is given in
\Cref{ss:projections}.

We now make the new vectors non-negative. Along with
$s_1,\ldots,s_{2N}$, we construct a non-negative step function $z_N$
that vanishes outside of $F$ and satisfies
\begin{equation}\label{eq:outline-z}
\pair{s_j}{z_N}\ge1\quad(1\le j\le N),\quad
\omega_N:=N^{-1/q}\norm{z_N}_q\longrightarrow0.
\end{equation}
In the construction below, the lower bound on the pairings will ensure that the constant we
add removes the negative values and the limit will ensure that the
map adding this constant is close to the identity.

We define
$C_A\colon L^p(\R)\to L^p(\R)$ by
\begin{equation}\label{eq:outline-map}
C_Ah=h+\kappa\theta\pair{h}{z_N}\1_A.
\end{equation}
This map adds a constant on $A$, determined by the values of $h$
on the new intervals. It fixes every old function because those
functions vanish on $F$. On the second member of a pair, it gives
$$
C_A(\kappa^{-1}s_j+\varphi_j)
=\kappa^{-1}s_j+\varphi_j
 +\theta\pair{s_j}{z_N}\1_A\ge0,
$$
by \eqref{eq:outline-z} and $\varphi_j\ge-\theta\1_A$.
The map $C_A$ also preserves the non-negativity of every $s_j$.
Subtracting $\kappa^{-1}$ times the first transformed vector from
the second cancels the added constant, so
$$
C_A(\kappa^{-1}s_j+\varphi_j)-\kappa^{-1}C_As_j=\varphi_j.
$$
Thus, the span of the transformed vectors still contains each $\varphi_j$.
\Cref{fig:positive-pair} illustrates this cancellation.

We next use the limit in \eqref{eq:outline-z} to estimate
$\norm{C_A-I}$. With $\mathcal P$ fixed, the choice of
$\theta=1/(2L)$ gives $\theta=R/(2N)$, as in
\eqref{eq:batch-parameters}. Since
$\kappa$ is proportional to $\theta^{-1/p}$, we obtain
\begin{equation}\label{eq:outline-map-bound}
\norm{C_A-I}
=\kappa\theta|A|^{1/p}\norm{z_N}_q
\le C\omega_N\longrightarrow0,
\end{equation}
where $C$ is independent of $N$. The map $C_A$ is invertible for
every $N$, as shown in \eqref{eq:local-inverse}.
By \eqref{eq:outline-map-bound}, both $C_A$ and $C_A^{-1}$ have
norms tending to one as $N\to\infty$. The map
in \eqref{eq:outline-map} adds a multiple of $\1_A$, which belongs
to the old span, so it maps the enlarged step-function space onto
itself. Applying $C_A$ to the basis changes a partial sum projection
$T$ to $C_ATC_A^{-1}$. Its norm therefore increases by at most
the factor $\norm{C_A}\norm{C_A^{-1}}$, which tends to one.
This gives the non-negative extension in \Cref{lem:positive-extension}.

We repeat this construction in \Cref{sec:assembly}, choosing
the norms of the added operators $C_A-I$ to have a sufficiently
small sum. Together with
small choices of $\delta$ and $\eta$, this keeps every finite-dimensional basis
constant below $1+\varepsilon$. The supports eventually cover
$\R$, and the longest subinterval of the partition on each fixed
unit interval has length tending to zero. Hence, the finite spans have dense union, and the uniformly
bounded partial sums converge to every function in $L^p$.

This construction follows the scheme of \cite[Section~2]{FPT}. There, orthogonal projection onto the old
span separates the part of the next Haar function that still needs
to be approximated. Here, the span at each stage contains all step functions on a
finite partition, so the averaging operator in \eqref{eq:averaging}
provides a projection of norm one for every $p$. We also express
each $\varphi_j$ exactly as a linear combination of the new
non-negative functions. These choices ensure that the next stage
again starts with a full step-function space.

To obtain a non-negative basis of $L^1$ with basis constant at most $1+\epsilon$, we use the simpler extension
in \Cref{sec:positive-l1}. We fix $0<\theta<1/2$, split
one interval $D$ into $B,C$ with $|B|=\theta|D|$, and add
one new unit interval $J$. Writing $e_V=\1_V/|V|$, we
append $e_C+e_J$ followed by $e_B+e_J$. Their difference recovers
the required Haar function up to a scalar factor. The projection
that discards both new functions is contractive, and the projection
that stops between them has norm $1+2\theta$. Thus, later stages
preserve all earlier partial sum norms, while the same refinement
and density arguments as in the reflexive case give a basis. The fact that a non-negative basis in $L^1(\R)$ cannot be monotone is a consequence of \cite[Theorem~3.1]{DO}.

To obtain a basis with positive biorthogonal functionals in $L^1$, we run the
construction presented in \Cref{sec:finite,sec:batch,sec:assembly} with
$p=\infty$ and prove norm convergence for the coefficient functions
in $L^1$. \Cref{sec:duality} completes the proofs of the main
results, including the lower bounds on the basis constants in the
stated ranges.

\section{A finite-dimensional construction}\label{sec:finite}

We construct the vectors $s_j$ and $z_N$ used in the outline.
The vectors $s_j$ must be non-negative, have uniformly bounded
norms, and form a basis with constant at most $1+\delta$.
The conditions on $z_N$ in \eqref{eq:aux-smallness} below will
allow us to make the functions in \eqref{eq:outline-pairs} non-negative
by applying a map close to the identity.

Throughout \Cref{sec:finite,sec:batch,sec:assembly}, we allow
$1<p\le\infty$, so that $1\le q<\infty$. The case $p=\infty$
will give the bases with positive biorthogonal functionals in
\Cref{prop:endpoint}. We write $(e_j)_{j=1}^m$
for the coordinate basis of $\F^m$, equip $\ell_r^m$ with its usual
norm, and use the bilinear pairing $\pair{x}{y}=\sum_i x_i y_i$.
The finite-dimensional construction can now be stated as follows.

\begin{proposition}[Non-negative bases in finite dimensions]\label{prop:aux}
Fix $1<p\le\infty$ and $\delta>0$. There is a constant $M\ge1$
and, for every $N\ge1$, a pair of biorthogonal bases
$(s_j)_{j=1}^{2N}$ and $(d_j)_{j=1}^{2N}$ of $\F^{2N}$ and a vector
$z_N\in\F^{2N}$ such that
\begin{align}
&s_j\ge0,\quad \norm{s_j}_p\le M,\quad
\sup_{0\le r\le2N}
\left\|\sum_{j=1}^r s_j\otimes d_j\right\|_{\ell_p^{2N}\to\ell_p^{2N}}
\le1+\delta,\label{eq:aux-bounds}\\
&z_N\ge0,\quad
\pair{s_j}{z_N}\ge1\quad(1\le j\le N),\quad
\omega_N:=N^{-1/q}\norm{z_N}_q\longrightarrow0
\quad(N\to\infty).\label{eq:aux-smallness}
\end{align}
The constant $M$ is independent of $N$.
\end{proposition}

Here, the vectors $s_j$ and $d_j$ depend on $N$, although we omit this dependence from the notation. Note that the coordinate basis cannot satisfy \eqref{eq:aux-smallness}. Indeed, if $s_j=e_j$, then the first $N$ coordinates of $z_N$ must be at least one, giving $\norm{z_N}_q\ge N^{1/q}$. We therefore need the supports of the $s_j$ to overlap, so that a single coordinate of $z_N$ can contribute
to several pairings.

For $1<p<\infty$ we choose
$0<\alpha\le1$ with $2\alpha(p+q)\le\delta$ and define
\begin{equation}\label{eq:toeplitz}
T_m=(t_{ij})_{i,j=1}^m,\quad
t_{ij}=
\begin{cases}
0,&i<j,\\
1,&i=j,\\
\alpha/(i-j+1),&i>j.
\end{cases}
\end{equation}
This lower triangular matrix is invertible. For $m=2N$, we set
\begin{equation}\label{eq:toeplitz-systems}
s_j=T_m e_j,\quad d_j=T_m^{-\mathsf T}e_j,\quad
\Lambda_N=1+\alpha\sum_{r=2}^{N+1}\frac{1}{r},\quad
z_N=\frac{1}{\Lambda_N}\sum_{l=1}^{2N}e_l.
\end{equation}
Here, $T_m^{-\mathsf T}$ denotes the transpose of $T_m^{-1}$.
The columns $s_j$ have uniformly bounded $\ell_p$ norms, whereas
each of the first $N$ columns has coordinate sum at least
$\Lambda_N$. Since $\Lambda_N\to\infty$, this choice of $z_N$
gives the required pairings and makes $\omega_N$ tend to zero.
The remaining point is to control the partial sum projections.
For this, we first bound $T_m^{-1}$. We then bound the submatrix
formed by the last $m-r$ rows and the first $r$ columns of $T_m$
for $1\le r<m$.

\begin{lemma}\label{lem:inverse}
For $0<\alpha\le1$, the matrix in \eqref{eq:toeplitz} satisfies
$$
\norm{T_m^{-1}}_{\ell_r^m\to\ell_r^m}\le2
\quad(1\le r\le\infty).
$$
\end{lemma}

\begin{proof}
Let $V_m$ be the coordinate shift given by
$V_m e_j=e_{j+1}$ for $j<m$ and $V_m e_m=0$.
We set $\mu_0=1$ and $\mu_n=\alpha/(n+1)$ for $n\ge1$, so that
$$
T_m=\sum_{k=0}^{m-1}\mu_kV_m^k.
$$
To compute the inverse, we define $c_n$ recursively by
\begin{equation}\label{eq:renewal}
\mu_n=\sum_{k=1}^n c_k \mu_{n-k}\quad(n\ge1).
\end{equation}
Since $V_m^m=0$, this identity gives
$$
T_m^{-1}=I-\sum_{k=1}^{m-1}c_kV_m^k.
$$
Indeed, \eqref{eq:renewal} makes the coefficient of every positive
power of $V_m$ vanish in the product. Since each power of $V_m$
is contractive on $\ell_r^m$, it remains to prove that $c_n\ge0$
and $\sum_{k=1}^n c_k\le1$.

We use the elementary argument behind Kaluza's sign criterion
\cite[Theorem~1.3]{BVV}. The ratios $\mu_n/\mu_{n-1}$ are
$\alpha/2,2/3,3/4,\ldots$ and are therefore nondecreasing.
Since $c_1=\mu_1>0$, we can prove $c_n\ge0$ by induction.
If $c_1,\ldots,c_{n-1}\ge0$, then
$$
\sum_{k=1}^{n-1}c_k \mu_{n-k}
\le\frac{\mu_n}{\mu_{n-1}}
 \sum_{k=1}^{n-1}c_k \mu_{n-1-k}
=\mu_n.
$$
It follows from \eqref{eq:renewal} that $c_n\ge0$. Moreover,
since $(\mu_n)$ is decreasing, we have
$$
\mu_n=\sum_{k=1}^n c_k \mu_{n-k}
\ge \mu_n\sum_{k=1}^n c_k,
$$
which proves the bound on the sum. Consequently,
$$
\norm{T_m^{-1}}_{\ell_r^m\to\ell_r^m}
\le1+\sum_{k=1}^{m-1}c_k\le2.
$$
\end{proof}

The submatrix formed by the last $m-r$ rows and the first $r$
columns of $T_m$ has entries $\alpha/(i+j)$ after its columns
are reversed. The following estimate will bound its norm.

\begin{lemma}\label{lem:hilbert}
Let $1<r<\infty$ and $r'=r/(r-1)$.
The matrix $\mathcal H=(1/(i+j))_{i,j\ge1}$ defines a bounded
operator on $\ell_r$, with norm at most $r+r'$.
\end{lemma}

\begin{proof}
For $0<u<1$, comparison with the integral of a decreasing function gives
$$
\sum_{j\ge1}\frac{j^{-u}}{i+j}
\le\int_0^\infty\frac{t^{-u}}{i+t}\,dt
\le\left(\frac{1}{1-u}+\frac{1}{u}\right)i^{-u}.
$$
In the last step, we split the integral at $i$, and then bounded the
denominator below by $i$ on $(0,i)$ and by $t$ on $(i,\infty)$.
Taking either $u=1/r$ or $u=1/r'$ gives the same constant
$C=r+r'$.

For a finitely supported sequence $x$, H\"older's inequalty
and the estimate with $u=1/r$ yield
$$
\left|\sum_j\frac{x_j}{i+j}\right|^r
\le
\left(\sum_j\frac{j^{-1/r}}{i+j}\right)^{r-1}
\sum_j\frac{j^{1/r'}|x_j|^r}{i+j}
\le C^{r-1}i^{-1/r'}
\sum_j\frac{j^{1/r'}|x_j|^r}{i+j}.
$$
Summing in $i$, interchanging the sums, and applying the estimate with
$u=1/r'$ gives
$$
\sum_i\left|\sum_j\frac{x_j}{i+j}\right|^r
\le C^r\sum_j|x_j|^r.
$$
The result then follows by density.
\end{proof}

\begin{proof}[Proof of \Cref{prop:aux} for $1<p<\infty$]
We use the vectors in \eqref{eq:toeplitz-systems}. They are
biorthogonal because $T_m^{-1}T_m=I$, and the columns $s_j$ are
non-negative. Their norms satisfy
$$
\norm{s_j}_p^p
=1+\alpha^p\sum_{r=2}^{m-j+1}r^{-p}
\le1+\int_1^\infty t^{-p}\,dt
=q.
$$
Thus, we may take $M=q^{1/p}$. Since each of the first $N$ columns
has coordinate sum at least $\Lambda_N$, we also have
$$
\pair{s_j}{z_N}\ge1,\quad
\omega_N=\frac{2^{1/q}}{\Lambda_N}\longrightarrow0.
$$

It remains to estimate the partial sum projections. If $D_r$
denotes the projection onto the first $r$ coordinates, then the
corresponding projection for $(s_j)$ is $T_mD_rT_m^{-1}$.
For $1\le r<m$, we write
$$
T_m=
\begin{pmatrix}
T_r&0\\
B_r&T_{m-r}
\end{pmatrix}.
$$
Here, $B_r$ consists of the last $m-r$ rows and the first $r$
columns of $T_m$. Block multiplication gives
$$
T_mD_rT_m^{-1}
=
\begin{pmatrix}
I_r&0\\
B_rT_r^{-1}&0
\end{pmatrix}.
$$
Thus, we only need to bound $B_rT_r^{-1}$. Reversing the columns
of $B_r$ gives the entries $\alpha/(u+v)$, where
$1\le u\le m-r$ and $1\le v\le r$. Coordinate reversal
preserves the norm, so \Cref{lem:hilbert} gives
$\norm{B_r}_{\ell_p^r\to\ell_p^{m-r}}\le\alpha(p+q)$.
Combining this with \Cref{lem:inverse}, we obtain
$$
\norm{T_mD_rT_m^{-1}}_{\ell_p^m\to\ell_p^m}
\le1+\norm{B_r}\norm{T_r^{-1}}
\le1+2\alpha(p+q)\le1+\delta.
$$
The cases $r=0,m$ give the zero operator and the identity.
Since $T_mD_rT_m^{-1}=\sum_{j=1}^r s_j\otimes d_j$, this proves
\eqref{eq:aux-bounds}. The transposed projections on $\ell_q^m$
have the same norms.
\end{proof}

\begin{proof}[Proof of \Cref{prop:aux} for $p=\infty$]
Here, $q=1$. We set $m=2N$ and take
\begin{equation}\label{eq:tail-system}
s_j=\sum_{l=j}^m e_l,\quad
d_1=e_1,\quad d_j=e_j-e_{j-1}\quad(2\le j\le m),\quad z_N=e_m.
\end{equation}
The families are biorthogonal, and each $s_j$ is non-negative with
$\norm{s_j}_\infty=1$. For $1\le r\le m$, telescoping gives
$$
\sum_{j=1}^r\pair{x}{d_j}s_j
=\sum_{l<r}x_l e_l+x_r\sum_{l=r}^m e_l.
$$
This projection keeps the first $r$ coordinates and replaces the
remaining ones by $x_r$, so it is contractive on $\ell_\infty^m$.
Since $\pair{s_j}{z_N}=1$ for every $j$ and $\omega_N=1/N$,
the proposition follows with $M=1$.
\end{proof}

\section{One stage of the construction}\label{sec:batch}

We now use the vectors from \Cref{prop:aux} to extend a non-negative
basis of a given step-function space. The extension must span the whole
space on the refined partition and increase the basis constant by
an arbitrarily small amount. We first construct the new vectors
and estimate their partial sum projections. After this, we will remove their
negative values without changing the old vectors.

Throughout this section, we fix $1<p\le\infty$ and
$\delta,\eta>0$, and let $M$ be as in \Cref{prop:aux}. Here,
$\eta$ bounds the additional term in a partial sum ending between
the two vectors of a pair. We let $A$ be a nonempty finite union
of unit intervals and $\mathcal P$ a finite partition of $A$ into
intervals, each contained in one unit interval. We set
$$
R=\#\mathcal P,\quad \ell=\min_{D\in\mathcal P}|D|.
$$
Thus, $R$ is the number of intervals in $\mathcal P$ and $\ell$ is their minimum
length. We keep $A,\mathcal P,\delta,\eta$ fixed and choose an
integer $L\ge2$ large enough so that the operator
$S_A$ defined in \eqref{eq:local-S} has arbitrarily small norm. We then set
\begin{equation}\label{eq:batch-parameters}
\theta=\frac{1}{2L},\quad N=RL,\quad m=2N.
\end{equation}
Our strategy will be to split each interval of $\mathcal P$ $L$ times and represent
the vectors from \Cref{prop:aux} as step functions on $m$ new
unit intervals. Thus, $N$ is the total
number of splittings.

For any finite interval partition $\mathcal Q$, we write
\begin{equation}\label{eq:averaging}
E_{\mathcal Q}f
=\sum_{D\in\mathcal Q}
\left(\frac{1}{|D|}\int_D f\right)\1_D.
\end{equation}
This operator replaces $f$ by its average on each interval of $\mathcal Q$ and sets
it equal to zero outside their union. It is therefore a projection
onto the associated step-function space. For $1\le r<\infty$,
H\"older's inequality on each interval gives
$\norm{E_{\mathcal Q}f}_r\le\norm{f}_r$, and for $r=\infty$
the same bound follows directly from the definition of an average.
In particular, these projections are contractive for every $r$.

\subsection{Refining the partition}\label{ss:splits}

For each interval of $\mathcal P$, we split off its left proportion
$\theta$ and repeat the operation on the right remainder, making
$L$ splittings in total. We finish one interval of $\mathcal P$ before moving to
the next, and enumerate the resulting splittings as
$$
D_j=B_j\sqcup C_j,\quad |B_j|=\theta|D_j|,
\quad |C_j|=(1-\theta)|D_j|\quad(1\le j\le N),
$$
where $B_j$ is the left subinterval. Note that every interval
being split has length at least $\ell/2$, since
$$
(1-\theta)^r\ge1-r\theta\ge\frac{1}{2}
\quad(0\le r\le L).
$$
On the other hand, every resulting interval has at most $2/3$ of the
length of its original interval. Indeed, each piece split off has
relative length at most $\theta\le1/4$, while the final remainder
satisfies
$$
(1-\theta)^L\le(1+\theta)^{-L}
\le(1+L\theta)^{-1}=\frac{2}{3}.
$$
Thus, repeated refinement makes the length of the longest
subinterval of each fixed unit interval tend to zero.

Each splitting adds one dimension to the step-function space,
since functions may now take different values on $B_j$ and $C_j$.
To obtain the enlarged space, we add the function $\varphi_j$
defined below. Its coefficient is obtained by pairing with $H_j$,
which measures the difference between the averages on the two pieces.
We set
\begin{equation}\label{eq:haar}
\varphi_j=\1_{B_j}-\theta\1_{D_j},
\quad
H_j=\frac{\1_{B_j}}{|B_j|}-\frac{\1_{C_j}}{|C_j|}.
\end{equation}
Both functions have integral zero. Since $\varphi_j$ equals
$1-\theta$ on $B_j$ and $-\theta$ on $C_j$, we have
\begin{equation}\label{eq:haar-identities}
\pair{\varphi_j}{H_j}=1,\quad
\varphi_j\ge-\theta\1_{D_j},\quad
H_j=\frac{\varphi_j}{\theta(1-\theta)|D_j|}.
\end{equation}
The functions in \eqref{eq:haar} play the role of Haar functions
for this unequal splitting. The small negative value of $\varphi_j$
is the reason for choosing unequal pieces. Adding $\theta\1_{D_j}$
is enough to make $\varphi_j$ non-negative. At the same time, making $L$ splittings divides each interval of
$\mathcal P$ into pieces at most $2/3$ of its length, even as $\theta\to0$. The construction
is illustrated in \Cref{fig:unequal-splitting}.

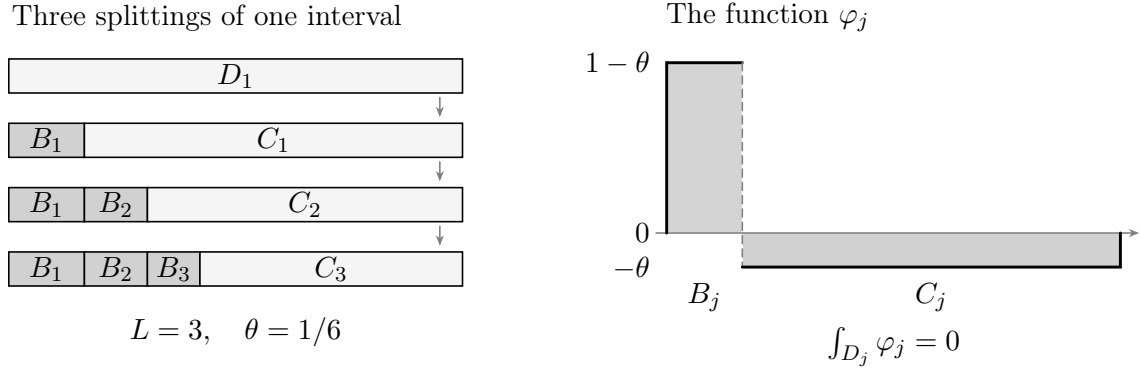
\begin{figure}[htbp]
\centering
\begin{tikzpicture}[x=1cm,y=1cm,font=\small,line cap=round]
\node[anchor=west,text=black,inner sep=1pt] at (0.00000,3.80000) {Three splittings of one interval};
\node[anchor=west,text=black,inner sep=1pt] at (8.65000,3.80000) {The function $\varphi_j$};
\path[draw=pbblue,fill=pbblue!4,line width=0.6pt] (0.00000,2.80000) rectangle (6.00000,3.25000);
\node[anchor=center,text=black,inner sep=1pt] at (3.00000,3.02500) {$D_1$};
\draw[draw=gray,line width=0.6pt,-{Stealth[length=4pt,width=3pt]}] (5.70000,2.74000) -- (5.70000,2.48000);
\path[draw=pbblue,fill=pbblue!18,line width=0.6pt] (0.00000,1.95000) rectangle (1.00000,2.40000);
\node[anchor=center,text=black,inner sep=1pt] at (0.50000,2.17500) {$B_1$};
\path[draw=pbblue,fill=pbblue!4,line width=0.6pt] (1.00000,1.95000) rectangle (6.00000,2.40000);
\node[anchor=center,text=black,inner sep=1pt] at (3.50000,2.17500) {$C_1$};
\draw[draw=gray,line width=0.6pt,-{Stealth[length=4pt,width=3pt]}] (5.70000,1.89000) -- (5.70000,1.63000);
\path[draw=pbblue,fill=pbblue!18,line width=0.6pt] (0.00000,1.10000) rectangle (1.00000,1.55000);
\node[anchor=center,text=black,inner sep=1pt] at (0.50000,1.32500) {$B_1$};
\path[draw=pbblue,fill=pbblue!18,line width=0.6pt] (1.00000,1.10000) rectangle (1.83333,1.55000);
\node[anchor=center,text=black,inner sep=1pt] at (1.41667,1.32500) {$B_2$};
\path[draw=pbblue,fill=pbblue!4,line width=0.6pt] (1.83333,1.10000) rectangle (6.00000,1.55000);
\node[anchor=center,text=black,inner sep=1pt] at (3.91667,1.32500) {$C_2$};
\draw[draw=gray,line width=0.6pt,-{Stealth[length=4pt,width=3pt]}] (5.70000,1.04000) -- (5.70000,0.78000);
\path[draw=pbblue,fill=pbblue!18,line width=0.6pt] (0.00000,0.25000) rectangle (1.00000,0.70000);
\node[anchor=center,text=black,inner sep=1pt] at (0.50000,0.47500) {$B_1$};
\path[draw=pbblue,fill=pbblue!18,line width=0.6pt] (1.00000,0.25000) rectangle (1.83333,0.70000);
\node[anchor=center,text=black,inner sep=1pt] at (1.41667,0.47500) {$B_2$};
\path[draw=pbblue,fill=pbblue!18,line width=0.6pt] (1.83333,0.25000) rectangle (2.52778,0.70000);
\node[anchor=center,text=black,inner sep=1pt] at (2.18056,0.47500) {$B_3$};
\path[draw=pbblue,fill=pbblue!4,line width=0.6pt] (2.52778,0.25000) rectangle (6.00000,0.70000);
\node[anchor=center,text=black,inner sep=1pt] at (4.26389,0.47500) {$C_3$};
\node[anchor=center,text=black,inner sep=1pt] at (3.00000,-0.35000) {$L=3,\quad\theta=1/6$};
\path[draw=white,fill=pbblue!17,line width=0pt] (8.70000,0.95000) rectangle (9.70000,3.20000);
\path[draw=white,fill=pbrust!18,line width=0pt] (9.70000,0.50000) rectangle (14.70000,0.95000);
\draw[draw=gray,line width=0.6pt,-{Stealth[length=4pt,width=3pt]}] (8.56000,0.95000) -- (14.95000,0.95000);
\draw[draw=pbblue,line width=1.1pt] (8.70000,0.95000) -- (8.70000,3.20000) -- (9.70000,3.20000);
\draw[draw=gray,line width=0.6pt,densely dashed] (9.70000,3.20000) -- (9.70000,0.50000);
\draw[draw=pbrust,line width=1.1pt] (9.70000,0.50000) -- (14.70000,0.50000) -- (14.70000,0.95000);
\node[anchor=east,text=black,inner sep=1pt] at (8.52000,3.20000) {$1-\theta$};
\node[anchor=east,text=black,inner sep=1pt] at (8.52000,0.95000) {$0$};
\node[anchor=east,text=black,inner sep=1pt] at (8.52000,0.48000) {$-\theta$};
\node[anchor=center,text=black,inner sep=1pt] at (9.20000,0.10000) {$B_j$};
\node[anchor=center,text=black,inner sep=1pt] at (12.20000,0.10000) {$C_j$};
\node[anchor=center,text=black,inner sep=1pt] at (11.70000,-0.55000) {$\int_{D_j}\varphi_j=0$};
\end{tikzpicture}
\caption{Unequal splittings. On the left, each step removes the left
proportion $\theta$ of the remaining interval. The example uses $L=3$
and $\theta=1/6$. On the right, the positive and negative areas of
$\varphi_j$ are equal. Its negative value is only $-\theta$, so adding
$\theta\1_{D_j}$ makes it non-negative.}
\label{fig:unequal-splitting}
\end{figure}

We let $E_j$ be the operator in \eqref{eq:averaging} for the
partition obtained after the first $j$ splittings, with
$E_0=E_{\mathcal P}$. Comparing the averages on $B_j$ and $C_j$
with the old average on $D_j$ gives
\begin{equation}\label{eq:increment}
E_j-E_{j-1}=\varphi_j\otimes H_j.
\end{equation}
For example, on $B_j$ the right-hand side applied to $f$ is
$$
(1-\theta)\pair{f}{H_j}
=\frac{1}{|B_j|}\int_{B_j}f-\frac{1}{|D_j|}\int_{D_j}f,
$$
and the calculation on $C_j$ is the same.

The families $(\varphi_j)$ and $(H_j)$ are biorthogonal. Indeed,
two different splittings either have disjoint supports, or the later
one takes place on an interval where both earlier functions are
constant. In the latter case, the zero integrals of the later
functions make the pairings vanish in both orders. Each of the $H_j$ and
$\varphi_j$ is supported in a single interval of $\mathcal P$ and has
integral zero there. Consequently, it has integral zero on every interval of
$\mathcal P$ and on every unit interval.

To estimate the partial sums ending between the two members of a
pair, we need a bound on the norm of $H_j$. Formula \eqref{eq:haar}
gives
\begin{align}
\norm{H_j}_q
&=|D_j|^{-1/p}
 \left(\theta^{1-q}+(1-\theta)^{1-q}\right)^{1/q}\notag\\
&\le2\theta^{-1/p}\ell^{-1/p}.
\label{eq:haar-bound}
\end{align}
For the last inequality, we used $|D_j|\ge\ell/2$ and
$\theta\le1-\theta$. At $q=1$, we have $\norm{H_j}_1=2$.

\subsection{The new vectors and their coefficient functions}\label{ss:pairing}

We choose distinct unit intervals $J_1,\ldots,J_m$ outside of $A$
and write $F=\bigcup_{i=1}^m J_i$. We identify $\F^m$ with the
functions constant on these intervals and zero outside of $F$ by the map
$$
(x_1,\ldots,x_m)\longmapsto\sum_{i=1}^m x_i\1_{J_i}.
$$
Since $|J_i|=1$, this map preserves the $\ell_r$ and $L^r$ norms
for $1\le r\le\infty$, as well as the bilinear pairings.
We use the same symbols $s_j,d_j,z_N$ for the vectors of
\Cref{prop:aux} and their corresponding step functions.
These functions are supported in $F$, whereas $H_j$ and
$\varphi_j$ are supported in $A$.

We combine each $\varphi_j$ with a small positive multiple of $s_j$. To preserve biorthogonality, we subtraact the same multiple of $H_j$ from $d_j$. The size of this multiple is determined by
the bound on the partial sum ending between the two new vectors.
We choose
\begin{equation}\label{eq:a}
\kappa=\frac{2M}{\eta}\theta^{-1/p}\ell^{-1/p}.
\end{equation}
By \eqref{eq:haar-bound}, this gives
$\kappa^{-1}\norm{s_j}_p\norm{H_j}_q\le\eta$.
For each $j=1,\ldots,N$, we add the two vectors
\begin{equation}\label{eq:pairs}
v_{j,0}=s_j,\quad v_{j,1}=\kappa^{-1}s_j+\varphi_j,
\end{equation}
placing $v_{j,0}$ immediately before $v_{j,1}$. After all $N$
pairs, we add $v_{N+1},\ldots,v_m$, where $v_j=s_j$ for $j>N$.
The corresponding coefficient functions are
\begin{equation}\label{eq:coefficients}
\xi_{j,0}=d_j-\kappa^{-1}H_j,\quad
\xi_{j,1}=H_j,\quad
\xi_j=d_j\quad(j>N).
\end{equation}
The terms added to $v_{j,1}$ and subtracted from $\xi_{j,0}$
cancel in their pairing, since
$$
\pair{\kappa^{-1}s_j+\varphi_j}{d_j-\kappa^{-1}H_j}
=\kappa^{-1}-\kappa^{-1}=0.
$$
The diagonal pairings equal one, while pairings with different
indices vanish by the biorthogonality of $(s_j,d_j)$ and
$(\varphi_j,H_j)$ and the disjointness of $A$ and $F$. Thus, the
new vectors and their coefficient functions are biorthogonal.

Each function in \eqref{eq:pairs} and \eqref{eq:coefficients},
and each remaining $s_j$ or $d_j$, has integral zero on every
interval of $\mathcal P$. Thus, $E_0$ sends every new vector to
zero, and every new coefficient function pairs to zero with all
functions constant on the intervals of $\mathcal P$ and zero
outside of $A$. It follows that each new coordinate projection has zero
product with $E_0$ in either order. We can therefore add any initial
sum of these coordinate projections to $E_0$ and still obtain a
projection.

\subsection{The projection bounds}\label{ss:projections}

We set
$$
Q_j=\sum_{l=1}^j s_l\otimes d_l,\quad Q_0=0,
$$
where $s_l,d_l$ are the step functions from \Cref{ss:pairing}.
The operator $Q_j$ first averages $f$ on the unit intervals $J_i$
and then applies the finite partial sum projection to the vector
of averages. Since averaging is contractive, \eqref{eq:aux-bounds}
gives $\norm{Q_j}_{L^p\to L^p}\le1+\delta$. For $j=m$, the
finite projection is the identity, and hence
$$
Q_m f=\sum_{i=1}^m\left(\int_{J_i}f\right)\1_{J_i}.
$$

For each pair in \eqref{eq:pairs}, the sum of its two coordinate
projections simplifies to
\begin{equation}\label{eq:cancellation}
s_j\otimes(d_j-\kappa^{-1}H_j)
+(\kappa^{-1}s_j+\varphi_j)\otimes H_j
=s_j\otimes d_j+\varphi_j\otimes H_j.
\end{equation}
Summing \eqref{eq:cancellation} and using \eqref{eq:increment}
shows that, after $j$ complete pairs, the projection is $E_j+Q_j$.
Since $E_i$ and $Q_j$ act on the disjoint sets $A$ and $F$,
respectively, we have
$$
\norm{E_i+Q_j}\le\max\{\norm{E_i},\norm{Q_j}\}\le1+\delta
\quad(0\le i\le N,\ 0\le j\le m).
$$
After only the first vector of the $j$-th pair, the projection is
\begin{equation}\label{eq:intermediate}
E_{j-1}+Q_j-\kappa^{-1}s_j\otimes H_j.
\end{equation}
By \eqref{eq:a} and \eqref{eq:haar-bound}, the last term satisfies
$$
\norm{\kappa^{-1}s_j\otimes H_j}_{L^p\to L^p}
=\kappa^{-1}\norm{s_j}_p\norm{H_j}_q\le\eta.
$$
After all pairs have been inserted, the remaining vectors give
the projections $E_N+Q_j$ for $j>N$. Thus, every partial sum
projection during this stage has norm at most
\begin{equation}\label{eq:raw-constant}
K_0=1+\delta+\eta.
\end{equation}
The last projection is $E_N+Q_m$, which averages on the refined
partition of $A$ and on the new unit intervals. Its norm is one.
Thus, if we repeat this construction with the vectors in
\eqref{eq:pairs}, every stage again starts from an averaging
projection, and the bound $K_0$ in \eqref{eq:raw-constant} remains
the same at all stages.

\subsection{Making the new vectors non-negative}\label{ss:correction}

Only the functions $\kappa^{-1}s_j+\varphi_j$ can take negative
values, and each is bounded below by $-\theta\1_A$. We will make
them non-negative by adding a constant of at least $\theta$ on
$A$. To do this with a single linear map, we define
$S_A\colon L^q(\R)\to L^q(\R)$ by
\begin{equation}\label{eq:local-S}
S_Ag=\kappa\theta\left(\int_A g\right)z_N.
\end{equation}
Since $(L^q)^*=L^p$, its adjoint acts on $L^p(\R)$ by
\begin{equation}\label{eq:local-adjoint}
S_A^*f=\kappa\theta\pair{f}{z_N}\1_A.
\end{equation}
We will apply $I+S_A^*$ to the vectors in \eqref{eq:pairs}.
Formula \eqref{eq:local-adjoint} adds a constant on $A$, determined
by the values of the input on $F$. We work with the adjoint of an
operator on $L^q$ so that the same construction also applies when
$p=\infty$ and $q=1$. Note that both $S_A$ and $S_A^*$ are positive.
We set $\beta_j=\pair{s_j}{z_N}$,
so that $\beta_j\ge1$ for $j\le N$ and $\beta_j\ge0$ for every
$j$. Since $z_N$ is supported outside of $A$, we have
$S_A^*\varphi_j=0$, and therefore
\begin{equation}\label{eq:local-positivity}
(I+S_A^*)v_{j,1}
=\kappa^{-1}s_j+\varphi_j+\theta\beta_j\1_A
\ge\varphi_j+\theta\1_A\ge0.
\end{equation}
Every other new vector is an $s_j$ and becomes
$s_j+\kappa\theta\beta_j\1_A\ge0$. Thus, the map
$C_A:=I+S_A^*$ makes all new vectors non-negative. It also fixes
every function supported in $A$, so the old basis remains unchanged.
Moreover, $(S_A^*)^2=0$, since $S_A^*$ takes values in functions
supported in $A$ and vanishes on such functions. Hence,
\begin{equation}\label{eq:local-inverse}
C_A^{-1}=I-S_A^*,\quad
\norm{C_A}\norm{C_A^{-1}}\le(1+\norm{S_A})^2.
\end{equation}
Applying $C_A$ to the two vectors in \eqref{eq:pairs} gives
\begin{equation}\label{eq:corrected-pairs}
\widehat v_{j,0}=s_j+\kappa\theta\beta_j\1_A,\quad
\widehat v_{j,1}=\kappa^{-1}s_j+\varphi_j+\theta\beta_j\1_A.
\end{equation}
Subtracting $\kappa^{-1}$ times the first vector in
\eqref{eq:corrected-pairs} from the second cancels both the
$s_j$ term and the added constant, giving
$$
\widehat v_{j,1}-\kappa^{-1}\widehat v_{j,0}=\varphi_j.
$$
Thus, each new Haar function lies n the span of two of the
non-negative vectors, as shown in \Cref{fig:positive-pair}.

\begin{figure}[htbp]
\centering
\begin{tikzpicture}[x=1cm,y=1cm,font=\small,line cap=round]
\node[anchor=west,text=black,inner sep=1pt] at (0.00000,7.32000) {$\kappa^{-1}\widehat v_{j,0}$};
\draw[draw=gray,line width=0.6pt,-{Stealth[length=4pt,width=3pt]}] (-0.08000,5.40000) -- (6.95000,5.40000);
\draw[draw=gray,line width=0.6pt,-{Stealth[length=4pt,width=3pt]}] (8.15000,5.40000) -- (14.85000,5.40000);
\node[anchor=east,text=black,inner sep=1pt] at (-0.05000,5.20000) {$0$};
\path[draw=white,fill=pbblue!18,line width=0pt] (0.00000,5.40000) rectangle (6.30000,5.65000);
\draw[draw=pbblue,line width=1.1pt] (0.00000,5.40000) -- (0.00000,5.65000) -- (6.30000,5.65000) -- (6.30000,5.40000);
\node[anchor=west,text=black,inner sep=1pt] at (6.44000,5.65000) {$\theta$};
\path[draw=pbteal,fill=pbteal!16,line width=0.6pt] (10.40000,5.40000) rectangle (11.40000,6.48000);
\path[draw=pbteal,fill=pbteal!16,line width=0.6pt] (11.40000,5.40000) rectangle (12.40000,5.67000);
\path[draw=pbteal,fill=pbteal!16,line width=0.6pt] (12.40000,5.40000) rectangle (13.40000,5.58000);
\path[draw=pbteal,fill=pbteal!16,line width=0.6pt] (13.40000,5.40000) rectangle (14.40000,5.53500);
\node[anchor=center,text=black,inner sep=1pt] at (11.40000,6.88000) {$\kappa^{-1}s_j$};
\node[anchor=center,text=black,inner sep=1pt] at (3.15000,5.00000) {$A$};
\node[anchor=center,text=black,inner sep=1pt] at (11.40000,5.00000) {$F$};
\node[anchor=west,text=black,inner sep=1pt] at (0.00000,4.52000) {$\widehat v_{j,1}$};
\draw[draw=gray,line width=0.6pt,-{Stealth[length=4pt,width=3pt]}] (-0.08000,2.60000) -- (6.95000,2.60000);
\draw[draw=gray,line width=0.6pt,-{Stealth[length=4pt,width=3pt]}] (8.15000,2.60000) -- (14.85000,2.60000);
\node[anchor=east,text=black,inner sep=1pt] at (-0.05000,2.40000) {$0$};
\path[draw=white,fill=pbblue!18,line width=0pt] (0.00000,2.60000) rectangle (1.40000,2.85000);
\path[draw=white,fill=pbblue!18,line width=0pt] (1.40000,2.60000) rectangle (2.10000,4.10000);
\path[draw=white,fill=pbblue!18,line width=0pt] (5.60000,2.60000) rectangle (6.30000,2.85000);
\draw[draw=pbblue,line width=1.1pt] (0.00000,2.60000) -- (0.00000,2.85000) -- (1.40000,2.85000) -- (1.40000,4.10000) -- (2.10000,4.10000) -- (2.10000,2.60000) -- (5.60000,2.60000) -- (5.60000,2.85000) -- (6.30000,2.85000) -- (6.30000,2.60000);
\node[anchor=east,text=black,inner sep=1pt] at (1.30000,4.10000) {$1$};
\path[draw=pbteal,fill=pbteal!16,line width=0.6pt] (10.40000,2.60000) rectangle (11.40000,3.68000);
\path[draw=pbteal,fill=pbteal!16,line width=0.6pt] (11.40000,2.60000) rectangle (12.40000,2.87000);
\path[draw=pbteal,fill=pbteal!16,line width=0.6pt] (12.40000,2.60000) rectangle (13.40000,2.78000);
\path[draw=pbteal,fill=pbteal!16,line width=0.6pt] (13.40000,2.60000) rectangle (14.40000,2.73500);
\node[anchor=center,text=black,inner sep=1pt] at (11.40000,4.08000) {$\kappa^{-1}s_j$};
\node[anchor=center,text=black,inner sep=1pt] at (3.15000,2.20000) {$A$};
\node[anchor=center,text=black,inner sep=1pt] at (11.40000,2.20000) {$F$};
\node[anchor=west,text=black,inner sep=1pt] at (0.00000,1.72000) {$\widehat v_{j,1}-\kappa^{-1}\widehat v_{j,0}=\varphi_j$};
\draw[draw=gray,line width=0.6pt,-{Stealth[length=4pt,width=3pt]}] (-0.08000,-0.20000) -- (6.95000,-0.20000);
\draw[draw=gray,line width=0.6pt,-{Stealth[length=4pt,width=3pt]}] (8.15000,-0.20000) -- (14.85000,-0.20000);
\node[anchor=east,text=black,inner sep=1pt] at (-0.05000,-0.40000) {$0$};
\path[draw=white,fill=pbblue!18,line width=0pt] (1.40000,-0.20000) rectangle (2.10000,1.05000);
\path[draw=white,fill=pbrust!18,line width=0pt] (2.10000,-0.45000) rectangle (5.60000,-0.20000);
\draw[draw=pbblue,line width=1.1pt] (0.00000,-0.20000) -- (1.40000,-0.20000) -- (1.40000,1.05000) -- (2.10000,1.05000);
\draw[draw=gray,line width=0.6pt,densely dashed] (2.10000,1.05000) -- (2.10000,-0.45000);
\draw[draw=pbrust,line width=1.1pt] (2.10000,-0.45000) -- (5.60000,-0.45000) -- (5.60000,-0.20000) -- (6.30000,-0.20000);
\node[anchor=east,text=black,inner sep=1pt] at (1.28000,1.05000) {$1-\theta$};
\node[anchor=west,text=black,inner sep=1pt] at (5.74000,-0.48000) {$-\theta$};
\draw[draw=pbteal,line width=1.2pt] (8.40000,-0.20000) -- (14.40000,-0.20000);
\node[anchor=center,text=black,inner sep=1pt] at (11.40000,0.25000) {$0$};
\node[anchor=center,text=black,inner sep=1pt] at (1.75000,-0.85000) {$B_j$};
\node[anchor=center,text=black,inner sep=1pt] at (3.85000,-0.85000) {$C_j$};
\draw[draw=gray,line width=0.6pt] (1.40000,-1.10000) -- (1.40000,-1.20000) -- (5.60000,-1.20000) -- (5.60000,-1.10000);
\node[anchor=center,text=black,inner sep=1pt] at (3.50000,-1.41000) {$D_j$};
\node[anchor=center,text=black,inner sep=1pt] at (11.40000,-0.85000) {$F$};
\end{tikzpicture}
\caption{Recovering a Haar function from two non-negative vectors.
The first two rows show the corrected pair, with the first vector
scaled by $\kappa^{-1}$. Both contain the same addition
$\theta\beta_j\1_A$ and the same function $\kappa^{-1}s_j$ on $F$.
Subtracting cancels both terms and leaves $\varphi_j$.
The drawing uses $\theta=1/6$ and $\beta_j=1$, the smallest allowed
value. The auxiliary profile and interval lengths are schematic.}
\label{fig:positive-pair}
\end{figure}
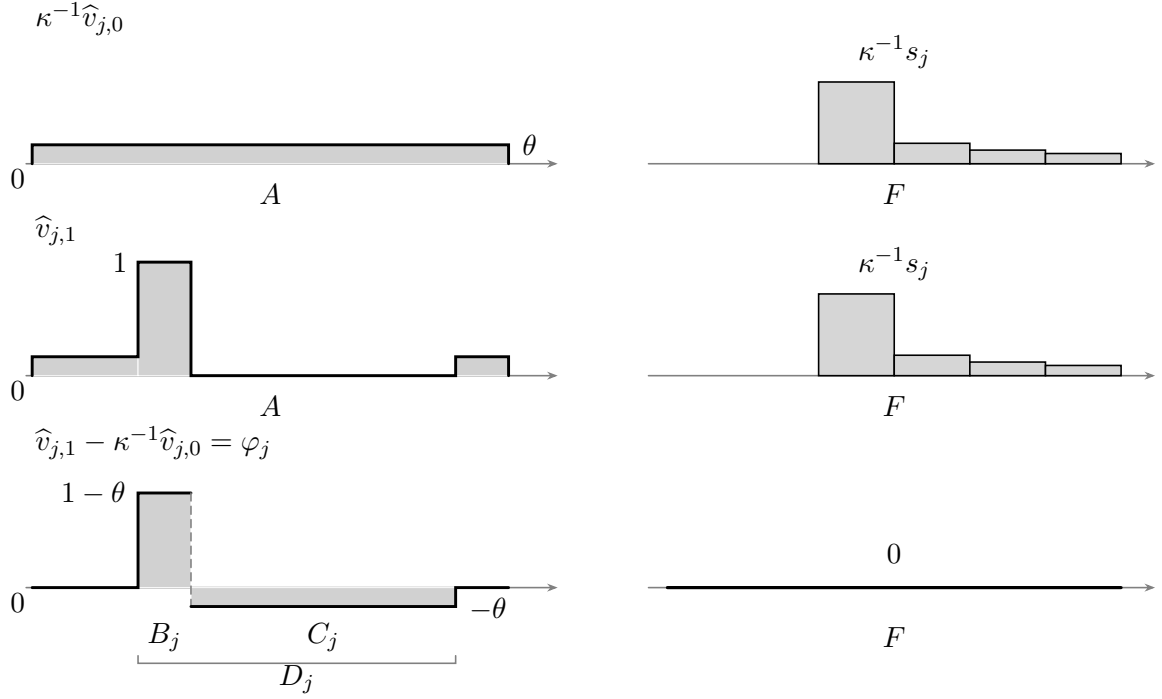

We now estimate $\norm{C_A-I}=\norm{S_A^*}$.
The operators $S_A$ on $L^q$ and $S_A^*$ on $L^p$ have the same
norm, and \eqref{eq:a} and \eqref{eq:batch-parameters} give
\begin{align}
\norm{S_A^*}=\norm{S_A}
&=\kappa\theta|A|^{1/p}\norm{z_N}_q\notag\\
&=\frac{2M}{\eta}\theta^{1/q}
 \left(\frac{|A|}{\ell}\right)^{1/p}\norm{z_N}_q\notag\\
&=\frac{2M}{\eta}\left(\frac{R}{2}\right)^{1/q}
 \left(\frac{|A|}{\ell}\right)^{1/p}\omega_{RL}
\longrightarrow0\quad(L\to\infty).
\label{eq:common-estimate}
\end{align}
Here, only $\omega_{RL}$ depends on $L$, since
$p,A,\mathcal P,\delta,\eta$ have already been fixed. Thus, we
can make $\norm{S_A}$ arbitrarily small by increasing $L$, while
the bound in \eqref{eq:raw-constant} remains unchanged.

We collect the properties of this construction in the following
proposition. We write $(w_l)_{l=1}^{m+N}$ for the ordered list
\begin{equation}\label{eq:stage-order}
v_{1,0},v_{1,1},\ldots,v_{N,0},v_{N,1},v_{N+1},\ldots,v_m,
\end{equation}
and $\psi_l$ for the corresponding coefficient function from
\eqref{eq:coefficients}. We form $\mathcal P'$ by adjoining the unit intervals of $F$ to the
refined partition of $A$.

\begin{proposition}[One stage of the construction]\label{prop:stage}
For $A,\mathcal P,L$ as above, the following assertions hold.
\begin{enumerate}
\item The families $(w_l)$ and $(\psi_l)$ are biorthogonal step
functions supported in $A\cup F$ and constant on each interval of
$\mathcal P'$. Every $w_l$ and $\psi_l$ has integral zero on
each interval of $\mathcal P$.
\item The projections
$E_{\mathcal P}+\sum_{l=1}^n w_l\otimes\psi_l$ have norm at most
$K_0=1+\delta+\eta$ on $L^p$ for $0\le n\le m+N$.
The final projection is $E_{\mathcal P'}$.
\item Each interval of the refined partition of $A$ is contained in an
interval of $\mathcal P$ and has at most $2/3$ of its length.
\item The positive operator $S_A$ satisfies \eqref{eq:common-estimate}
and $(I+S_A^*)w_l\ge0$ for every $l$.
In particular, its norm tends to zero as $L\to\infty$ with
$A$ and $\mathcal P$ fixed.
\end{enumerate}
\end{proposition}

\begin{proof}
The first assertion follows from \eqref{eq:haar},
\eqref{eq:pairs}, and \eqref{eq:coefficients}, together with the
biorthogonality calculation in \Cref{ss:pairing}.
The projection and interval-length bounds were proved in
\Cref{ss:projections,ss:splits}. The last assertion follows from
\eqref{eq:common-estimate} and \eqref{eq:local-positivity}.
\end{proof}

We can now extend any non-negative basis of the old step-function
space. As usual, the basis constant of a finite ordered basis is
the supremum of the norms of its partial sum projections on its span.

\begin{lemma}\label{lem:positive-extension}
Suppose that $(f_i)_{i=1}^d$ is a non-negative basis of
$E_{\mathcal P}L^p$ with basis constant $K$.
Let $(w_l)_{l=1}^{m+N}$ be the list in \Cref{prop:stage}, set
$\widehat w_l=(I+S_A^*)w_l$, and write $\tau=\norm{S_A}$.
Then
$$
f_1,\ldots,f_d,\widehat w_1,\ldots,\widehat w_{m+N}
$$
is a non-negative basis of $E_{\mathcal P'}L^p$ with basis constant
at most $\max\{K,K_0\}(1+\tau)^2$.
For fixed $A,\mathcal P,\delta,\eta$, the number $\tau$ can be
made arbitrarily small by increasing $L$.
\end{lemma}

\begin{proof}
We first add the vectors $w_l$ in the order \eqref{eq:stage-order},
before applying $I+S_A^*$. Biorthogonality
makes them linearly independent, and $E_{\mathcal P}w_l=0$ shows
that their span has zero intersection with the old span. Since
\Cref{prop:stage}(ii) gives
$$
E_{\mathcal P'}=E_{\mathcal P}
 +\sum_{l=1}^{m+N}w_l\otimes\psi_l,
$$
the combined sequence spans $E_{\mathcal P'}L^p$ and is therefore
a basis of that space. A partial sum ending among the old vectors
is an old partial sum composed with $E_{\mathcal P}$, so its norm
is at most $K$. A partial sum ending among the new vectors is
$E_{\mathcal P}+\sum_{l=1}^r w_l\otimes\psi_l$, whose norm is
at most $K_0$.

We now apply $C_A=I+S_A^*$. This fixes the old vectors and makes
the new ones non-negative. Since $S_A^*$ takes values in the old
span, both $C_A$ and $C_A^{-1}$ preserve $E_{\mathcal P'}L^p$.
Each partial sum projection $T$ becomes $C_ATC_A^{-1}$. By
\eqref{eq:local-inverse}, its norm is at most $(1+\tau)^2\norm{T}$. This proves the bound on the
basis constant, and \eqref{eq:common-estimate} shows that $\tau$
can be made arbitrarily small.
\end{proof}

\section{Assembling the basis}\label{sec:assembly}

We now repeat the extension from \Cref{lem:positive-extension}.
At each stage, the functions already constructed span all step
functions on a finite partition. The averaging operator in
\eqref{eq:averaging} is therefore the projection of norm one
needed to apply the lemma again.

We fix $\delta,\eta>0$ and $0<\rho<1$. The first two parameters
determine the bound $K_0$ in \eqref{eq:raw-constant}, while
$\rho$ bounds the sum of the norms of the operators $S_k$ used
to make the new functions non-negative. We choose their values in \Cref{ss:errors}. To ensure that
the supports eventually cover $\R$, we enumerate the unit intervals
$[l,l+1)$, $l\in\mathbb Z$, as $(I_n)_{n\ge1}$, with $I_1=[0,1)$.

\subsection{Choosing the stages}\label{ss:stages}

We start with the single function $f_1=v_*=\xi_*=\1_{I_1}$,
which spans the step-function space for $A_1=I_1$ and
$\mathcal P_1=\{I_1\}$. Suppose that $A_k$ and its finite interval
partition $\mathcal P_k$ have been constructed. We set
$$
R_k=\#\mathcal P_k,\quad
\ell_k=\min_{D\in\mathcal P_k}|D|.
$$
For a choice $L_k\ge2$, the parameters in
\eqref{eq:batch-parameters} and \eqref{eq:a} are
$$
\theta_k=\frac{1}{2L_k},\quad
N_k=R_kL_k,\quad m_k=2N_k,\quad
\kappa_k=\frac{2M}{\eta}\theta_k^{-1/p}\ell_k^{-1/p}.
$$
We take the first $m_k$ unused unit intervals in our enumeration
and denote their union by $F_k$. We write $z_k$ for the step
function obtained by placing $z_{N_k}$ on these intervals and
$S_k$ for the corresponding operator in \eqref{eq:local-S}.
The bound in \eqref{eq:common-estimate} is independent of the
locations of the new intervals. Thus, with $A_k$ and
$\mathcal P_k$ fixed, we can choose $L_k$ large enough that
\begin{equation}\label{eq:stage-smallness}
\norm{S_k}<\rho\,2^{-k}.
\end{equation}
We apply \Cref{prop:stage} with these choices and denote its
vectors and coefficient functions by
$(w_{k,l},\psi_{k,l})_{l=1}^{m_k+N_k}$, ordered as in
\eqref{eq:stage-order}.
Writing $\nu_k$ for the number of vectors before stage $k$, so
that $\nu_1=1$ and $\nu_{k+1}=\nu_k+m_k+N_k$, we append
the non-negative functions
\begin{equation}\label{eq:stage-functions}
f_{\nu_k+l}=(I+S_k^*)w_{k,l}
\quad(1\le l\le m_k+N_k).
\end{equation}
We set $A_{k+1}=A_k\cup F_k$ and obtain $\mathcal P_{k+1}$ by
refining $\mathcal P_k$ as above and adjoining the new unit intervals.
By \Cref{lem:positive-extension}, this extends the old basis to
a non-negative basis of $E_{\mathcal P_{k+1}}L^p$ without changing
any earlier vector. The change in the partition is illustrated
in \Cref{fig:successive-stages}.

\begin{figure}[htbp]
\centering
\begin{tikzpicture}[x=1cm,y=1cm,font=\small,line cap=round]
\node[anchor=center,text=black,inner sep=1pt] at (2.90000,4.20000) {Old support $A_k$};
\node[anchor=center,text=black,inner sep=1pt] at (11.40000,4.20000) {New intervals $F_k$};
\path[draw=pbblue,fill=pbblue!10,line width=0.6pt] (0.00000,3.20000) rectangle (1.40000,3.70000);
\path[draw=pbblue,fill=pbblue!10,line width=0.6pt] (1.40000,3.20000) rectangle (2.80000,3.70000);
\path[draw=pbblue,fill=pbblue!10,line width=0.6pt] (3.00000,3.20000) rectangle (4.40000,3.70000);
\path[draw=pbblue,fill=pbblue!10,line width=0.6pt] (4.40000,3.20000) rectangle (5.80000,3.70000);
\path[draw=pbteal,fill=pbteal!13,line width=0.6pt] (8.40000,3.20000) rectangle (9.60000,3.70000);
\node[anchor=center,text=black,inner sep=1pt] at (9.00000,3.45000) {$J_1$};
\path[draw=pbteal,fill=pbteal!13,line width=0.6pt] (9.95000,3.20000) rectangle (11.15000,3.70000);
\node[anchor=center,text=black,inner sep=1pt] at (10.55000,3.45000) {$J_2$};
\path[draw=pbteal,fill=pbteal!13,line width=0.6pt] (13.20000,3.20000) rectangle (14.40000,3.70000);
\node[anchor=center,text=black,inner sep=1pt] at (13.80000,3.45000) {$J_{m_k}$};
\node[anchor=center,text=black,inner sep=1pt] at (12.25000,3.45000) {$\cdots$};
\path[draw=pbblue,fill=pbblue!18,line width=0.6pt] (0.00000,1.30000) rectangle (0.35000,1.80000);
\path[draw=pbblue,fill=pbblue!18,line width=0.6pt] (0.35000,1.30000) rectangle (0.61250,1.80000);
\path[draw=pbblue,fill=pbblue!5,line width=0.6pt] (0.61250,1.30000) rectangle (1.40000,1.80000);
\path[draw=pbblue,fill=pbblue!18,line width=0.6pt] (1.40000,1.30000) rectangle (1.75000,1.80000);
\path[draw=pbblue,fill=pbblue!18,line width=0.6pt] (1.75000,1.30000) rectangle (2.01250,1.80000);
\path[draw=pbblue,fill=pbblue!5,line width=0.6pt] (2.01250,1.30000) rectangle (2.80000,1.80000);
\path[draw=pbblue,fill=pbblue!18,line width=0.6pt] (3.00000,1.30000) rectangle (3.35000,1.80000);
\path[draw=pbblue,fill=pbblue!18,line width=0.6pt] (3.35000,1.30000) rectangle (3.61250,1.80000);
\path[draw=pbblue,fill=pbblue!5,line width=0.6pt] (3.61250,1.30000) rectangle (4.40000,1.80000);
\path[draw=pbblue,fill=pbblue!18,line width=0.6pt] (4.40000,1.30000) rectangle (4.75000,1.80000);
\path[draw=pbblue,fill=pbblue!18,line width=0.6pt] (4.75000,1.30000) rectangle (5.01250,1.80000);
\path[draw=pbblue,fill=pbblue!5,line width=0.6pt] (5.01250,1.30000) rectangle (5.80000,1.80000);
\path[draw=pbteal,fill=pbteal!13,line width=0.6pt] (8.40000,1.30000) rectangle (9.60000,1.80000);
\node[anchor=center,text=black,inner sep=1pt] at (9.00000,1.55000) {$J_1$};
\path[draw=pbteal,fill=pbteal!13,line width=0.6pt] (9.95000,1.30000) rectangle (11.15000,1.80000);
\node[anchor=center,text=black,inner sep=1pt] at (10.55000,1.55000) {$J_2$};
\path[draw=pbteal,fill=pbteal!13,line width=0.6pt] (13.20000,1.30000) rectangle (14.40000,1.80000);
\node[anchor=center,text=black,inner sep=1pt] at (13.80000,1.55000) {$J_{m_k}$};
\node[anchor=center,text=black,inner sep=1pt] at (12.25000,1.55000) {$\cdots$};
\node[anchor=center,text=black,inner sep=1pt] at (2.90000,2.75000) {Refine the partition};
\node[anchor=center,text=black,inner sep=1pt] at (11.40000,2.75000) {Include the new intervals};
\draw[draw=pbblue,line width=0.6pt,-{Stealth[length=4pt,width=3pt]}] (2.90000,2.50000) -- (2.90000,1.94000);
\draw[draw=pbteal,line width=0.6pt,-{Stealth[length=4pt,width=3pt]}] (11.40000,2.50000) -- (11.40000,1.94000);
\node[anchor=center,text=black,inner sep=1pt] at (2.90000,0.95000) {Refinement of $\mathcal P_k$};
\node[anchor=center,text=black,inner sep=1pt] at (11.40000,0.95000) {Whole unit intervals};
\draw[draw=gray,line width=0.6pt] (0.00000,0.60000) -- (0.00000,0.42000) -- (14.40000,0.42000) -- (14.40000,0.60000);
\node[anchor=center,text=black,inner sep=1pt] at (7.20000,0.05000) {$\mathcal P_{k+1}\quad\hbox{on}\quad A_{k+1}=A_k\cup F_k$};
\end{tikzpicture}
\caption{One stage of the induction. We refine $\mathcal P_k$ and
adjoin the auxiliary unit intervals to obtain $\mathcal P_{k+1}$.
The resulting non-negative vectors, together with the unchanged old
vectors, span the entire space $E_{\mathcal P_{k+1}}L^p$.
The drawing uses two splittings per interval of $\mathcal P_k$ and
groups the intervals by their role. Repeating the construction both
exhausts $\R$ and refines the partition on each fixed unit interval
into arbitrarily short intervals.}
\label{fig:successive-stages}
\end{figure}
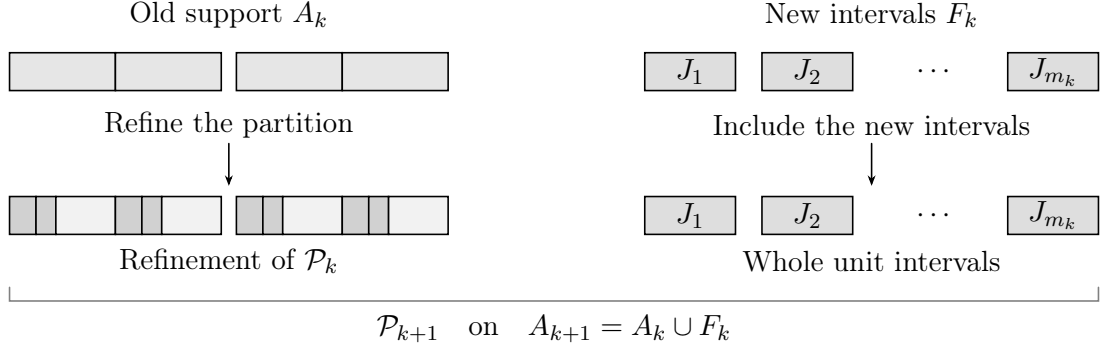

We also add the index $k$ to the functions used at stage $k$.
Thus, $H_{k,j}$ and $\varphi_{k,j}$ are the functions in
\eqref{eq:haar} at stage $k$ and are supported in $A_k$.
The functions $d_{k,j}$ and $s_{k,j}$ from \Cref{prop:aux} are
placed on the new unit intervals and are supported in $F_k$.

Since we always use the first unused intervals, each unit interval
is included after finitely many stages. At every subsequent stage,
the length of the longest subinterval in its partition decreases
by a factor of at most $2/3$. Consequently, the partitions on each
fixed unit interval consist of arbitrarily short intervals as
$k\to\infty$.

By \Cref{lem:positive-extension}, every $f_n$ is a non-negative
compactly supported step function, and
$$
\operatorname{span}\{f_n:n\le\nu_k\}=E_{\mathcal P_k}L^p
\quad(1<p\le\infty).
$$
We now prove convergence of the partial sums and identify the
coefficient functions.

\subsection{Density of the step-function spaces}\label{ss:positive-completeness}

We first show that $E_{\mathcal P_k}h\to h$ in $L^r(\R)$ for
every $1\le r<\infty$. Suppose that $h$ is continuous and compactly
supported. All unit intervals meeting its support eventually belong
to $A_k$, and on each interval of $\mathcal P_k$ the difference
between $h$ and its average is bounded by the oscillation of $h$
there. On the fixed unit intervals meeting the support of $h$,
the lengths of the intervals in $\mathcal P_k$ tend uniformly to zero.
Uniform continuity therefore makes the difference between $h$ and
$E_{\mathcal P_k}h$ tend uniformly to zero. Moreover, the difference is supported in
the fixed finite union of unit intervals meeting the support of
$h$. This proves convergence in $L^r$, and the general case follows
by density and contractivity of the averaging projections.

Thus, the step-function spaces have dense union in every $L^r(\R)$
with $1\le r<\infty$. To use the convergence of the averaging
projections, we first return to the vectors $w_{k,l}$ before applying
$I+S_k^*$ in \eqref{eq:stage-functions}. We will then recover $(f_n)$
by applying a single invertible operator to these vectors.

\subsection{The uncorrected vectors and their coefficients}\label{ss:raw}

We combine the ordered lists \eqref{eq:stage-order} from all
stages into one sequence. We set $v_1=v_*$, $\xi_1=\xi_*$,
$v_{\nu_k+l}=w_{k,l}$, and $\xi_{\nu_k+l}=\psi_{k,l}$.
For $(v_n)$ and $(\xi_n)$, the partial sum projections at the
ends of stages are exactly the averaging operators in
\eqref{eq:averaging}. This lets us prove
convergence on $L^q$, including when $q=1$.

\begin{lemma}\label{lem:raw}
If $1<p<\infty$, the sequence $(v_n)$ is a Schauder basis of
$L^p(\R)$ with coefficient functions $(\xi_n)$.
For the full range $1<p\le\infty$, the sequence $(\xi_n)$ is a
Schauder basis of $L^q(\R)$ with coefficient functions $(v_n)$.
The partial sum projections
\begin{equation}\label{eq:raw-projections}
P_nf=\sum_{j=1}^n\pair{f}{\xi_j}v_j,\quad
R_ng=\sum_{j=1}^n\pair{v_j}{g}\xi_j
\end{equation}
satisfy $\norm{P_n}_{L^p\to L^p}=\norm{R_n}_{L^q\to L^q}
\le K_0=1+\delta+\eta$.
\end{lemma}

\begin{proof}
Each stage gives a biorthogonal system by \Cref{prop:stage}(i).
For $n\le\nu_k$, the functions $v_n$ and $\xi_n$ are supported
in $A_k$ and constant on each interval of $\mathcal P_k$, whereas
every $w_{k,l}$ and $\psi_{k,l}$ has integral zero on these intervals.
Thus, the pairings between different stages vanish in both orders,
and the full system is biorthogonal.

The first projection is $P_1=E_{\mathcal P_1}$.
If $P_{\nu_k}=E_{\mathcal P_k}$, then the projections during
stage $k$ are
$$
P_{\nu_k+r}=E_{\mathcal P_k}
 +\sum_{l=1}^r w_{k,l}\otimes\psi_{k,l}
\quad(0\le r\le m_k+N_k).
$$
By \Cref{prop:stage}(ii), these projections have norm at most
$K_0$, and the last one equals $E_{\mathcal P_{k+1}}$.
Induction therefore gives
$$
P_{\nu_k}=E_{\mathcal P_k},\quad \sup_n\norm{P_n}\le K_0.
$$
Since $q<\infty$, we have $(L^q)^*=L^p$ and $R_n^*=P_n$, so
$\norm{R_n}=\norm{P_n}$. Moreover, the averaging projections
satisfy $\pair{E_{\mathcal P_k}f}{g}=\pair{f}{E_{\mathcal P_k}g}$.
It follows that $R_{\nu_k}=E_{\mathcal P_k}$ on $L^q$ as well.

We have already proved that $E_{\mathcal P_k}h\to h$ in $L^r$
for every $h\in L^r$ and every $1\le r<\infty$. To pass from the
ends of stages to all partial sums, we let $T_n$ denote $P_n$ on
$L^p$ when $p<\infty$, or $R_n$ on $L^q$, and write $L^r$ for
the corresponding space. For $n\ge\nu_k$, biorthogonality gives
$T_nT_{\nu_k}=T_{\nu_k}$, and hence
$$
\norm{h-T_nh}_r
=\norm{(I-T_n)(h-E_{\mathcal P_k}h)}_r
\le(1+K_0)\norm{h-E_{\mathcal P_k}h}_r.
$$
The right-hand side tends to zero as $k\to\infty$, uniformly
for $n\ge\nu_k$. Thus, the full sequence of partial sums converges,
and biorthogonality gives uniqueness of the expansions.
\end{proof}

\subsection{Identifying the coefficient functions}\label{ss:global}

We now sum the operators $S_k$ from \eqref{eq:local-S}.
We will show that the adjoint of $I+\sum_k S_k$ takes each $v_n$
to the corresponding non-negative function $f_n$. The inverse
of $I+\sum_k S_k$ will then give the coefficient functions. We set
\begin{equation}\label{eq:global-S}
S=\sum_{k=1}^\infty S_k,\quad
S_kg=\kappa_k\theta_k\left(\int_{A_k}g\right)z_k.
\end{equation}
By \eqref{eq:stage-smallness}, the series converges in operator
norm on $L^q$ and satisfies $\norm{S}\le\rho<1$. Consequently,
$U=I+S\colon L^q(\R)\to L^q(\R)$ is invertible, with
\begin{equation}\label{eq:U-bounds}
U^{-1}=\sum_{r=0}^\infty(-S)^r,\quad
\norm{U}\le1+\rho,\quad
\norm{U^{-1}}\le\frac{1}{1-\rho}.
\end{equation}
Its adjoint $U^*=I+S^*$ is therefore invertible on $L^p(\R)$.
We will show that the non-negative functions already constructed
and their coefficient functions are given by
\begin{equation}\label{eq:final-basis}
f_n=U^*v_n,\quad \gamma_n=U^{-1}\xi_n.
\end{equation}

\begin{proposition}[The non-negative basis]\label{prop:corrected}
Every $f_n$ is a non-negative compactly supported step function.
If $1<p<\infty$, the sequence $(f_n)$ is a Schauder basis of
$L^p(\R)$ with coefficient functions $(\gamma_n)$ and
\begin{equation}\label{eq:final-bound}
K((f_n))\le(1+\delta+\eta)\frac{1+\rho}{1-\rho}.
\end{equation}
At $p=\infty$, the sequence $(\gamma_n)$ is a Schauder basis
of $L^1(\R)$ with coefficient functions $(f_n)$ and the same
upper bound for its basis constant.
\end{proposition}

\begin{proof}
We first show that $S_l^*v_n=0$ when $v_n$ is a vector from
stage $k$ and $l\ne k$. The adjoint of \eqref{eq:global-S}
is given by
$$
S^*f=\sum_{l=1}^\infty
\kappa_l\theta_l\pair{f}{z_l}\1_{A_l},
$$
with convergence in $L^p$, including when $p=\infty$.
Each $z_l$ is constant on every unit interval, while
$\varphi_{k,j}$ has integral zero on every such interval by
\eqref{eq:haar}. Hence, $\pair{\varphi_{k,j}}{z_l}=0$ for every $l$. The function
$s_{k,j}$ is supported in $F_k$, and the sets $F_l$ are pairwise
disjoint and disjoint from $I_1$. Therefore,
\begin{equation}\label{eq:global-identities}
S^*\varphi_{k,j}=0,\quad
S^*s_{k,j}=S_k^*s_{k,j},\quad S^*v_*=0.
\end{equation}
Every vector from stage $k$ is either $s_{k,j}$ or
$\kappa_k^{-1}s_{k,j}+\varphi_{k,j}$. Since
$S_k^*\varphi_{k,j}=0$ as well, \eqref{eq:global-identities}
shows that each such vector $v_n$ satisfies
\begin{equation}\label{eq:positive-functions}
f_n=(I+S^*)v_n=(I+S_k^*)v_n\ge0
\end{equation}
by the construction at stage $k$. This also proves the first
identity in \eqref{eq:final-basis}. Formula \eqref{eq:positive-functions} also shows that $f_n$
is a step function supported in $A_k\cup F_k$, because $v_n$
is such a function and $S_k^*v_n$ is a multiple of $\1_{A_k}$. The initial vector remains $f_1=\1_{I_1}$.
In particular, every $f_n$ has compact support despite the
infinite sum defining $S$.

The formula $\gamma_n=U^{-1}\xi_n$ preserves biorthogonality,
since
$$
\pair{U^*v_n}{U^{-1}\xi_l}=\pair{v_n}{\xi_l}=\delta_{nl}.
$$
For $f\in L^p(\R)$, \eqref{eq:final-basis} gives
$\pair{f}{\gamma_j}=\pair{f}{U^{-1}\xi_j}
=\pair{(U^*)^{-1}f}{\xi_j}$.
Thus, for $p<\infty$, the partial sum projections are obtained
from $P_n$ in \eqref{eq:raw-projections} by
$$
\sum_{j=1}^n f_j\otimes\gamma_j
=U^*P_n(U^*)^{-1}.
$$
By \Cref{lem:raw}, these partial sums converge in $L^p$ to the
function being expanded. Biorthogonality gives uniqueness, and
the equalities $\norm{U^*}=\norm{U}$ and
$\norm{(U^*)^{-1}}=\norm{U^{-1}}$, together with
\eqref{eq:U-bounds}, give \eqref{eq:final-bound}.

When $p=\infty$, we use $R_n$ from \eqref{eq:raw-projections},
which acts on $L^1$. The partial sum projections for $(\gamma_n)$
are
$$
\sum_{j=1}^n\gamma_j\otimes f_j=U^{-1}R_nU.
$$
For every $g\in L^1$, \Cref{lem:raw} gives $R_nUg\to Ug$
in $L^1$. Applying the bounded operator $U^{-1}$ proves convergence
of these partial sums to $g$. Their norms satisfy the same bound,
and biorthogonality again gives uniqueness. This proves the
endpoint assertion.
\end{proof}

\subsection{Choice of the parameters}\label{ss:errors}

Given $\varepsilon>0$, we choose
$$
t=\frac{\min\{\varepsilon,1\}}{8},\quad
\delta=\eta=\frac{t}{2},\quad \rho=t.
$$
For $1<p<\infty$, \eqref{eq:final-bound} yields
$$
K((f_n))\le\frac{(1+t)^2}{1-t}
=1+\frac{3t+t^2}{1-t}<1+4t<1+\varepsilon,
$$
where we used $0<t\le1/8$. When $p=\infty$, the same estimate
holds for the basis $(\gamma_n)$ of $L^1$.

Note that the parameters
$\delta,\eta,\rho$, and hence the constants in \Cref{prop:aux},
are fixed before the induction begins. At stage $k$, the set
$A_k$ and partition $\mathcal P_k$ have already been constructed.
We then choose $L_k$ using \eqref{eq:common-estimate} to obtain
\eqref{eq:stage-smallness}. Thus, each choice depends only on
the fixed parameters and the preceding stages.

\section{Non-negative bases in \texorpdfstring{$L^1(\R)$}{L1(R)}}\label{sec:positive-l1}

We now prove the existence assertion of \Cref{thm:positive} at
$p=1$. We again extend a basis of an entire step-function space
by refining its partition and adding new unit intervals. Here,
one new interval and two new functions suffice for each splitting.
The key point is that the projection discarding both new functions
is contractive. Every earlier partial sum therefore keeps the same
norm when the basis is extended. Consequently, we only need to estimate the
partial sums that stop between the two new functions.

\subsection{Adding two functions}

For a bounded interval $V$ of positive length, we write
$$
e_V=\frac{\1_V}{|V|},\quad \norm{e_V}_1=1.
$$
 We suppose that $\mathcal P$ partitions a
finite union $A$ of unit intervals and that we have chosen a
non-negative basis of $E_{\mathcal P}L^1$. We fix $0<\theta<1/2$ and split one interval
$D\in\mathcal P$ into a left interval $B$ and a right interval
$C$, with
$$
|B|=\theta|D|,\quad |C|=(1-\theta)|D|.
$$
We also choose a unit interval $J$ disjoint from $A$ and let
$\mathcal P'$ be the partition obtained by replacing $D$ with
$B,C$ and adjoining $J$.

\begin{lemma}\label{lem:l1-extension}
Let the notation be as above. Appending the two non-negative functions
\begin{equation}\label{eq:l1-pair}
u=e_C+e_J,\quad v=e_B+e_J,
\end{equation}
in this order, extends the old basis to a basis of
$E_{\mathcal P'}L^1$. On this enlarged space, let $R$ retain
the old vectors but discard $u,v$, and let $S$ retain the old
vectors and $u$ but discard $v$. Then we have
\begin{equation}\label{eq:l1-extension-bounds}
\norm{R}=1,\quad \norm{S}=1+2\theta.
\end{equation}
Moreover, every partial sum projection ending among the old vectors has
the same norm on the enlarged space as on $E_{\mathcal P}L^1$.
\end{lemma}

\begin{proof}
The old span contains $e_D$. Since
\begin{equation}\label{eq:l1-span}
e_D=\theta e_B+(1-\theta)e_C,\quad
e_J=(1-\theta)u+\theta v-e_D,
\end{equation}
the enlarged list spans $e_J$ and hence also
$e_B=v-e_J$ and $e_C=u-e_J$. It therefore spans the whole
space $E_{\mathcal P'}L^1$. This space has two more dimensions
than the old space, so the enlarged list is a basis.

The functions $e_V$, $V\in\mathcal P'$, have disjoint supports
and norm one. Thus, for arbitrary real or complex coefficients, we have
$$
\left\|\sum_{V\in\mathcal P'}a_Ve_V\right\|_1
=\sum_{V\in\mathcal P'}|a_V|.
$$
It follows that the norm of any operator on this space is the
largest of the norms of its values on these functions. The map
$R$ fixes every unaffected $e_V$ and satisfies
\begin{equation}\label{eq:l1-old-projection}
Re_B=e_D,\quad Re_C=e_D,\quad Re_J=-e_D.
\end{equation}
Indeed, these formulas fix $e_D$ and annihilate $u,v$.
Every image has norm one, so $\norm{R}=1$.

The map $S$ also fixes the unaffected $e_V$. On the three
remaining functions, it is given by
\begin{equation}\label{eq:l1-intermediate}
\begin{aligned}
Se_B&=\theta e_B-(1-\theta)e_J,\\
Se_C&=\theta e_B+e_C+\theta e_J,\\
Se_J&=-\theta e_B+(1-\theta)e_J.
\end{aligned}
\end{equation}
These formulas imply that $Se_D=e_D$, $Su=u$, and $Sv=0$, as required.
The three images have norms $1$, $1+2\theta$, and $1$,
respectively. Hence, $\norm{S}=1+2\theta$.

Finally, if $P$ is an old partial sum projection, its extension
to the enlarged space is $PR$. Since $R$ is contractive,
$\norm{PR}\le\norm{P}$. Restricting $PR$ to the old space
gives the reverse inequality. Thus, its norm is unchanged.
\end{proof}

The pair in \eqref{eq:l1-pair} recovers the same unequal Haar
difference used in \Cref{ss:splits}:
\begin{equation}\label{eq:l1-haar-recovery}
\varphi:=\1_B-\theta\1_D
=\theta(1-\theta)|D|(v-u).
\end{equation}
Thus, the two constructions use pairs of non-negative functions
to separate the values on $B$ and $C$. At $p=1$, the exact
calculation in \eqref{eq:l1-extension-bounds} also ensures that
later extensions do not increase any earlier projection norm.

\subsection{Repeating the extension}

\begin{proposition}\label{prop:positive-l1}
For every $0<\theta<1/2$, there is a non-negative Schauder basis
of $L^1(\R)$ consisting of compactly supported step functions
whose partial sum projections satisfy
\begin{equation}\label{eq:l1-exact-projections}
\norm{P_{2n-1}}=1,\quad
\norm{P_{2n}}=1+2\theta\quad(n\ge1).
\end{equation}
The construction works over either scalar field.
\end{proposition}

\begin{proof}
We enumerate the unit intervals as $(I_n)_{n\ge1}$, with
$I_1=[0,1)$, and start with $f_1=\1_{I_1}$. At each splitting,
we use the first unit interval in this enumeration that has not
yet been included. We organize the splittings into finite rounds.
At the beginning of a round, we list the intervals of the current
partition in increasing order of their left endpoints. We split
each listed interval once and append the pair from
\Cref{lem:l1-extension}. The two pieces of a split interval and
the newly added unit interval are first split in the next round.

We number the splittings consecutively and denote the intervals
at splitting $n$ by $D_n,B_n,C_n,J_n$. The basis vectors added
at this step are
\begin{equation}\label{eq:l1-basis}
f_{2n}=e_{C_n}+e_{J_n},\quad
f_{2n+1}=e_{B_n}+e_{J_n}.
\end{equation}
By \Cref{lem:l1-extension}, the first $2n+1$ vectors form a
non-negative basis of the full step-function space for the
partition after $n$ splittings. When a partial sum ending after
a complete pair is first formed, it is the identity on that
space and has norm one. A partial sum ending between the two
functions of a pair has norm $1+2\theta$. Every later extension
preserves these norms, by the last assertion of the lemma.
Consequently, \eqref{eq:l1-exact-projections} holds on every
finite span large enough to contain the corresponding pair.

It remains to prove that the finite spans have dense union.
Every unit interval is eventually included. At each subsequent
round, every interval of its partition is replaced by two
intervals, each of length at most $1-\theta$ times the old
length. The longest interval of the partition on each fixed
$I_n$ therefore has length tending to zero. As in
\Cref{ss:positive-completeness}, uniform continuity gives
approximation of continuous compactly supported functions by
these step-function spaces. Their union is consequently dense
in $L^1(\R)$.

The partial sum maps agree on the smaller spans and are uniformly
bounded by $1+2\theta$. They therefore extend to bounded operators
$P_m$ on $L^1$. To see that $P_mh\to h$, choose a finite linear
combination $h_0$ of the constructed vectors. Once $m$ includes
its last index, we have
$$
\norm{h-P_mh}_1
\le(2+2\theta)\norm{h-h_0}_1.
$$
Density proves convergence. The differences $P_m-P_{m-1}$
give bounded coefficient functionals and hence uniqueness of
the expansions. Thus, the sequence is a Schauder basis.

The operator norms on the dense union are preserved under
extension to $L^1$, so the equalities in
\eqref{eq:l1-exact-projections} continue to hold. In particular,
the lower bound for $P_{2n}$ is attained on $e_{C_n}$ by
\eqref{eq:l1-intermediate}. All the functions are real and
non-negative, and the norm calculations apply to real and
complex coefficients alike.
\end{proof}

Since $\theta$ can be arbitrarily small, the basis constants in
\Cref{prop:positive-l1} are arbitrarily close to one. In this
construction the first coordinate projection is contractive.

\section{Duality and the conclusions}\label{sec:duality}

The constructions give non-negative bases for $1\le p<\infty$
and the asserted upper bounds. We now prove the strict lower bound
for non-negative bases when $1<p<\infty$ and for bases with positive
biorthogonal functionals when $1\le p<\infty$. We then complete
the proofs of the main results.

\begin{samepage}
\subsection{Why the first projection cannot be contractive}\label{ss:lower}

For a non-negative basis of $L^p$ with $1<p<\infty$, a contractive
first coordinate projection would give equality in H\"older's
inequality. This would force the first coefficient function to be
non-negative as well.
Biorthogonality would then make every later basis vector vanish
on a set of positive measure, which is incompatible with completeness.

\begin{proposition}[No contractive first coordinate projection]\label{prop:lower}
Let $(f_n)$ be a Schauder basis of $L^p(\R)$ with coefficient
functions $(\gamma_n)$.
\begin{enumerate}
\item If $1<p<\infty$ and every $f_n$ is non-negative, then
$f_1\otimes \gamma_1$ is not contractive.
\item If $1\le p<\infty$ and every $\gamma_n$ is non-negative, then
$f_1\otimes \gamma_1$ is not contractive.
\end{enumerate}
In either case, $K((f_n))>1$.
\end{proposition}
\end{samepage}

\begin{proof}
To prove (i), suppose that the first coordinate projection is
contractive. Its norm must then equal one, so
$\norm{f_1}_p\norm{\gamma_1}_q=1$.
We set $h=\max\{\operatorname{Re}\gamma_1,0\}$.
Since $f_1\ge0$, H\"older's inequality gives
$$
1=\operatorname{Re}\int f_1\gamma_1
\le\int f_1h\le\norm{f_1}_p\norm{h}_q
\le\norm{f_1}_p\norm{\gamma_1}_q=1.
$$
Since $0\le h\le|\gamma_1|$ and $q<\infty$, equality of their norms
forces $h=|\gamma_1|$ almost everywhere. Consequently, $\gamma_1$ is real
and non-negative, even when the scalar field is complex.

The set $B=\{\gamma_1>0\}$ has positive measure. For $n\ge2$,
non-negativity and biorthogonality give $\int f_n\gamma_1=0$, so
$f_n=0$ almost everywhere on $B$. Restricting the basis expansion
of a function $f\in L^p$ supported in $B$ therefore gives
$$
f=\pair{f}{\gamma_1}f_1\1_B.
$$
This would make $L^p(B)$ at most one-dimensional, whereas $B$
contains two disjoint subsets of positive finite measure whose
indicators are linearly independent. This contradiction proves (i).

For (ii) with $p>1$, duality gives a non-negative Schauder basis
$(\gamma_n)$ of $L^q(\R)$ with coefficient functions $(f_n)$.
Part (i) gives $\norm{\gamma_1\otimes f_1}>1$. This operator is the
adjoint of $f_1\otimes \gamma_1$, so they have the same norm.

For $p=1$, we argue directly. If the first projection were
contractive, we would have
$\norm{f_1}_1\norm{\gamma_1}_\infty=1$.
Setting $h=\max\{\operatorname{Re}f_1,0\}$ and using $\gamma_1\ge0$,
we obtain
$$
1=\operatorname{Re}\int f_1\gamma_1
\le\int h\gamma_1\le\norm{h}_1\norm{\gamma_1}_\infty
\le\norm{f_1}_1\norm{\gamma_1}_\infty=1.
$$
As in (i), equality forces $f_1$ to be real and non-negative.
On the set $B=\{f_1>0\}$, which has positive measure, every
$\gamma_n$ with $n\ge2$ vanishes by non-negativity and biorthogonality.
Thus, a function $f\in L^1$ supported in $B$ would have all but
its first basis coefficient equal to zero and would satisfy
$f=\pair{f}{\gamma_1}f_1$. Choosing two disjoint subsets of $B$ of positive finite measure
gives two linearly independent indicator functions. They cannot
both be multiples of $f_1$, which proves (ii).

\end{proof}

\subsection{The non-negative basis and its dual consequence}\label{ss:proofs}

\begin{proof}[Proof of \Cref{thm:positive}]
For $1<p<\infty$, we carry out the construction with the
parameters in \Cref{ss:errors}. By \Cref{prop:corrected},
$(f_n)$ is a non-negative Schauder basis of $L^p(\R)$ consisting
of compactly supported step functions, with $K((f_n))<1+\varepsilon$.

For $p=1$, we apply \Cref{prop:positive-l1} with
$$
\theta=\frac{\varepsilon}{2(1+\varepsilon)}\in(0,1/2).
$$
The resulting basis has constant
$1+2\theta=1+\varepsilon/(1+\varepsilon)<1+\varepsilon$.
This proves the existence assertion over either scalar field.

Since $\varepsilon>0$ is arbitrary, the infimum of the basis
constants is one for every $1\le p<\infty$. For $1<p<\infty$,
\Cref{prop:lower}(i) shows that the first coordinate projection
of a non-negative basis cannot be contractive, so the infimum
is not attained in this range.
For $p=1$, nonattainment follows from the description of
monotone bases in $L^1$ due to Dor and Odell
\cite[Theorem~3.1]{DO}.
\end{proof}

\begin{proof}[Proof of \Cref{thm:functionals} when \texorpdfstring{$1<p<\infty$}{1<p<\infty}]
\begin{samepage}
We apply \Cref{thm:positive} with $q$ in place of $p$ and denote
the resulting non-negative basis of $L^q(\R)$ by $(\gamma_n)$.
By duality, its coefficient functions $(f_n)$ form a
Schauder basis of $L^p(\R)$ with the same basis constant. The
coefficient functions of $(f_n)$ are the original $\gamma_n$,
which are non-negative compactly supported step functions.
Thus, this basis has all the asserted properties, and
\Cref{prop:lower}(ii) shows that its basis constant is strictly
greater than one. The argument applies over either scalar field.

\end{samepage}
\end{proof}

\subsection{Positive biorthogonal functionals in \texorpdfstring{$L^1$}{L1}}\label{ss:endpoint}

\begin{proof}[Proof of \Cref{prop:endpoint} when \texorpdfstring{$p=1$}{p=1}]
We run the construction with $p=\infty$ and $q=1$, using the
vectors in \eqref{eq:tail-system}. By \Cref{prop:corrected},
$(\gamma_n)$ is a Schauder basis of $L^1(\R)$ whose coefficient
functions $(f_n)$ are non-negative compactly supported step functions.
The choice in \Cref{ss:errors} makes its basis
constant less than $1+\varepsilon$.

Interchanging the names of the basis vectors and coefficient
functions gives the notation in the statement. The first coordinate
projection cannot be contractive by \Cref{prop:lower}(ii).
The dual expansions converge in $\sigma(L^\infty,L^1)$, meaning
convergence after pairing with any function in $L^1$. For $g=\1_{\R}$,
convergence in norm fails because each partial sum is compactly supported.
All arguments apply over either scalar field.

\end{proof}

\section{Outline of the construction of a basis of translates for \texorpdfstring{$p>2$}{p>2}}\label{tr:sec:outline}

We now outline the construction of a Schauder basis of translates
for $2<p<\infty$. The case $p=2$ is treated separately in
\Cref{h2:sec:weighted}. Until then, we fix $2<p<\infty$ and write $q=p/(p-1)$.
We construct the generator as a limit of entire functions
$G_0,G_1,\ldots$, starting with a Gaussian smoothing of a real
function $e_0\in C_c^\infty((0,1))$ with $\norm{e_0}_p=1$. At each stage, we also choose a smooth compactly supported reference function $V_t$ and finitely many of its translates. We keep these reference vectors fixed when we proceed to the next stage. We first prove that the resulting sequence is a Schauder basis. We then replace
the generator used for each vector by the single limit
 $g=\lim_tG_t|_\R$. The total change will be small enough to preserve
the basis property.

The main step makes several new translates contain a copy of a
function that we want to approximate. Suppose that the current
 reference function $v$ and a real smooth compactly supported function $h$ with $\norm{h}_p=1$
are supported in $I=[-R,R]$. We choose large positive translation parameters $\lambda_1,\ldots,\lambda_N$ and add small copies of $h$ far to
the left of $I$
$$
v'=v+a\sum_{j=1}^N T_{-\lambda_j}h.
$$
Translating $v'$ by $\lambda_i$ brings the $i$-th copy back onto
$I$. The other terms remain outside $I$. We choose the parameters
so that all these remaining terms have disjoint supports, including
terms belonging to different translates. We choose $a$ and a
factor $d>0$ so that, after division by $d$, each remaining part
has norm one and the coefficient of $h$ is $N^{-1/q}$. Thus,
$$
d^{-1}T_{\lambda_i}v'=e_i+N^{-1/q}h,
\quad \norm{e_i}_p=1,
$$
where the $e_i$'s have disjoint supports outside $I$. The change to
the generator satisfies
$$
\norm{v'-v}_p=dN^{2/p-1}\longrightarrow0.
$$
Here $d$ tends to $\norm{v}_p$ as $N\to\infty$. This is where
we use $p>2$. The parameters only need to be chosen successively
far enough apart. This can be done due to the unboundedness of the prescribed
set, after the reduction to positive parameters at the start of
\Cref{tr:sec:packet}; see \Cref{tr:lem:placement,tr:lem:packet}.

We next explain how to approximate $h$ using these functions. Write
$x_i=e_i+N^{-1/q}h$ and $u=N^{-1/p}\sum_i e_i$. Note that
$\norm{u}_p=1$, $u$ and $h$ have disjoint supports, and
$$
\frac{1}{2}\sum_{i=1}^N N^{-1/p}x_i=\frac{1}{2}(h+u),\quad
\left\|h-\frac{1}{2}\sum_{i=1}^N N^{-1/p}x_i\right\|_p
=2^{-1/q}<1.
$$
We divide the error by its norm and apply the same construction to the resulting function of norm one, using new disjoint supports. Repeating this procedure gives approximations to $h$ with errors tending to zero. Once $h$ is in the closed span, so is every
$e_i=x_i-N^{-1/q}h$. In \Cref{tr:sec:branch}, we carry out this
iteration and estimate every partial sum, including those that stop
before all $N$ vectors of one family have been added.

There is also a completeness issue on the new supports. A single
function $e_i$ does not span all functions supported on
 a reserved finite union of open intervals $A_i$ containing its support. We therefore use \Cref{tr:lem:local-complement} to choose a basis
of a complementary subspace to $\F e_i$ in $L^p(A_i)$.
\Cref{tr:lem:local-complement} obtains such a basis from the Haar
system. We include these functions, together with smooth approximations to Haar functions
on the intervals between the new regions, among the functions to
be approximated at later stages. The bounded intervals used in
the construction eventually cover $\R$, and every one of these
functions is selected after finitely many stages.

To control the partial sums when these approximations are combined,
we must specify how often each function is used. Before starting with the next function $f_{j+1}$, we perform
sufficiently many steps for every earlier function, ensuring that
the number of completed steps for $f_1,\ldots,f_j$ is non-increasing
in this order. The relevant estimate in equation \eqref{tr:eq:four-bounds} shows that the part of a partial sum mapping the added coordinates back into the direction $f_j$ tends to zero with the number of steps. We make these bounds summable in $j$, and then continue to return to every $f_j$ infinitely often. This gives the uniform basis bound in \Cref{tr:prop:assembly}.

Finally, we let $M_{t-1}$ be the number of vectors chosen before stage $t$, and let $N_t$ be the number chosen at that stage. We choose the entire functions $G_t$ and the smooth compactly
supported reference functions $V_t$ so that
$$
 \norm{G_0-e_0}_p+
 \sum_{t=1}^\infty N_t\norm{G_t-V_t}_p+
\sum_{t=1}^\infty M_{t-1} \norm{G_t-G_{t-1}} _p
$$
is sufficiently small. The second term accounts for the new vectors
 and the third for changes to earlier vectors. We also ensure that the series of increments $G_t-G_{t-1}$
converges absolutely and uniformly on every compact subset of $\C$. The estimate in \eqref{tr:eq:total-error} then allows us to
replace the reference vectors by translates of one entire function,
 using an invertible operator close to the identity.

\section{A finite family of translates}\label{tr:sec:packet}

We first prove the finite construction described in the outline. To do so, let $\Lambda_0\subset\R$ be the prescribed unbounded set. After applying a reflection if necessary, we may suppose that it is
unbounded above. We choose $b\in\Lambda_0$ and set
$$
\Gamma=\{0\}\cup\{\lambda-b:\lambda\in\Lambda_0,\ \lambda>b\}.
$$
We will construct the basis using parameters in $\Gamma$. Translating
the entire basis by $b$ will then put every parameter in $\Lambda_0$.
Thus, throughout this part, we will assume that $\Gamma\subset[0,\infty)$ is unbounded
and contains zero.

To make the copies outside $I$ disjoint, we need to separate both
the parameters $\lambda_i$ and their differences $\lambda_i-\lambda_j$.
The following lemma achieves this using only the fact that $\Gamma$ is unbounded.

\begin{lemma}\label{tr:lem:placement}
Let $I=[-R,R]$, where $R\ge1$, and let $N\ge1$. There are
$\lambda_1<\cdots<\lambda_N$ in $\Gamma$, larger than any prescribed
finite set of old parameters, such that the intervals $I+c$, for
$$
c\in\{0\}\cup\{\lambda_i,-\lambda_i:1\le i\le N\}
\cup\{\lambda_i-\lambda_j:i\ne j\},
$$
are pairwise disjoint.
\end{lemma}

\begin{proof}
Since $\Gamma$ is unbounded above, we may choose its elements
successively so that
\begin{equation}\label{tr:eq:placement}
\lambda_j>4\sum_{i<j}\lambda_i+4R,
\end{equation}
with $\lambda_1$ also larger than all the old parameters. The centers are given by $0$, $\pm\lambda_i$, and
$\lambda_i-\lambda_j$ with $i\ne j$. Any two of these
expressions differ in the coefficient of at least one $\lambda_i$. Subtracting two different expressions therefore gives $\sum_i a_i\lambda_i$, where $a_i\in\{-2,-1,0,1,2\}$
and not all $a_i$ are zero.
If $j$ is the largest index with $a_j\ne0$, then
$$
\left|\sum_i a_i\lambda_i\right|
\ge\lambda_j-2\sum_{i<j}\lambda_i>2R.
$$
This proves the required disjointness.
\end{proof}

We now use the translation parameters provided by
\Cref{tr:lem:placement} to add copies of the target function
to the generator. The common norm of the remaining parts determines the coefficient of each copy.

\begin{lemma}\label{tr:lem:packet}
Let $v,h$ be real smooth compactly supported functions supported in $I=[-R,R]$, with
$v\ne0$ and $\norm{h}_p=1$. Given $\varepsilon>0$ and finitely many
old parameters, there are a real smooth compactly supported function $v'$, new parameters
$\lambda_1<\cdots<\lambda_N$ in $\Gamma$ larger than all the old
parameters, a number $d>0$, and smooth compactly supported functions $e_1,\ldots,e_N$ of norm one with disjoint supports outside $I$
such that
\begin{equation}\label{tr:eq:exact-packet}
T_{\lambda_i}v'=d\bigl(e_i+N^{-1/q}h\bigr),\quad
\norm{v'-v}_p\le\varepsilon.
\end{equation}
Moreover, $v'-v$ is supported to the left of $I$.
\end{lemma}

\begin{proof}
For a sufficiently large integer $N$, we set
\begin{equation}\label{tr:eq:packet-scales}
d^p=\frac{\norm{v}_p^p}{1-(N-1)N^{-(p-1)}},\quad
a=dN^{-1/q}.
\end{equation}
The denominator tends to one because $p>2$. We choose the parameters
by \Cref{tr:lem:placement} and define
\begin{equation}\label{tr:eq:generator-edit}
v'=v+a\sum_{j=1}^N T_{-\lambda_j}h.
\end{equation}
The summands added to $v$ have disjoint supports to the left of
$I$, since $\lambda_j>4R$.
After translation by $\lambda_i$, the term with $j=i$ is $ah$,
supported in $I$. Separating this term gives
$T_{\lambda_i}v'=ah+w_i$, where
$$
w_i=T_{\lambda_i}v+a\sum_{j\ne i}T_{\lambda_i-\lambda_j}h.
$$
All the terms in these expressions have disjoint supports, both
within one $w_i$ and between different $w_i$. They also vanish on
$I$. The placement is illustrated in \Cref{fig:translate-placement}. Hence,
$$
\norm{w_i}_p^p=\norm{v}_p^p+(N-1)d^pN^{-(p-1)}=d^p.
$$
Taking $e_i=w_i/d$ gives the first identity in
\eqref{tr:eq:exact-packet}. Disjointness also gives
$\norm{v'-v}_p^p=Na^p$, and hence
\begin{equation}\label{tr:eq:edit-cost}
\norm{v'-v}_p=dN^{2/p-1}\longrightarrow0\quad(N\to\infty).
\end{equation}
Here, $d\to\norm{v}_p$, and the exponent $2/p-1$ is negative.
We first choose $N$ to make this norm at most $\varepsilon$, and
then choose the parameters from $\Gamma$.
\end{proof}

\begin{figure}[htbp]
\color{black}
\centering
\begin{tikzpicture}[x=1cm,y=1cm,font=\small,line cap=round,
 every node/.style={text=black,inner sep=1pt}]
\node[anchor=west] at (0,4.65) {Copies and their translates};
\node at (7.8,4.65) {On $I$};
\path[fill=black!4] (7.1,0.45) rectangle (8.5,4.25);
\draw[draw=black,line width=0.6pt,densely dashed] (7.1,0.45) -- (7.1,4.25);
\draw[draw=black,line width=0.6pt,densely dashed] (8.5,0.45) -- (8.5,4.25);
\foreach \y in {3.35,2.05,0.75}{
 \draw[draw=black,line width=0.6pt,-{Stealth[length=4pt,width=3pt]}]
 (1.0,\y) -- (14.8,\y);
}
\node[anchor=east] at (0.7,3.6) {$v'$};
\node[anchor=east] at (0.7,2.3) {$T_{\lambda_1}v'$};
\node[anchor=east] at (0.7,1.0) {$T_{\lambda_2}v'$};
\foreach \x in {1.8,5.8}{
 \path[draw=black,fill=black!10,line width=0.6pt]
 (\x-0.6,3.35) rectangle (\x+0.6,3.85);
 \node at (\x,3.6) {$ah$};
}
\path[draw=black,fill=black!4,line width=0.6pt] (7.2,3.35) rectangle (8.4,3.85);
\node at (7.8,3.6) {$v$};
\foreach \x in {3.8,7.8}{
 \path[draw=black,fill=black!10,line width=0.6pt]
 (\x-0.6,2.05) rectangle (\x+0.6,2.55);
 \node at (\x,2.3) {$ah$};
}
\path[draw=black,fill=black!4,line width=0.6pt] (9.2,2.05) rectangle (10.4,2.55);
\node at (9.8,2.3) {$v$};
\foreach \x in {7.8,11.8}{
 \path[draw=black,fill=black!10,line width=0.6pt]
 (\x-0.6,0.75) rectangle (\x+0.6,1.25);
 \node at (\x,1.0) {$ah$};
}
\path[draw=black,fill=black!4,line width=0.6pt] (13.2,0.75) rectangle (14.4,1.25);
\node at (13.8,1.0) {$v$};
\node at (1.8,0.05) {$-\lambda_2$};
\node at (3.8,0.05) {$\lambda_1-\lambda_2$};
\node at (5.8,0.05) {$-\lambda_1$};
\node at (7.8,0.05) {$0$};
\node at (9.8,0.05) {$\lambda_1$};
\node at (11.8,0.05) {$\lambda_2-\lambda_1$};
\node at (13.8,0.05) {$\lambda_2$};
\end{tikzpicture}
\caption{Placement of the copies, shown for $N=2$. Each rectangle
represents a containing interval $I+c$ and is labelled by the function
copied there. Translation by $\lambda_i$ moves one copy of $ah$ onto
$I$. The other rectangles in that row form $w_i$. Their intervals
are disjoint from $I$ and from those forming the other remainder.
The centers are ordered as on the line, but their spacings are
schematic. The proof uses large $N$ to make the change to the
generator small.}
\label{fig:translate-placement}
\end{figure}
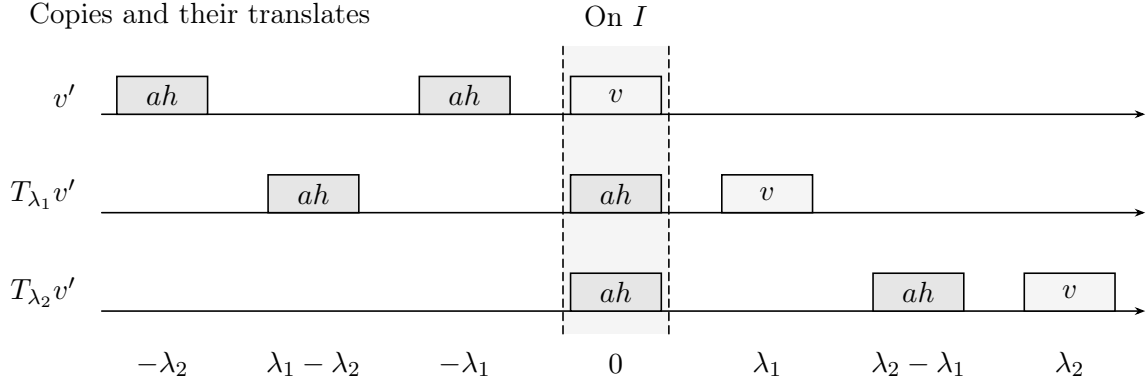

\section{Repeated approximation of one function}\label{tr:sec:branch}

We now repeatedly apply \Cref{tr:lem:packet} to approximate
a function $f$ of norm one. At each step, we apply the lemma
to the difference between $f$ and its current approximation,
normalized to have $L^p$-norm one. For the moment, we use abstract disjoint coordinates in place of the functions supported outside the old interval. This lets us estimate the partial sums independently of
where those supports are placed.

We set
\begin{equation}\label{tr:eq:parameters}
q=\frac{p}{p-1},\quad \beta=2^{-1/p},\quad r=2^{-1/q},
\end{equation}
so that $\beta^p=r^q=1/2$ and $\beta r=1/2$. We work in the space
$$
Z_f=\F f\oplus_p Y_f,\quad
Y_f=\left(\bigoplus_{n=0}^\infty\ell_p^{N_n}\right)_p.
$$
An element is $tf+\sum_n y_n$, with $y_n\in\ell_p^{N_n}$ and
$$
\left\|tf+\sum_n y_n\right\|^p
=|t|^p+\sum_n\norm{y_n}_p^p.
$$
Thus, each summand supplies a new finite set of coordinates,
disjoint from $f$ and from all earlier coordinates. The positive
integers $N_n$ may be chosen arbitrarily. We write $e_{n,i}$ for
the coordinate vectors in the $n$-th summand and define
\begin{align}
u_n&=N_n^{-1/p}\sum_{i=1}^{N_n}e_{n,i},&
m_n(y_n)&=N_n^{-1/q}\sum_{i=1}^{N_n}y_{n,i},
\label{tr:eq:mean}\\
h_0&=f,&h_{n+1}&=\beta(h_n-u_n),
\label{tr:eq:residual}\\
x_{n,i}&=e_{n,i}+N_n^{-1/q}h_n.
&&\label{tr:eq:branch-atoms}
\end{align}
Thus, one clearly has $\norm{u_n}_p=\norm{m_n}=m_n(u_n)=1$.
The vectors $h_n$ and $u_n$ have disjoint coordinates, so induction
gives $\norm{h_n}=1$. More explicitly, one has the identities
\begin{equation}\label{tr:eq:h-formula}
h_n=\beta^nf+H_n,\quad
H_n=-\sum_{l<n}\beta^{n-l}u_l,\quad
\norm{H_n}_p^p=1-2^{-n}\le1.
\end{equation}
We add the functions in the order
$$
x_{0,1},\ldots,x_{0,N_0},\quad
x_{1,1},\ldots,x_{1,N_1},\quad\ldots.
$$
The first family uses $h_0=f$. Each subsequent family uses the normalized error $h_n$ defined in \eqref{tr:eq:residual}.

The closed span of the $x_{n,i}$ contains every coordinate of $Z_f$. Indeed, \eqref{tr:eq:residual} gives
\begin{equation}\label{tr:eq:recovery}
h_n=\frac{1}{2}\sum_{i=1}^{N_n}N_n^{-1/p}x_{n,i}+r h_{n+1}.
\end{equation}
After $L$ repetitions, the remainder is $r^Lh_{n+L}$, whose norm
tends to zero and therefore every $h_n$, including $f=h_0$, belongs to
the closed span. The same holds for
$e_{n,i}=x_{n,i}-N_n^{-1/q}h_n$.

It remains to bound the partial sums. Let $z=tf+\sum_n y_n$,
and suppose that, after subtracting the terms corresponding
to the first $n$ families, the remainder has the form
$$
t_nh_n+y_n+\sum_{l>n}y_l.
$$
We choose the coefficient vector $c_n$ so that subtracting the $n$-th family leaves a multiple of $h_{n+1}$ together with the later coordinates. The required formulas are
\begin{equation}\label{tr:eq:coefficient-recursion}
t_0=t,\quad t_{n+1}=r\bigl(t_n-m_n(y_n)\bigr),\quad
c_n=y_n+\frac{1}{2}\bigl(t_n-m_n(y_n)\bigr)u_n.
\end{equation}
Each coefficient in these formulas depends linearly and continuously
on $t,y_0,\ldots,y_n$. They also give the identity
\begin{equation}\label{tr:eq:elimination-identity}
t_nh_n+y_n=\sum_{i=1}^{N_n}c_{n,i}x_{n,i}+t_{n+1}h_{n+1}.
\end{equation}
Indeed, if $b=t_n-m_n(y_n)$, the two terms on the right are
$y_n+(b/2)u_n+(m_n(y_n)+b/2)h_n$ and $(b/2)(h_n-u_n)$.
Their sum is the left-hand side.

The coefficient maps are biorthogonal to the $x_{n,i}$. To check
this, we first apply \eqref{tr:eq:coefficient-recursion} to $h_k$.
For $n<k$, its input in the $n$-th summand is
$-\beta^{k-n}u_n$, and the recursion gives $t_n=\beta^{k-n}$,
$c_n=0$, and $t_k=1$. For the input $x_{k,i}$, the earlier
coefficients are therefore zero, while at the $k$-th step
$t_k=m_k(e_{k,i})=N_k^{-1/q}$. Hence, $c_k=e_{k,i}$ and
$t_{k+1}=0$. Every later coefficient is zero as well.

Let $P_{n,k}$ retain the complete families numbered
$0,\ldots,n-1$, followed by the first $k$ vectors of family $n$, where $0\le k\le N_n$. We write $E_{n,k}$ for the projection onto the first $k$
coordinates of $\ell_p^{N_n}$ and set
$\theta=k/N_n$. Telescoping
\eqref{tr:eq:elimination-identity} therefore gives
\begin{align}
P_{n,k}z&=tf+\sum_{l<n}y_l+E_{n,k}c_n-(t_n-\ell)h_n,
\label{tr:eq:prefix-formula}\\
\ell&=m_n(E_{n,k}y_n)+\frac{\theta}{2}\bigl(t_n-m_n(y_n)\bigr),
&t_n&=r^nt-\sum_{l<n}r^{n-l}m_l(y_l).
\label{tr:eq:ell}
\end{align}

The complete families correspond to $k=0$ or $k=N_n$ and the
intermediate values of $k$ account for every possible partial sum.
We record separately the parts of $P_{n,k}$ acting between
$\F f$ and $Y_f$. This is needed when we combine the constructions for different functions. In particular, the bound from $Y_f$
back to $\F f$ must tend to zero with $n$.

\begin{lemma}\label{tr:lem:branch-bounds}
With respect to $Z_f=\F f\oplus_p Y_f$, write
$$
P_{n,k}=\begin{pmatrix}\alpha_{n,k}&A_{n,k}\\
B_{n,k}&D_{n,k}\end{pmatrix}.
$$
Here, $A_{n,k}$ maps $Y_f$ to $\F f$ and $B_{n,k}$ maps
$\F f$ to $Y_f$. Uniformly in all the integers $N_l$, we have
\begin{equation}\label{tr:eq:four-bounds}
\alpha_{n,k}=1-2^{-n}(1-\theta/2)\in[0,1],\quad
\norm{A_{n,k}}\le3\beta^n,\quad
\norm{B_{n,k}} \le\frac{1}{2},\quad \norm{D_{n,k}}\le6.
\end{equation}
\end{lemma}

\begin{proof}
For input $tf$, \eqref{tr:eq:ell} gives $t_n=r^nt$ and
$\ell=(\theta/2)r^nt$. By \eqref{tr:eq:h-formula}, the $f$
coordinate of \eqref{tr:eq:prefix-formula} is $\alpha_{n,k}t$,
and its $Y_f$ component is
$$
r^nt\left[\frac{1}{2}E_{n,k}u_n-(1-\theta/2)H_n\right].
$$
This proves the assertion about $\alpha_{n,k}$. The two terms in
brackets have disjoint supports. With $s=2^{-n}$, we therefore have
$$
\norm{B_{n,k}}^p
=s^{p-1}\bigl[\theta2^{-p}+(1-\theta/2)^p(1-s)\bigr].
$$
This expression is convex in $\theta$, so it suffices to check the
endpoints. At $\theta=0$, it is zero for $n=0$; for $n\ge1$,
$s\le1/2$ and $s^{p-1}(1-s)$ increases on $[0,1/2]$, with maximum
$2^{-p}$. At $\theta=1$, it is $2^{-p}s^{p-1}(2-s)$, which
increases on $[0,1]$ to $2^{-p}$. Thus $\norm{B_{n,k}}\le1/2$.

For input $y\in Y_f$, H\"older's inequality gives
$$
|t_n|\le
\left(\sum_{a=1}^n r^{aq}\right)^{1/q}
\left(\sum_{l<n}\norm{y_l}_p^p\right)^{1/p}
\le\norm{y}_p.
$$
Also, $|m_n(y_n)|\le\norm{y}_p$, $|\ell|\le2\norm{y}_p$, and
$\norm{E_{n,k}c_n}_p\le2\norm{y}_p$.
The $f$ coordinate is $-\beta^n(t_n-\ell)$, which proves the
bound on $A_{n,k}$. The $Y_f$ component is
$\sum_{l<n}y_l-t_nH_n+E_{n,k}c_n+\ell H_n$.
The sum of the norms of these four terms is at most $6\norm{y}_p$.
\end{proof}

\section{Combining the approximations}\label{tr:sec:assembly}

We next apply the construction to a sequence of functions $f_j$.
The new coordinate functions assigned to different $f_j$ will
have disjoint supports, whereas the supports of the $f_j$ need
not be disjoint. Each $f_j$ has support disjoint from the new
coordinate functions assigned to it, but its support may intersect
those of coordinate functions assigned to earlier targets. We will choose the $f_j$
to form an unconditional basis of a complementary space $X$.
A contractive projection onto the closed span of the disjoint
coordinates will give the decomposition used to estimate partial sums.
 We will choose the normalized basis $(f_j)$ so that its initial
projections $Q_m$ satisfy, for a fixed $D>1$,
\begin{equation}\label{tr:eq:multiplier}
\norm{Q_m}\le D.
\end{equation}
We will also require the coefficient estimate
\begin{equation}\label{tr:eq:lower-lp}
\left(\sum_j|t_j|^p\right)^{1/p}
\le D\norm{\sum_jt_jf_j}_p.
\end{equation}
\Cref{tr:lem:local-complement} supplies both estimates.

The decay of $A_{n,k}$ in \eqref{tr:eq:four-bounds} determines
the order in which we use the functions $f_j$. We make enough
steps for each one before starting the next, so that the
contributions from the added coordinates back into the earlier
directions $f_j$ have summable bounds. The following proposition
gives the resulting estimate.

\begin{proposition}\label{tr:prop:assembly}
Suppose $L^p(\R)=X\oplus Y$, where $Y$ is the range of a
contractive projection $P_Y$ and $X=\ker P_Y$.
Suppose $X$ has a normalized unconditional basis $(f_j)$ satisfying
\eqref{tr:eq:multiplier} and \eqref{tr:eq:lower-lp}. Let $Y$ be spanned by disjoint
functions $e_0$ and $e_{j,n,i}$ of unit norm, and suppose each $f_j$ has support
disjoint from all the functions $e_{j,n,i}$ assigned to it.
For each $j$, define $h_{j,n}$ and $x_{j,n,i}$ by
\eqref{tr:eq:residual} and \eqref{tr:eq:branch-atoms}.

 Fix $\tau>0$ and choose strictly increasing integers $L_j$, with
 $L_0=0$, such that $3\beta^{L_j}\le\tau2^{-j}$.
Begin with $e_0$. At round $j$, all earlier targets have completed
$L_{j-1}$ families. Introduce $f_j$ and construct its first
$L_{j-1}$ families. Then, for each level from $L_{j-1}$ to $L_j-1$,
append one family for each of $f_1,\ldots,f_j$, in that order.
Within each family, retain the coordinate order.
If every $f_j$ is introduced, the resulting sequence is a Schauder
basis of $L^p(\R)$ with basis constant at most
\begin{equation}\label{tr:eq:K}
K= 3D + 9+\tau.
 \end{equation}
\end{proposition}

\begin{proof}
We write $F=x+y$, where $x=(I-P_Y)F$ and $y=P_YF$. Then
$$
\norm{x}_p\le2\norm{F}_p,\quad \norm{y}_p\le\norm{F}_p.
$$
We expand $x=\sum_jt_jf_j$ and
$y=y_0e_0+\sum_jy_j$, with $y_j$ supported in the coordinates
assigned to $f_j$. At any partial sum, only finitely many of the
constructions have begun, and at most one stops inside a finite
family. We apply the appropriate operator $P_{n,k}$ to each
$(t_j,y_j)$ and retain $y_0e_0$.

The number of completed families is non-increasing in $j$. While a family associated with $f_j$, $j>1$, is being added,
at least one more family has been completed for $f_{j-1}$
than for $f_j$. Hence the coefficients $\alpha_j$ in
\eqref{tr:eq:four-bounds}, taken to be zero for unstarted targets,
are nonincreasing. Indeed, if target $j$ has completed $n$ families
and the current fraction is $\theta$, then
$$
\alpha_{j-1}\ge1-2^{-(n+1)}
\ge1-2^{-n}(1-\theta/2)=\alpha_j.
$$
The comparison with $\alpha_{j+1}$ follows from the ordering
of the numbers of completed families. A partial sum during one round is shown in
\Cref{fig:translate-order}. Abel summation now gives
$$
\sum_j\alpha_jt_jf_j
=\sum_m(\alpha_m-\alpha_{m+1})Q_mx.
$$
These nonnegative coefficients sum to $\alpha_1\le1$, so the
$X$ component arising from $x$ has norm at most $D\norm{x}_p$.
By disjointness, \eqref{tr:eq:four-bounds}, and
\eqref{tr:eq:lower-lp}, the $Y$ component arising from $x$
has norm at most
$\frac{1}{2}\norm{(t_j)}_{\ell_p}\le\frac{D}{2}\norm{x}_p$.
The $Y$ component arising from $y$ has norm at most $6\norm{y}_p$.

For the part mapping $Y$ to $X$, every construction except possibly
the most recently introduced one has completed at least $L_j$
families. By \eqref{tr:eq:four-bounds}, their contributions have
total norm at most
$$
\sum_j \tau2 ^{-j}\norm{y_j}_p\le \tau \norm{y}_p.
$$
The one possible exception contributes at most $3\norm{y}_p$.
Thus every partial sum $S$ satisfies
\begin{equation}\label{tr:eq:global-bound}
\norm{SF}_p \le\frac{3}{2}D \norm{x}_p+ (9+\tau) \norm{y}_p\le K\norm{F}_p.
\end{equation}

The decomposition and \eqref{tr:eq:coefficient-recursion} give
bounded coefficient functionals. They are biorthogonal within
each construction by the calculation in \Cref{tr:sec:branch},
and between different constructions by the decomposition $X\oplus Y$.
Equation \eqref{tr:eq:recovery} shows that the closed span contains
every $f_j$ and $e_{j,n,i}$, as well as $e_0$. It is therefore
all of $L^p(\R)$. The uniformly bounded partial sums fix each
finite linear combination once its terms have appeared, so density
gives norm convergence for every $F$. Biorthogonality gives
uniqueness.
\end{proof}

\begin{figure}[htbp]
\color{black}
\centering
\begin{tikzpicture}[x=1cm,y=1cm,font=\small,line cap=round,
 every node/.style={text=black,inner sep=1pt}]
\node at (3.65,4.2) {First $n$ families};
\node at (7.1,4.2) {Next family};
\node[anchor=west] at (9.25,4.2) {Coefficient on $f_i$};
\foreach \y/\i in {3.3/1,2.35/2,1.4/3,0.45/4}{
 \node[anchor=east] at (0.9,\y) {$f_{\i}$};
 \path[draw=black,fill=black!10,line width=0.6pt]
 (1.25,\y-0.25) rectangle (6.05,\y+0.25);
 \path[draw=black,fill=black!4,line width=0.6pt]
 (6.3,\y-0.25) rectangle (7.9,\y+0.25);
}
\foreach \y in {3.3,2.35}{
 \path[draw=black,fill=black!18,line width=0.6pt]
 (6.3,\y-0.25) rectangle (7.9,\y+0.25);
}
\path[draw=black,fill=black!18,line width=0.6pt] (6.3,1.15) rectangle (7.26,1.65);
\draw[draw=black,line width=0.6pt,densely dashed] (7.26,0.95) -- (7.26,1.9);
\node at (6.78,1.4) {$\theta$};
\node[font=\scriptsize] at (7.1,0.45) {not begun};
\node[anchor=west] at (9.25,3.3) {$\alpha_1=1-2^{-(n+1)}$};
\node[anchor=west] at (9.25,2.35) {$\alpha_2=1-2^{-(n+1)}$};
\node[anchor=west] at (9.25,1.4) {$\alpha_3=1-2^{-n}(1-\theta/2)$};
\node[anchor=west] at (9.25,0.45) {$\alpha_4=1-2^{-n}$};
\draw[draw=black,line width=0.6pt,-{Stealth[length=4pt,width=3pt]}]
 (8.55,3.5) -- (8.55,0.25);
\node at (4.55,-0.25) {Order within one level of round $4$};
\node[anchor=west] at (9.25,-0.25) {$\alpha_1\ge\alpha_2\ge\alpha_3\ge\alpha_4$};
\end{tikzpicture}
\caption{The order of the families. At the illustrated partial sum,
each of $f_1,\ldots,f_4$ has completed $n$ families. At the next
level, the families for $f_1$ and $f_2$ are complete, a fraction
$\theta$ of the family for $f_3$ has appeared, and the family for
$f_4$ has not begun. Thus the coefficients $\alpha_i$ are
nonincreasing. This is the order needed for the nonnegative
coefficients in Abel summation.}
\label{fig:translate-order}
\end{figure}
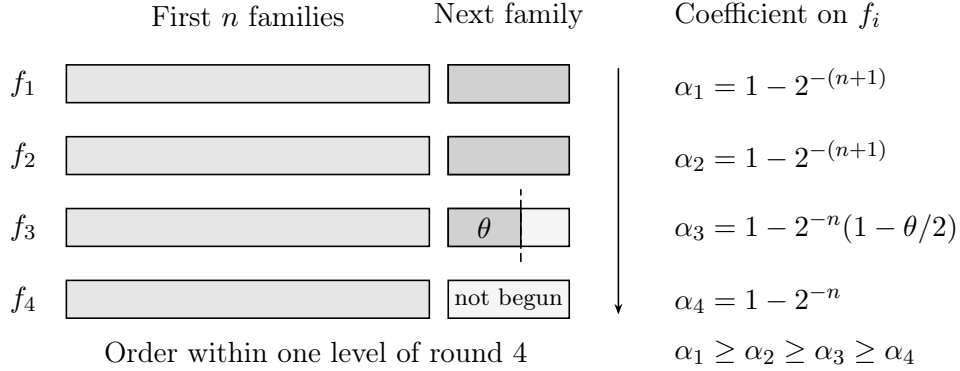

The support assumption in the proposition is needed to make
$\norm{h_{j,n}}_p=1$ throughout the construction for $f_j$.
No disjointness is required between $f_j$ and the coordinates
assigned to other targets. Their interaction is handled by $P_Y$.

\section{Constructing the functions on the line}\label{tr:sec:physical}

We now construct the decomposition required by
\Cref{tr:prop:assembly} together with the generators $G_t$.
Whenever a new function of norm one $e$ appears in
\eqref{tr:eq:exact-packet}, we reserve its one-dimensional span
for $Y$. On a reserved finite union of open intervals $A$ containing
its support, we choose a basis of $\ker\varphi_e\subset L^p(A)$,
where $\varphi_e$ is the norming functional defined in
\Cref{tr:lem:local-complement}. These functions belong to $X$ and will be selected as later
targets $f_j$. Thus, the functions to be approximated are chosen
as the new supports appear. We first establish the required estimates for the Haar basis.
We then use them in \Cref{tr:lem:local-complement} to construct
the complementary bases with bounds independent of $e$.

\begin{lemma} \label{tr:lem:haar}
Let $p\ge2$. The normalized dyadic Haar basis of $L^p(0,1)$,
ordered with the constant first and every parent before its
children, has contractive initial projections and satisfies
\begin{equation}\label{tr:eq:haar-lower}
\left(\sum_j|t_j|^p\right)^{1/p}
\le\norm{\sum_jt_jh_j}_p.
\end{equation}
The same assertions hold for the Haar basis of the mean-zero
subspace, with the constant omitted. They are preserved by
isometries, by disjoint $L^p$ sums, and by interleavings that
preserve the order within each summand.
\end{lemma}

\begin{proof}
A finite initial Haar segment consists of the constant and the
differences obtained by successively splitting atoms of a finite
dyadic partition. Its projection is conditional expectation
onto that partition, hence is contractive. On the mean-zero
subspace the constant coefficient is zero, so the same statement
holds after omitting the constant.

For complex numbers $a,b$, and therefore also for real numbers,
\begin{equation}\label{tr:eq:clarkson}
\frac{|a+b|^p+|a-b|^p}{2}
\ge (|a|^2+|b|^2)^{p/2}
\ge |a|^p+|b|^p.
\end{equation}
Indeed, the average of $|a+b|^2$ and $|a-b|^2$ is
$|a|^2+|b|^2$. The first inequality follows from the convexity of
$u\mapsto u^{p/2}$, and the second from $p/2\ge1$. Just before a Haar function is added, the preceding partial sum is constant on its support. The new Haar function takes opposite
values of equal modulus on the two halves of that support.
Integrating~\eqref{tr:eq:clarkson} there, and using its
normalization, shows that adding $t_jh_j$ increases the $p$th
powerr of the norm by at least $|t_j|^p$.
Induction proves~\eqref{tr:eq:haar-lower} for finite sums.
The infinite version follows from convergence of the partial sums.

An initial projection in an order-preserving interleaving is the
direct sum of initial projections on its summands, some possibly
zero. Its norm is at most one. Summing
\eqref{tr:eq:haar-lower} to the $p$-th power over the disjoint
summands gives the remaining assertion.
\end{proof}

\begin{lemma} \label{tr:lem:local-complement}
Let $A$ be a bounded finite union of open intervals, let
$e\in C_c^\infty(A)$ be real with $\norm{e}_p=1$, and let $D>1$.
Define
$$
\varphi_e(F)=\int_A F e|e|^{p-2}.
$$
Then $\norm{\varphi_e}=\varphi_e(e)=1$, and
$\ker\varphi_e\subset L^p(A)$ has a normalized unconditional
basis of real functions in $C_c^\infty(A)$ satisfying
\eqref{tr:eq:multiplier} and \eqref{tr:eq:lower-lp}.
The same conclusion holds for $L^p(A)$ itself. The unconditional bound may be chosen independently of $A$ and $e$.
\end{lemma}

\begin{proof}
 In order to prove the lemma, we put $E=\{e\ne0\}$ and $Z=A\setminus E$. On $E$, the cumulative
distribution
 $$
C(x)=\int _{E\cap(-\infty,x]} |e(t)|^p\,dt
$$
 identifies the nonatomic probability measure $|e|^p\,dx$ with
 Lebesgue measure on $(0,1)$, up to null sets. Consequently
 $V\psi=e\,\psi\circ C$ is an onto isometry from $L^p(0,1)$
to $L^p(E)$, and $\varphi_e(V\psi)=\int_0^1\psi$.
 Thus
$$
\ker\varphi_e
=(\ker\varphi_e\cap L^p(E))\oplus_p L^p(Z)
$$
has a normalized real Haar basis with both bounds in
\Cref{tr:lem:haar}. Here the second summand, if nonzero, is
 identified isometrically with $L^p(0,1)$ using normalized Lebesgue
measure. Fix an interleaving preserving the order in each summand.
 H\"older's inequality gives
$\norm{\varphi_e}=\varphi_e(e)=1$.

Smooth compactly supported functions are dense in the kernel. Indeed,
for $F\in\ker\varphi_e$ and $\psi\in C_c^\infty(A)$,
$$
\psi-\varphi_e(\psi)e\in C_c^\infty(A)\cap\ker\varphi_e,
\quad
\norm{\psi-\varphi_e(\psi)e-F}_p\le2\norm{\psi-F}_p.
$$
Let $(u_n)$ be the normalized Haar basis just obtained. Choose
normalized real smooth functions $v_n$ in the kernel so close that
$$
\sum_n\norm{u_n^*}\,\norm{v_n-u_n}_p<\delta<1.
$$
The operator $Uu_n=v_n$ then satisfies $\norm{U-I}<\delta$.
Its inverse has norm at most $(1-\delta)^{-1}$, and the initial
projections of $(v_n)$ have norms at most
$(1+\delta)/(1-\delta)$. Moreover,
$$
\left(\sum_n|t_n|^p\right)^{1/p}
\le\norm{\sum_nt_nu_n}_p
\le(1-\delta)^{-1}\norm{\sum_nt_nv_n}_p.
$$
Choosing $(1+\delta)/(1-\delta)\le D$ proves the two estimates.
The unconditional bound is at most $D$ times the Haar multiplier
bound. This comparison is on the whole kernel; the approximations
need not preserve the separate supports $E$ and $Z$.
For the full space, start with a Haar basis of $L^p(A)$ and use
the density of $C_c^\infty(A)$ in the same way.
\end{proof}

For $F\in L^p(A)$, the decomposition supplied by the lemma is
$$
F=\varphi_e(F)e+\bigl(F-\varphi_e(F)e\bigr).
$$
The first term is a multiple of $e$, and the second belongs to
$\ker\varphi_e$. Since $\norm{\varphi_e}=\norm{e}_p=1$, the
map taking $F$ to its first component is contractive.

Fix $\varepsilon>0$. Choose $D>1$ sufficiently close to one
and $\tau>0$ and $0<\rho<1$ sufficiently small so that,
with $K=3D+9+\tau$,,
\begin{equation}\label{tr:eq:final-constant}
K\frac{1+\rho}{1-\rho}<12+\varepsilon.
\end{equation}
We aim to use \Cref{tr:lem:local-complement} with this $D$ on every new
region. The reference functions will be compactly supported and
 smooth. We construct a separate sequence of entire functions and
compare its translates with the exact reference vectors.

\subsection{Gaussian smoothing}

For $\sigma>0$, let
 $$
 \phi_\sigma(z) = \frac{1}{\sqrt{\pi}\sigma} e ^{-z^2/\sigma^2},
\quad
 H_\sigma(z) = \int_\R h(u)\phi_\sigma(z-u)\,du.
 $$
 Since $h$ is compactly supported, $H_\sigma$ is entire,
real-valued on $\R$, and
 $\norm{H_\sigma-h}_p\to0$ as $\sigma\downarrow0$.
It and all its derivatives are rapidly decreasing on the real line. For every fixed $r$,
 \begin{equation}\label{tr:analytic:eq:complex-tail}
\sup_{|z|\le r}|H_\sigma(z+\lambda)|\longrightarrow0
\quad(\lambda\to+\infty).
\end{equation}
For example, if $\supp h\subset[-R_h,R_h]$ and
$\lambda>R_h+r$, the left side is bounded by
$$
\frac{\norm{h}_1}{\sqrt{\pi}\sigma}
\exp\!\left(\frac{r^2-(\lambda-r-R_h)^2}{\sigma^2}\right).
$$

The analytic increment associated with the finite family is
\begin{equation}\label{tr:analytic:eq:analytic-increment}
\Delta G=a\sum_{i=1}^N T_{-\lambda_i}H_\sigma
=\left(a\sum_{i=1}^N T_{-\lambda_i}h\right)*\phi_\sigma
\quad\text{on }\R.
\end{equation}
The same Gaussian is used for every term in this family.
 By Young's inequality and the disjoint supports of
$T_{-\lambda_i}h^{(k)}$,
 \begin{align}
\norm{(\Delta G)^{(k)}}_s
&\le dN^{1/s-1/q}\norm{h^{(k)}}_s,
&&1\le s<\infty,\label{tr:analytic:eq:derivative-s}\\
\norm{(\Delta G)^{(k)}}_\infty
&\le dN^{-1/q}\norm{h^{(k)}}_\infty.\label{tr:analytic:eq:derivative-infty}
\end{align}
These estimates are independent of $\sigma$.
In particular,
 \begin{equation}\label{tr:analytic:eq:increment-p}
\norm{\Delta G}_p\le dN^{2/p-1}.
\end{equation}
The exponent in \eqref{tr:analytic:eq:derivative-s} is negative
precisely when $s>q$. Since the estimates are independent of
$\sigma$, we may first choose $N$ large enough to satisfy any
prescribed finite collection of bounds on these derivative norms,
and then choose $\sigma$ to obtain the required approximation.

\subsection{Choosing the stages}

We start by putting
$$
0<\kappa<\min\{1/100,\rho/(12K)\},\quad s_j=q+1/j\quad(j\ge1).
$$
Choose a real normalized $e_0\in C_c^\infty((0,1))$.
Set $A_0=(0,1)$ and $I_0=(-1,1)$. Apply
\Cref{tr:lem:local-complement} to choose a basis of
$\ker\varphi_{e_0}\subset L^p(A_0)$ and a basis of $L^p(-1,0)$,
in each case consisting of smooth compactly supported functions.
 Take an entire Gaussian smoothing $G_0$ of $e_0$ with
\begin{equation}\label{tr:analytic:eq:initial}
\norm{G_0-e_0}_p<\kappa.
\end{equation}
The initial reference vector is $b_0=e_0$, with parameter zero.

 At stage $t\ge1$, the order in \Cref{tr:prop:assembly} selects a
known smooth compactly supported residual $h=h_{j,n}$ of
 norm one. Let
$$
M_{t-1}=1+\sum_{r<t}N_r
$$
be the number of existing reference vectors. The entire function
 $G_{t-1}$ is also already known. Set
\begin{equation}\label{tr:analytic:eq:budgets}
\eta_t=\frac{\kappa2^{-t}}{M_{t-1}}.
\end{equation}
Make the choices in the following order.

\begin{enumerate}
\item Choose $N=N_t$ so large that $1-(N-1)N^{-(p-1)}\ge1/2$ and
\begin{equation}\label{tr:analytic:eq:Np}
4N^{2/p-1}\le\eta_t.
\end{equation}
Increase $N$, if needed, to ensure that for $0\le k\le t$
and $1\le j\le t$,
\begin{equation}\label{tr:analytic:eq:Ns}
4N^{1/s_j-1/q}\norm{h^{(k)}}_{s_j}\le2^{-t},
\quad
4N^{-1/q}\norm{h^{(k)}}_\infty\le2^{-t}.
\end{equation}
Every exponent here is negative, and all the norms of $h$
are finite. These finitely many inequalities therefore hold
for all sufficiently large $N$.

\item Set
$$
\rho_t=\frac{\kappa2^{-t}}{N_t}.
$$
Choose a real $v=v_t\in C_c^\infty(\R)$ satisfying
\begin{equation}\label{tr:analytic:eq:v-approx}
\norm{v_t-G_{t-1}}_p<\rho_t/2.
\end{equation}
Enlarge the current open interval $I=(-R,R)$ to contain its support,
 the support of $h$, and all previously assigned regions,
compactly in its interior. Assign smooth full bases to the
 newly included gaps. These are simply additional targets
for later stages.

\item Calculate $d_t,a_t$ from~\eqref{tr:eq:packet-scales}.
 The induction gives
 $$
 1-2\kappa\le\norm{G_{t-1}}_p\le1 + 2\kappa,
 $$
 because $\norm{G_r-G_{r-1}}_p\le\eta_r$ and
$\sum_r\eta_r\le\kappa$.
 Together with~\eqref{tr:analytic:eq:v-approx}, this implies
$$
\frac{1}{2}\le\norm{v_t}_p\le2,\quad
\frac{1}{2}\le d_t<4.
$$
Consequently the uniform factor $4$ used when choosing
$N_t$ is justified.

\item Choose a Gaussian smoothing $H_t=h*\phi_{\sigma_t}$
so accurate that
\begin{equation}\label{tr:analytic:eq:H-approx}
\norm{H_t-h}_p<\frac{\rho_t}{2a_tN_t}.
\end{equation}
Making $\sigma_t$ smaller does not affect the estimates~\eqref{tr:analytic:eq:derivative-s}--\eqref{tr:analytic:eq:derivative-infty}.

\item Finally choose $\lambda_{t,1},\ldots,\lambda_{t,N_t}$
in $\Gamma$, beyond all old parameters, to satisfy the
separation in \Cref{tr:lem:placement} and
\begin{equation} \label{tr:analytic:eq:disk}
a _t \sum_{i=1}^{N_t}
\sup_{|z|\le t}|H_t(z+\lambda_{t,i})|\le2^{-t}.
\end{equation}
This is possible by~\eqref{tr:analytic:eq:complex-tail} and the
unboundedness of $\Gamma$.
\end{enumerate}

The calculation in \Cref{tr:lem:packet} applies with the chosen
$v_t,h,N_t$ and parameters. Define the compactly supported reference function
and the entire generator by
\begin{align}
V_t&=v_t+a_t\sum_{i=1}^{N_t}T_{-\lambda_{t,i}}h,
\label{tr:analytic:eq:Vt}\\
G_t&=G_{t-1}+a_t\sum_{i=1}^{N_t}T_{-\lambda_{t,i}}H_t.
\label{tr:analytic:eq:Gt}
\end{align}
For each new coordinate reserve the whole set
\begin{equation}\label{tr:eq:regions}
A_{t,i}=(I+\lambda_{t,i})\cup
\bigcup_{l\ne i}(I+\lambda_{t,i}-\lambda_{t,l}).
\end{equation}
These are pairwise disjoint finite unions of open intervals, outside
$I$, and $e_{j,n,i}\in C_c^\infty(A_{t,i})$.
Append the reference vectors
\begin{equation}\label{tr:eq:reference-vectors}
b_{t,i}=T_{\lambda_{t,i}} V _t
=d_t\bigl(e_{j,n,i}+N _t ^{-1/q}h_{j,n}\bigr).
\end{equation}
 Set $h_{j,n+1}=\beta(h_{j,n}-u_{j,n})$.
 Enlarge the covered interval to $I_t=(-R_t,R_t)$, containing the closures of all
 new regions~\eqref{tr:eq:regions} and $\supp V_t$, compactly
in its interior, and make its radius at least $t$.
 Reserve those new regions, assign their smooth complement bases,
and fill the remaining new gaps with smooth full bases.

 The estimates above give, simultaneously,
 \begin{align}
\delta_t:=\norm{G_t-G_{t-1}}_p&\le\eta_t,
\label{tr:analytic:eq:delta}\\
\sigma_t':=\norm{G_t-V_t}_p&<\rho_t,
\label{tr:analytic:eq:sigma}\\
\sup_{|z|\le t}|G_t(z)-G_{t-1}(z)|&\le2^{-t}.
\label{tr:analytic:eq:entire-budget}
\end{align}
For~\eqref{tr:analytic:eq:sigma}, use~\eqref{tr:analytic:eq:v-approx},
\eqref{tr:analytic:eq:H-approx}, and the triangle inequality:
$$
\norm{G_t-V_t}_p
\le\norm{G_{t-1}-v_t}_p
+a_tN_t\norm{H_t-h}_p<\rho_t.
$$

We organize the steps into rounds. Each round introduces one
new target and advances the earlier targets as in \Cref{tr:prop:assembly}. To decide
which function is introduced next, we label the assigned sets
by consecutive positive integers and enumerate the basis on each
set. The pair $(a,b)$ specifies the $b$-th basis function on
the $a$-th set. We order these pairs first by $a+b$ and then
by $a$.

At round $j$, we choose the first available unused pair and call
its function $f_j$. We construct its first $L_{j-1}$ families.
We then add one family for each of $f_1,\ldots,f_j$, in this
order, and repeat until each function has $L_j$ completed families. Every function used
in this round was available before its first approximation step.

Every round has finitely many steps. Every available function is
eventually chosen, since only finitely many pairs precede its
pair, including pairs belonging to sets not yet assigned.
By the end of round $j$, we have constructed $L_j$ families
for $f_j$. Infinitely many further families are constructed
for $f_j$ in subsequent rounds. This is precisely the order required in
\Cref{tr:prop:assembly}. At each step, the residual and the count
$M_{t-1}$ are known before $N_t$ is chosen.

Since $R_t\ge t$, the assigned sets eventually partition all of
$\R$, up to endpoints.
Let $\mathcal A$ be the collection of sets assigned to the
pairwise disjoint functions $e_A$ of norm one, including the set
assigned to $e_0$, and define
\begin{equation}\label{tr:eq:PY}
P_YF=\sum_{A\in\mathcal A}\varphi_{e_A}(F\1_A)e_A.
\end{equation}
Disjointness and H\"older's inequality give
$$
\norm{P_YF}_p^p
=\sum_{A\in\mathcal A}|\varphi_{e_A}(F\1_A)|^p
\le\sum_{A\in\mathcal A}\norm{F\1_A}_p^p\le\norm{F}_p^p.
$$
Thus, the series converges in $L^p$, and $P_Y$ is a contractive
projection onto the closed span $Y$ of these coordinates.
Its kernel $X$ is the $\ell_p$ sum of the spaces
$\ker\varphi_{e_A}$ and of the full $L^p$ spaces on the remaining
intervals. The available bases, taken together in the order
$(f_j)$, therefore form a normalized unconditional basis of $X$
with a common unconditional bound. Since the order preserves each
local order, its initial projections satisfy \eqref{tr:eq:multiplier}.
Summing the local coefficient estimates to the $p$-th power gives
\eqref{tr:eq:lower-lp}. Indeed, the multiplier bound follows by
summing its $p$-th powers on the disjoint sets, and density follows
by first retaining finitely many sets and then approximating
within each of them.

All the hypotheses of \Cref{tr:prop:assembly} now hold. Hence,
$e_0$ and the $x_{j,n,i}$ form a Schauder basis with constant
at most $K$. Replacing each $x_{j,n,i}$ by $d_t x_{j,n,i}$
does not change any partial sum projection. Consequently,
 $b_0=e_0$ and the vectors in \eqref{tr:eq:reference-vectors}
also form a basis with constant at most $K$.

\section{Passing to a single generator}\label{tr:sec:limit}

Equation~\eqref{tr:analytic:eq:entire-budget} gives locally uniform convergence of $G_t$ on $\C$ to an entire function $G$. Equation~\eqref{tr:analytic:eq:delta} gives convergence in $L^p(\R)$ to a function $g$. These limits agree on the real line: on every bounded interval, local uniform convergence also implies $L^p$ convergence. Thus $g=G|_\R$, with its continuous representative.

Enumerate all reference vectors as $(b_m)_{m\ge0}$ with their
chosen parameters, including $\lambda_0=0$.
For a vector introduced at stage $t$,
 $$
\norm{T_{\lambda_{t,i}}g-b_{t,i}}_p
\le \sigma_t'+\sum_{r>t}\delta_r.
$$
The initial vector contributes at most
$\norm{G_0-e_0}_p+\sum_{r\ge1}\delta_r$.
 Summing these nonnegative estimates and counting which
vectors precede each later increment gives
 \begin{align}
\sum _{m\ge0} \norm{T_{\lambda_m}g-b_m}_p
&\le \norm{G_0-e_0}_p+
 \sum _{t\ge1}N _t \sigma_t'
+ \sum _{t\ge1} M_{t-1}\delta_t \notag\\
&< \kappa +
\sum_{t\ge1}\kappa2^{-t}
+\sum_{t\ge1}\kappa2^{-t}
=3\kappa.\label{tr:eq:total-error}
\end{align}
The factor $N_t$ accounts for the approximation at stage $t$,
and $M_{t-1}$ accounts for the effect of the change at that stage
on every earlier vector. The choices of $\rho_t$ and $\eta_t$ give
$N_t\sigma_t'<\kappa2^{-t}$ and
$M_{t-1}\delta_t\le\kappa2^{-t}$.

 The reference basis has constant at most $K$. Its vectors satisfy
$\norm{b_0}_p=1$ and $\norm{b_{t,i}}_p=\norm{V_t}_p\ge1/2$,
 by disjointness and the bound on $\norm{v_t}_p$.
The difference of consecutive partial sums therefore gives
 $\norm{b_m^*}\le4K$. It follows that
 $$
 Ef =\sum _{m\ge0} b_m^*( f )
 \bigl( T_{\lambda_m}g-b_m \bigr)
 $$
defines a bounded operator with
 $$
\norm{E}<12K\kappa<\rho<1.
$$
Thus $I+E$ is invertible and sends every $b_m$ to
$T_{\lambda_m}g$. Consequently, the basis $(T_{\lambda_m}g)$
has basis constant at most $K(1+\rho)/(1-\rho)<12+\varepsilon$
by \eqref{tr:eq:final-constant}.
 All functions used are real on $\R$, so the same construction
works over either scalar field. Translation by $b$ and undoing
 any initial reflection restore the original set $\Lambda_0$
 and preserve the existence of an entire extension.

\begin{remark}\label{rem:translate-basis-constant}
For every $2<p<\infty$, $\varepsilon>0$, and unbounded set
$S\subset\R$, the preceding construction gives a real-valued
function $g\in L^p(\R)$ extending to an entire function and
distinct parameters $(\lambda_n)\subset S$ such that
$(T_{\lambda_n}g)$ is a Schauder basis of $L^p(\R)$ with
$$
K((T_{\lambda_n}g))<12+\varepsilon.
$$
Indeed, the bound $K=3D+9+\tau$ in \eqref{tr:eq:K}, together
with \eqref{tr:eq:final-constant}, gives this estimate by taking
$D>1$ sufficiently close to one and $\tau,\rho>0$ sufficiently
small. No non-negativity condition is imposed on $g$ here.
\end{remark}

\subsection{The derivatives of the generator}

Fix $k\ge0$ and $j\ge1$. For
$t\ge\max\{k,j,1\}$, equations~\eqref{tr:analytic:eq:Ns} and
\eqref{tr:analytic:eq:derivative-s}--\eqref{tr:analytic:eq:derivative-infty} give
$$
\norm{(G_t-G_{t-1})^{(k)}}_{s_j}\le2^{-t},
\quad
\norm{(G_t-G_{t-1})^{(k)}}_\infty\le2^{-t}.
$$
Every finite-stage derivative is in $C_0(\R)$ and in every
$L^s(\R)$. Hence the derivatives converge uniformly on
$\R$, and in $L^{s_j}(\R)$.
Locally uniform convergence of the entire functions also gives
locally uniform convergence of their complex derivatives, so
these limits are exactly $G^{(k)}|_\R=g^{(k)}$.
It follows that $g^{(k)}\in C_0(\R)\cap L^{s_j}(\R)$
for all $j$.
For any $s>q$, choose $j$ with $s_j<s$ and use
$$
\int_\R |g^{(k)}|^s
\le \norm{g^{(k)}}_\infty^{\,s-s_j}
\int_\R |g^{(k)}|^{s_j}<\infty.
$$
This proves \Cref{thm:translates} for $p>2$, with a generator
that extends to an entire function and translation parameters
in the prescribed unbounded set.

\subsection{A non-negative generator}\label{tr:ss:positive}

We now prove part (ii) of \Cref{thm:translate-properties}. We use
non-negative functions as targets, and replace the signed residual
by the positive average of the new coordinates. The local complements
must also have non-negative bases. We first record the coefficient
estimate needed for these bases.

\begin{lemma}\label{tr:lem:positive-lower}
For $2\le p<\infty$, there is a normalized non-negative Schauder
basis $(r_n)_{n\ge0}$ of $L^p(\R)$, with $r_0=\1_{[0,1)}$, and
a constant $C_p<\infty$ such that
\begin{equation}\label{tr:eq:positive-lower}
\left(\sum_n|a_n|^p\right)^{1/p}
\le C_p\norm{\sum_n a_nr_n}_p
\end{equation}
for every finitely supported scalar sequence.
\end{lemma}

\begin{proof}
Use the construction in \Cref{sec:assembly} with fixed
$\delta,\eta,\rho$. We first estimate its uncorrected vectors
$(v_n)$. Let $E$ average on the integer unit intervals. For a
finite linear combination $F$ of these vectors, put $G=EF$ and
$H=F-G$. Write the contribution of each pair in \eqref{eq:pairs} to $F$
as $a_js_j+b_j(\kappa^{-1}s_j+\varphi_j)$.
Here we suppress the dependence of $\kappa$ on the stage. The coefficient on $s_j$ in
$G$ is $a_j+\kappa^{-1}b_j$, while
$H=\sum_jb_j\varphi_j$. The remaining columns and the initial
indicator contribute only to $G$.

For real or complex scalars $a,b$, there is $c_p>0$ such that
$$
|a+b|^p\ge |a|^p+p|a|^{p-2}\operatorname{Re}(\overline a b)
+c_p|b|^p.
$$
Indeed, the second derivative of $|a+tb|^p$ is at least
$p|a+tb|^{p-2}|b|^2$. If $b\ne0$, on a subset of $[0,1/2]$
of length at least $1/4$ we have $|a/b+t|\ge1/8$.
Taylor's formula therefore gives $c_p=p8^{1-p}$.
Apply this inequality in the actual order of the unequal
splittings, inserting zero coefficients for omitted differences.
Before each splitting, the preceding sum is constant on $D_j$,
and $\int_{D_j}\varphi_j=0$. The linear term integrates to zero,
so telescoping gives
$$
\sum_j|b_j|^p\norm{\varphi_j}_p^p\le c_p^{-1}\norm H_p^p.
$$

The auxiliary blocks occupy disjoint unit intervals.
By \Cref{lem:inverse} and $\norm{s_j}_p\le M$, the weighted
$\ell_p$ norm of their combined coefficients in $G$, including
the initial coefficient, is at most $2M\norm G_p$.
Moreover, the identities in  \eqref{eq:haar-identities} and the estimate following equation \eqref{eq:a} imply
$\kappa^{-1}\norm{s_j}_p\le\eta\norm{\varphi_j}_p$.
Since $s_j$ and $\varphi_j$ have disjoint supports, the triangle
inequality in the coefficient space gives, for $F=\sum_n t_nv_n$,
\begin{align*}
\left(\sum_n|t_n|^p\norm{v_n}_p^p\right)^{1/p}
&\le 2M\norm G_p
+(1+2\eta^p)^{1/p}c_p^{-1/p}\norm H_p\\
&\le \bigl(2M+2(1+2\eta^p)^{1/p}c_p^{-1/p}\bigr)\norm F_p.
\end{align*}
For the first inequality, the difference on a paired column has
weighted norm $|b_j|(2\kappa^{-p}\norm{s_j}_p^p+
\norm{\varphi_j}_p^p)^{1/p}$. The last inequality uses
$\norm E\le1$ and $\norm{I-E}\le2$.
Apply the bounded invertible operator $U^*$ from
\eqref{eq:final-basis} and normalize its images. The estimate
changes by at most the factor $\norm{U^*}\norm{(U^*)^{-1}}$.
These images are non-negative, and the initial indicator is fixed
by \eqref{eq:global-identities}. This proves the lemma.
\end{proof}

\begin{lemma}\label{tr:lem:positive-complement}
Fix $2\le p<\infty$. There are constants $B,D<\infty$ such that, for every
bounded finite union of open intervals $A$ and every non-negative
function of norm one $e\in C_c^\infty(A)$, the space $L^p(A)$ has a
normalized non-negative basis $e,f_1,f_2,\ldots$ in $C_c^\infty(A)$
with basis constant at most $B$ and lower $\ell_p$ coefficient
constant at most $D$. If $\phi$ is its first coordinate functional,
then $\norm\phi\le B$, and $(f_j)_{j\ge1}$ is a basis of
$\ker\phi$ with the same bounds. The full space $L^p(A)$ has
such a basis without specifying its first vector as well.
\end{lemma}

\begin{proof}
Put $E_e=\{e>0\}$ and $Z=A\setminus E_e$. The latter set has
positive measure since $e$ has compact support in $A$.
On $E_e$, the map
$$
\psi\longmapsto e(x)\psi\left(\int_{E_e\cap(-\infty,x]}e(t)^p\,dt\right)
$$
is a lattice isometry from $L^p(0,1)$ onto $L^p(E_e)$, as in
\Cref{tr:lem:local-complement}. Combine it with a lattice isometry
from $L^p(\R\setminus[0,1))$ onto $L^p(Z)$. Such an isometry
exists since both measure spaces are separable and nonatomic.
This gives a lattice isometry from $L^p(\R)$ onto $L^p(A)$
sending $r_0$ in \Cref{tr:lem:positive-lower} to $e$.

Transport that basis and keep its first vector fixed. Approximate
each remaining vector by a normalized non-negative function in
$C_c^\infty(A)$, using non-negative cutoffs and mollification.
Choose the errors so that their products with the original
coefficient-functional norms have sum less than $1/4$.
The resulting change is an invertible operator whose distance
from the identity is less than $1/4$. Thus both bounds remain
uniform in $A,e$. The first coordinate projection has norm at
most $B$. On its kernel, the remaining initial projections are
restrictions of the full initial projections, and the coefficient
estimate also restricts. The same argument without fixing the
first vector proves the full-space assertion.
\end{proof}

\begin{proof}[Proof of \Cref{thm:translate-properties}(ii)]
Fix $p>2$, and enlarge $D$ so that $D\ge B$.
We first describe the replacement for \Cref{tr:sec:branch}.
For a non-negative function $f$ of norm one, let $e_{n,i}$ be non-negative
function of norm ones on disjoint supports, also disjoint from $f$.
In the notation of that section, set
$$
u_n=N_n^{-1/p}\sum_i e_{n,i},\quad
h_0=f,\quad h_{n+1}=u_n,\quad
x_{n,i}=e_{n,i}+2N_n^{-1/q}h_n.
$$
These functions are non-negative. They form a basis of
$\F f\oplus_pY_f$, where $Y_f$ is the closed span of the
$e_{n,i}$, with basis constant at most $9$, independently of
the integers $N_n$.

To verify this, write an input as $tf+\sum_n y_n$, and put
$$
m_n(y_n)=N_n^{-1/q}\sum_i y_{n,i},\quad
b_n=m_n(y_n),\quad y_n^0=y_n-b_nu_n.
$$
Its coefficient vector in the $n$-th family is
\begin{equation}\label{tr:eq:positive-coefficients}
C_n=y_n^0+a_nu_n,\quad
a_0=t/2,\quad a_{n+1}=(b_n-a_n)/2.
\end{equation}
Indeed, synthesis sends $(v_n)$ to
$$
2m_0(v_0)f+\sum_{n\ge0}
\bigl(v_n+2m_{n+1}(v_{n+1})u_n\bigr)
$$
and has norm at most $3$. On the mean coordinates this is
$2I+R$, where $R$ is the unilateral right shift, so its inverse
has norm at most one. The mean-zero parts have norm at most
twice the input norm. Thus synthesis is invertible with inverse
norm at most $3$, proving the assertion and the coefficient
formulas.

We need more precise bounds when these families are combined.
Write an initial projection, relative to $\F f\oplus_pY_f$, as
$\left(\begin{smallmatrix}\alpha&A'\\ B'&D'\end{smallmatrix}\right)$.
During the first family, after $k$ of its $N_0$ vectors,
$\alpha=k/N_0$ and $\norm{A'}\le4$. After that family is
complete, $\alpha=1$ and $A'=0$, since no later vector contains
$f$. Also $\norm{D'}\le9$ and $\norm{B'}\le1/2$ at every
partial sum. For the latter estimate, input $tf$ gives
$a_n=(-1)^n2^{-n-1}t$. In family $n\ge1$, with
$\theta=k/N_n$, the remaining coordinate part is
$$
(1-\theta)a_{n-1}u_{n-1}+a_nE_ku_n,
$$
where $E_k$ retains the first $k$ coordinates. Its norm to the
$p$-th power is
$2^{-np}|t|^p((1-\theta)^p+2^{-p}\theta)\le2^{-p}|t|^p$.
The inequality follows by convexity in $\theta$ and its endpoint
values. In the first family the estimate is immediate.

Now let $Y$ be the closed span of all the disjoint coordinates,
including $e_0$, and suppose that a projection $P_Y$ onto $Y$
has norm at most $B$. Use the local bases from
\Cref{tr:lem:positive-complement} to give $X=\ker P_Y$ a
normalized non-negative basis $(f_j)$, with initial projection
and lower coefficient bounds at most $D$. Start the reference
list with $e_0$. Introduce the targets in their basis order,
completing the first family for each before introducing the next.
Return to every target infinitely often, always completing one
family before starting another.

For $F=x+y$, where $x=(I-P_Y)F=\sum_jt_jf_j$ and
$y=P_YF$, the $X$ component arising from $x$ is a convex
combination of two successive initial projections of $(f_j)$.
Its norm is at most $D\norm x_p$. The $Y$ component arising
from $x$ has norm at most
$\frac12(\sum_j|t_j|^p)^{1/p}\le\frac D2\norm x_p$.
The $Y$ component arising from $y$ has norm at most $9\norm y_p$.
Only the target whose first family is currently incomplete can
contribute an $X$ component from $y$, of norm at most
$4\norm y_p$. Consequently every initial projection has norm
at most
\begin{equation}\label{tr:eq:positive-constant}
K=\frac32D(1+B)+13B.
\end{equation}
The coefficient formulas \eqref{tr:eq:positive-coefficients}
are bounded and biorthogonal on $X\oplus Y$. Completeness follows
by iterating
$$
h_n=\frac12\sum_iN_n^{-1/p}x_{n,i}-\frac12u_n:
$$
after $L$ iterations the remainder has norm $2^{-L}$.
Thus every $f_j$ and every $u_{j,n}$ belongs to the closed span;
the identities defining $x_{j,n,i}$ then recover every coordinate.
Together with the uniform projection bound this proves the
reference basis property.

We realize these functions on the line by the same placement
as in \Cref{tr:lem:packet}. For non-negative $v,h$, replace
\eqref{tr:eq:packet-scales} by
$$
d^p=\frac{\norm v_p^p}{1-2^p(N-1)N^{-(p-1)}},\quad
a=2dN^{-1/q},\quad
v'=v+a\sum_{i=1}^NT_{-\lambda_i}h.
$$
For large $N$ the denominator is positive. The same disjointness
calculation gives non-negative remainders $e_i$ of norm one and
$$
T_{\lambda_i}v'=d(e_i+2N^{-1/q}h),\quad
\norm{v'-v}_p=2dN^{2/p-1}\longrightarrow0.
$$
Reserve the sets in \eqref{tr:eq:regions}, using the positive
local bases beginning with $e_i$ on them and positive full bases
on the gaps. The first coordinate functionals $\phi_A$ of these
local bases replace the norming functionals in \eqref{tr:eq:PY}.
Their disjoint supports give $\norm{P_Y}\le B$. Its kernel is
the $\ell_p$ sum of their kernels and the full spaces on the gaps.
An interleaving preserving each local order therefore gives the
required bounds $D$ for $(f_j)$, by summing the local estimates
to the $p$-th power. Every target is disjoint from the new
coordinates subsequently assigned to it.

At round $j$, choose the next available target by the enumeration
of pairs used in \Cref{tr:sec:physical}, give it its first family,
and then give one further family to each of the first $j$ targets.
Every available function is eventually chosen, and every chosen
target receives infinitely many families. The expanding intervals
and their gaps exhaust $\R$ as before. This realizes all the
hypotheses used to obtain \eqref{tr:eq:positive-constant}.

Finally we retain the entire smoothing construction. Choose $e_0$
non-negative and let $G_0$ be its Gaussian smoothing. At each
stage, $h$ is non-negative; choose the approximation $v_t$ to
$G_{t-1}$ non-negative as well. Such compactly supported smooth
approximations exist by cutoff and mollification. The Gaussian
smoothing $H_t=h*\phi_{\sigma_t}$ is non-negative on $\R$.
Thus \eqref{tr:analytic:eq:Gt} preserves non-negativity of $G_t$.
Use the value of $K$ in \eqref{tr:eq:positive-constant}, choose
$0<\rho<1$, and take
$0<\kappa<\min\{1/100,\rho/(12K)\}$ as before. In the choice
of $N_t$, require
$1-2^p(N_t-1)N_t^{-(p-1)}\ge1/2$ and replace the factor $4$
in \eqref{tr:analytic:eq:Np}--\eqref{tr:analytic:eq:Ns} by $8$.
The subsequent bounds $1/2\le\norm{v_t}_p\le2$ and $d_t<4$
remain valid, and $2d_t<8$ justifies these choices.
The estimates for the approximation in \(L^p\) and the locally uniform convergence on \(\C\) therefore remain valid. The argument in \eqref{tr:eq:total-error} gives
translates of one entire function, with basis constant at most
$K(1+\rho)/(1-\rho)$. Their generator is non-negative as the
pointwise limit of the non-negative $G_t$.
The derivative estimates also remain valid with the factor $2$
in the increments. Translation of the parameters and reflection
preserve non-negativity. This proves (ii).
\end{proof}

The smallness in \eqref{tr:eq:edit-cost} is the point at which
$p>2$ is essential. At $p=2$, the denominator in
\eqref{tr:eq:packet-scales} is $1/N$, so the norm of the change
is $\sqrt N\norm{v}_2$. We therefore use a different construction
at the Hilbert-space endpoint.

\section{The weighted Fourier construction}\label{h2:sec:weighted}

We now prove \Cref{thm:translates,thm:translate-properties} for $p=2$. We use the Fourier
transform
$$
\widehat g(t)=\int_\R g(x)e^{-2\pi ixt}\,dx,
$$
extended unitarily to $L^2(\R)$, and write $e_\mu(t)=e^{2\pi i\mu t}$.
If $w>0$ almost everywhere is integrable, multiplication by $\sqrt w$
is a unitary map from $L^2(w):=L^2(\R,w(t)\,dt)$ onto $L^2(\R)$.
An exponential basis $(e_{\mu_n})$ in $L^2(w)$ therefore gives a
basis $(e_{\mu_n}\sqrt w)$ in $L^2(\R)$. Its inverse Fourier
transform consists of translates of the function $g$ with
$\widehat g=\sqrt w$, since
$\widehat{T_{-\mu}g}=e_\mu\sqrt w$.
An even weight makes this generator real and even. It is therefore
enough to prove the following theorem.

\begin{theorem}\label{h2:thm:weighted}
 For every $\varepsilon>0$, there is an even integrable weight $w>0$
almost everywhere on $\R$, bounded above by a positive Gaussian, such
that $L^2(w)$ has a Schauder basis of exponentials with basis constant
 less than $1+\varepsilon$.
The frequencies can be rational and can have arbitrarily large
pairwise separations outside finite initial portions of the basis.
\end{theorem}

\subsection{Inner products and finite matrices}

We work over $\C$ in the weighted exponential construction and
pass to real translates at the end. In this part, we use the Hermitian inner product
$\ip{f}{h}_\rho=\int f\overline h\,d\rho$, linear in its first
entry, and Hilbert space adjoints. We write
$\|f\|_\rho^2=\ip{f}{f}_\rho$. For a closed subspace $V$, let
$P_V$ be its orthogonal projection; if $e$ is a unit vector, put
$P_e=P_{\spanop\{e\}}$.

For an ordered family $(u_j)_{j=1}^d$, define its Gram matrix by
$G_{ij}=\ip{u_j}{u_i}$. Thus
$$
\left\|\sum_{j=1}^d c_ju_j\right\|^2=c^*Gc.
$$
We use $A\le B$ for Hermitian matrices when $B-A$ is positive
semidefinite. For a linearly independent family, $G$ is positive
definite and defines a norm on the coefficient space.
If $D_k$ is the diagonal projection onto the first $k$ coordinates,
the basis constant is at most $K$ precisely when
\begin{equation}\label{h2:eq:gram-prefix}
D_kGD_k\le K^2G\quad(0\le k\le d).
\end{equation}
These inequalities are preserved by convex combinations of Gram
matrices and by multiplication of the whole matrix by a positive
scalar. In particular, normalizing a measure does not change a
basis constant.

We give $\T^d=(\R/\Z)^d$ probability Haar measure. Its characters
are $e^{2\pi i k\cdot y}$, $k\in\Z^d$; on $\R\times\T^d$,
they have the form $e_\mu(t)e^{2\pi i k\cdot y}$. All tori used
below are finite-dimensional. When new variables are added,
conditional expectation means integration in those variables with
the stated conditional density. The \emph{centered part} of a
function is the function minus this expectation. We will use
orthonormal coordinates for matrix estimates, while retaining
individual characters for the eventual basis.

\subsection{Outline of the construction}

The use of decreasing weights and separated groups of frequencies
also appears in the construction of Lev and Tselishchev~\cite[Section~3]{LT2}.
In their weighted formulation, a Riesz basis of trigonometric
polynomials is expanded into individual exponentials to obtain a
Schauder frame. Here, we require the inequalities in
\eqref{h2:eq:gram-prefix} for every coefficient vector at each finite
stage, including partial sums ending inside the newly added family.
These uniform bounds, together with completeness, give the Schauder
basis property for the individual exponentials.

We start with the even Gaussian weight
$w_0(t)=\pi^{-1/2}e^{-t^2}$ and the constant character $1$.
At each stage, we append finitely many characters and decrease the
weight. The characters already chosen remain unchanged. Suppose
that their span is $E$, the current weight is $w$, and the character
$f$ is the next function we want to approximate. We write
$$
r=f-P_Ef,
$$
where $P_E$ is orthogonal projection in $L^2(w)$. Thus, $r$ is
the part of $f$ that still needs to be approximated. We will choose
new characters and a positive trigonometric polynomial $W\le1$
such that, for $w'=wW$ and the enlarged span $E'$, we have
$$
\dist_{w'}(f,E')^2\le\frac{1}{2}\dist_w(f,E)^2.
$$
The loss $\int w_0(1-W)$ can be arbitrarily small. If the old
basis constant is $K$, the enlarged family has basis constant at
most $\max\{K,K_*\}+\tau$, where $K_*=1+\varepsilon/10$ and $\tau>0$ can also be arbitrarily
small. These are the assertions of \Cref{h2:prop:extension}.

We first construct the new characters on the line together with
finitely many auxiliary circle variables. Averaging over these
variables will give prescribed functions of the original variable
$t$. We will remove the auxiliary variables at the end of the stage.
The finite construction in \Cref{h2:sec:scalar} is the first
ingredient. It gives a finite family of nonconstant torus characters and
a weight of mass arbitrarily close to one such that the distance
from $1$ to their span is bounded above by a fixed constant
less than one. Both their partial sums
and the partial sums after subtracting their means are uniformly
bounded. We use independent copies of these finite systems to
treat the characters occurring in $r$. For these finite calculations,
we divide $w$ by its integral and use the resulting probability
measure. This does not change orthogonal projections or basis constants.

To see why this is useful, write $r=\sum_{j=1}^d c_j\phi_j$
with distinct characters $\phi_j$. The coefficients $c_j$ can
be large compared with $\norm{r}$, because the $\phi_j$ can be
almost linearly dependent in $L^2(w)$. We choose polynomials
$g_j$ of mean one in separate auxiliary variables and form
$$
h=\sum_{j=1}^d c_j\phi_jg_j.
$$
Averaging $h$ in the new variables gives $r$. The parts of the
$g_j$ with mean zero are mutually orthogonal. We choose their
variances so that these orthogonal parts control the partial sums
of the individual characters occurring in $h$. Their basis
constant can then be made arbitrarily close to one, regardless of the size of the
$c_j$; see \eqref{h2:eq:carrier-gram}.

The unit vector $e=h/\norm{h}$ has average $\alpha e_r$, where
$e_r=r/\norm{r}$ and $\alpha>0$. This number may be small, so
we take several independent copies. They have the form
$$
e_i=\alpha e_r+\xi_i,
$$
where the $\xi_i$ are mutually orthogonal, orthogonal to functions
of $t$, and satisfy $\norm{\xi_i}^2=1-\alpha^2$. Consequently,
$$
\left\|\frac{1}{N\alpha}\sum_{i=1}^N e_i-e_r\right\|^2
=\frac{1-\alpha^2}{N\alpha^2}.
$$
 We use sufficiently many copies to reduce the squared approximation
error by a fixed positive proportion, while keeping the norm
of the associated operator $J$, defined in \eqref{h2:eq:info},
small.
 Repeating this step finitely many times gives the halving estimate.

The $e_i$ are polynomials, so we cannot append them to a basis
of individual characters. Let $V_i$ be the span of the characters
occurring in $e_i$. We choose another finite character family
whose span approximates $e_i$ much more effectively than the
directions in $V_i\cap e_i^\perp$. The identity in \eqref{h2:eq:info} allows us to compute squared
projection norms from the conditional averages of these characters
and the Gram matrix of their parts with mean zero. In \Cref{h2:lem:selective},
we use this identity to make the contribution in
$V_i\cap e_i^\perp$ and the loss of weight small relative to
the contribution in the direction $e_i$.

We then repeat this new construction in independent variables.
The contributions in the direction $e_i$ add, while the total
contribution in the other directions and the total loss of weight
remain small; see \eqref{h2:eq:selective-repetition}. Combining these families for all $i$ reduces the squared approximation
error for $r$ by a fixed factor, while keeping the new span almost
orthogonal to $E$. This controls
partial sums ending among either the old or the new characters.

Once the finite families have been chosen, we replace all auxiliary
variables $y$ by $A\omega t$, where $\omega$ is a suitable fixed
integer vector and $A$ is a sufficiently large positive integer.
Each product character becomes an ordinary character on the line:
$$
e_\nu(t)e^{2\pi i k\cdot y}
\quad\longmapsto\quad e_{\nu+A k\cdot\omega}(t).
$$
The Riemann--Lebesgue lemma shows that the finitely many inner
products converge to those of the product construction. The strict
approximation and partial sum estimates therefore persist for
large $A$. The same choice allows rational frequencies with large
separations. We keep every weight invariant under simultaneous
inversion of its variables, so the resulting weight on the line
is even.

We repeat the extension, returning to every rational-frequency
character infinitely often. Their distances from the constructed
spans tend to zero. We choose the losses of weight and the increases
in the partial sum bounds to be summable. The weights then decrease
to a weight $w_\infty$ that is positive almost everywhere, and the
partial sum bounds pass to the limit. Density of the rational-frequency
characters gives completeness. Finally, taking
$\widehat g=\sqrt{w_\infty}$ yields the required real generator.

\section{A finite construction for the constant function}\label{h2:sec:scalar}

We begin with the finite systems used in the outline. For Haar
measure, every nonconstant torus character is orthogonal to $1$.
We will modify the measure so that the orthogonal projection
of $1$ onto the span of a finite family of these characters
has norm bounded below by a fixed positive constant, while
the loss of mass is arbitrarily small.
The family must also have uniformly bounded partial sums. We
express these requirements in terms of its means and Gram matrix.

\subsection{Means and centered parts}\label{h2:ss:scalar-notation}

Let $0<w\le1$ be a weight on a finite torus, bounded below by a
positive constant, and let $\chi=(\chi_1,\ldots,\chi_d)$ be an
ordered list of distinct nonconstant characters. The Gram matrix of
the augmented list $(1,\chi_1,\ldots,\chi_d)$, before normalizing
the measure, has the form
\begin{equation}\label{h2:eq:scalar-gram}
B=\begin{pmatrix}m&b^*\\b&G\end{pmatrix},\quad
m=\int w,\quad A=G-\frac{bb^*}{m},\quad C=\frac{A}{m},
\quad a=\frac{b}{m}.
\end{equation}
The matrix $C$ is the Gram matrix of the centered characters in the
probability measure $w/m$. Here $b_k=\int\overline{\chi_k}w$, so
the mean of $\chi_k$ is $\overline{a_k}$. This convention agrees
with the order of the entries in our Gram matrices.
For a polynomial $u=\sum_kc_k\chi_k$,
its mean is $a^*c$. Its variance, the squared norm after subtracting
its mean, is $c^*Cc$. We set
\begin{equation}\label{h2:eq:scalar-gain}
Z=a^*C^{-1}a
 =\sup_{c\ne0}\frac{|a^*c|^2}{c^*Cc}.
\end{equation}
Thus, $Z$ is the largest possible ratio of the squared mean to
the variance. Equivalently, if $Z>0$, the smallest variance among
polynomials in the character span with mean one is $1/Z$. This
will let us prescribe the variance of the polynomials $g_j$ in
the outline.

The full normalized Gram matrix of $\chi$ is $C+aa^*$. Consequently, the
squared norm of the orthogonal projection of $1$ onto its span is
\begin{equation}\label{h2:eq:scalar-capture}
a^*(C+aa^*)^{-1}a=\frac{Z}{1+Z}.
\end{equation}
The block inverse of $B$ also gives
\begin{equation}\label{h2:eq:constant-inverse}
m(B^{-1})_{00}=1+Z,
\end{equation}
where the index $0$ denotes the constant character. These two
identities relate $Z$ to approximation of the constant and to the
inverse Gram matrix.

\begin{proposition}\label{h2:lem:scalar}
 For every $0<d\le1$, there is $c_d>0$ and a sequence of even
positive trigonometric polynomial weights $w\le1$, with finite
ordered lists of nonconstant characters, for which
$\epsilon=1-\int w\to0$ and the following hold:
\begin{enumerate}
\item the number $Z$ in \eqref{h2:eq:scalar-gain} is at least $c_d$ ;
\item the centered and augmented character families have basis
constant less than $1+d$ ;
\item the corresponding unnormalized augmented Gram matrix satisfies
\begin{equation}\label{h2:eq:scalar-extra}
B\ge\epsilon^2I,\quad
(B^{-1})_{kk}\le1+2\epsilon\quad(k\ne0),\quad
(B^{-1})_{00}<2.
\end{equation}
\end{enumerate}
\end{proposition}

By \eqref{h2:eq:scalar-capture}, the first assertion means that
the squared norm of the projection of $1$ is at least $c_d/(1+c_d)$.
It therefore stays bounded away from zero as the loss of mass
tends to zero. The extra Gram estimates will be needed when we
modify the weight again in \Cref{h2:sec:selective}.

To prove the proposition, we first allow bounded measurable weights
and replace them by polynomials at the end. Restricting the change of weight to a small
set decreases both the loss of mass and $Z$. Taking independent
copies adds their values of $Z$. We will show that these two
operations decrease the loss of mass while keeping $Z$ bounded
away from zero. All weights will be invariant under simultaneous
inversion of their variables.

\subsection{Applying the weight on a small set}\label{h2:ss:rare}

Suppose that $w$ and $\chi$ are given, with the notation of
\eqref{h2:eq:scalar-gram}. We assume that $w$ is even under simultaneous
inversion of its variables. Its character Gram matrix is then real,
and $0<B\le I$, since the characters are orthonormal for Haar measure
and $w\le1$.

Let $D$ be a symmetric measurable set of measure $\vartheta$ in a additional circle,
and define
\begin{equation}\label{h2:eq:rare-weight}
W(x,y)=1-\1_D(y)(1-w(x)),\quad M=\int W=1-\vartheta\epsilon.
\end{equation}
For $y\notin D$, the weight is one. For $y\in D$, it is
$w(x)$. Thus, the old change of weight is used only on a
proportion $\vartheta$ of the new circle.
We choose a finite symmetric set $Y=(1,\psi_1,\ldots,\psi_s)$
of Fourier characters in $y$ and add the products in the order
$$
1,\psi_1,\ldots,\psi_s,\quad
\chi_1,\chi_1\psi_1,\ldots,\chi_1\psi_s,\quad\ldots.
$$
Each row uses one old character and all the characters in $Y$;
the constant function is first in the entire list.

On the full Fourier space in the new variable, multiplication by
$\1_D$ is an orthogonal projection, denoted by $M_D$. The augmented
Gram operator and its inverse are
\begin{equation}\label{h2:eq:rare-operators}
\mathcal G=I+(B-I)\otimes M_D,\quad
\mathcal G^{-1}=I+(B^{-1}-I)\otimes M_D.
\end{equation}
Indeed, on the range of $M_D$ the old Gram is $B$, while on its
kernel the old Gram is the identity. Both operators in
\eqref{h2:eq:rare-operators} are therefore bounded and invertible.

A partial sum can stop before all the new-variable characters
in one row have been retained. We use the following matrix estimate
to control such a cut. Its application to \eqref{h2:eq:rare-operators}
is given immediately after the proof.

Let $0<B\le I$ be a real symmetric positive definite matrix and put
$H=B^{-1}$. The matrix $B$ will be the Gram matrix of an ordered
system, and $H$ the Gram matrix of its biorthogonal system. Fix a
row $j$, and let $P$ be the set of preceding row indices.
For $0\le t\le1$, define
\begin{equation}\label{h2:eq:split-matrices}
\begin{split}
C_t&=\begin{pmatrix}
B_{PP}&\sqrt t\,B_{Pj}\\
\sqrt t\,B_{jP}&1-t+tB_{jj}
\end{pmatrix},\\
D_t&=\begin{pmatrix}
H_{PP}&\sqrt t\,H_{Pj}\\
\sqrt t\,H_{jP}&1-t+tH_{jj}
\end{pmatrix}.
\end{split}
\end{equation}
At the endpoints, these matrices describe the partial sums before
and after row $j$, with an extra orthogonal coordinate at $t=0$.
The lemma bounds the intermediate projections in terms of those
endpoints and two diagonal entries.

\begin{lemma}\label{h2:lem:split}
Set $m=B_{jj}$, $h=H_{jj}$, and
$F(t)=\norm{C_t^{1/2}D_tC_t^{1/2}}$. Then
\begin{equation}\label{h2:eq:split-bound}
F(t)\le\max\{F(0),F(1)\}
 +t(1-t)(1-m)(h-1)
\quad(0\le t\le1).
\end{equation}
\end{lemma}

\begin{proof}
We first record an estimate for real skew-symmetric matrices. If $J$
is skew-symmetric and $b$ is a real vector, write
$$
T(s)=\begin{pmatrix}J&sb\\-sb^T&0\end{pmatrix}.
$$
Then
\begin{equation}\label{h2:eq:arrowhead}
\norm{T(s)}^2\le\norm{T(r)}^2+(s^2-r^2)\norm{b}^2
\quad(0\le r\le s).
\end{equation}
Indeed, put $J$ into real two-dimensional rotation blocks. Within
each block, rotate the coordinates so that the component of $b$ lies
on the first axis. A similar choice can be made in the kernel of $J$.
Grouping the first axes separately from the second axes and the new
coordinate gives
$$
T(s)=\begin{pmatrix}0&R_s\\-R_s^T&0\end{pmatrix},
\quad R_sR_s^T=D+s^2aa^T,
$$
where $D\ge0$ is fixed and $\norm{a}=\norm{b}$. Taking norms proves
\eqref{h2:eq:arrowhead}.

Let $E_P$ be the coordinate projection onto $P$, and set
$$
x=B^{1/2}e_j,\quad y=B^{-1/2}e_j,\quad
P_0=B^{1/2}E_PB^{-1/2},\quad P_1=P_0+xy^T.
$$
The identities $y^Tx=1$, $P_0x=0$, and $y^TP_0=0$ show that
\begin{equation}\label{h2:eq:split-projection}
R_t=\begin{pmatrix}
P_0+txy^T&sx\\ sy^T&1-t
\end{pmatrix},\quad s=\sqrt{t(1-t)},
\end{equation}
is a projection. To identify its norm, take the range vectors
$$
(B^{1/2}e_i,0)\quad(i\in P),\quad
(\sqrt t\,x,\sqrt{1-t}),
$$
and the biorthogonal vectors obtained by replacing $B^{1/2}$
with $B^{-1/2}$ and $x$ with $y$. Their Gram matrices are
$C_t$ and $D_t$. Factoring $R_t$ into coefficient extraction
and synthesis gives $\norm{R_t}^2=F(t)$.
For a nonzero Hilbert space projection $R$, we have
\begin{equation}\label{h2:eq:projection-skew}
\norm{R}^2=1+\norm{R-R^*}^2.
\end{equation}
To see this, write $R=\left(\begin{smallmatrix}I&Z\\0&0\end{smallmatrix}\right)$
relative to its range and its orthogonal complement. Both sides of
\eqref{h2:eq:projection-skew} equal $1+\norm{Z}^2$.

Put $J_i=P_i-P_i^T$, $A=xy^T-yx^T$, and $v=x-y$. Then
\begin{equation}\label{h2:eq:split-skew}
R_t-R_t^T=
\begin{pmatrix}(1-t)J_0+tJ_1&sv\\-sv^T&0\end{pmatrix},
\quad
\norm{v}^2=(mh-1)+(1-m)(h-1).
\end{equation}
Suppose first that $a=\sqrt{mh-1}>0$. On $\spanop\{x,y\}$,
$A^2=-a^2I$. Consequently,
$$
u=-\frac{Av}{a\norm{v}}
$$
is a unit vector and $Au=av/\norm{v}$. Conjugate
$\diag(J_0,J_1)$ by the rotation between its two copies with
coefficients $\sqrt{1-t}$ and $\sqrt t$. Compress the second copy
to $\spanop\{u\}$. Up to the sign of its last coordinate, the
result is
$$
\begin{pmatrix}
(1-t)J_0+tJ_1&sa\,v/\norm{v}\\
-sa\,v^T/\norm{v}&0
\end{pmatrix}.
$$
Its norm is at most $\max\{\norm{J_0},\norm{J_1}\}$, because
orthogonal conjugation preserves norm and compression cannot
increase it. By \eqref{h2:eq:arrowhead}, increasing the last column
from $sa\,v/\norm{v}$ to $sv$ increases the squared norm by at most
$$
s^2(\norm{v}^2-a^2)=s^2(1-m)(h-1).
$$
Now apply \eqref{h2:eq:projection-skew} to \eqref{h2:eq:split-skew} and
to the endpoint projections. If $a=0$, then $A=0$ and $J_0=J_1$;
the same argument starts with a zero last column. This proves
\eqref{h2:eq:split-bound} in all cases.
\end{proof}

To bound a partial sum ending inside old row $j$, we first retain
all Fourier modes in each completed row. Let $F$ be the Fourier
projection onto the selected part of row $j$. Since $F$ has finite
rank, diagonalizing $FM_DF$ on the range of $F$ decomposes $F$ and $M_D$
into common eigenspaces and two-dimensional pieces. More explicitly,
if $v$ is a unit vector
with $Fv=v$ and $FM_DFv=tv$ for $0<t<1$, then
$$
\frac{M_Dv-tv}{\sqrt{t(1-t)}}
$$
is a companion to $v$ in $\ker F$ of unit norm. In this ordered pair,
the matrices of $F$ and $M_D$ are
$$
\begin{pmatrix}1&0\\0&0\end{pmatrix},\quad
\begin{pmatrix}t&\sqrt{t(1-t)}\\
\sqrt{t(1-t)}&1-t\end{pmatrix}.
$$
Diagonalize $M_D$ in each completed row and retain $v$ in row
$j$. The completed coordinates in $\ker M_D$ split off as
identity summands. The remaining range and dual Gram matrices
are exactly $C_t$ and $D_t$. The common eigenspaces correspond
to the endpoint cases. Thus \Cref{h2:lem:split} controls every
partial sum inside a row.
Restriction to the finite Fourier set $Y$ preserves the estimate,
because the finite character span is invariant under these coordinate
projections. Every partial sum fixes the first constant function,
so it induces an operator of no larger norm on the quotient by
the constants. Identifying this quotient with the functions of
mean zero replaces each character by the character minus its mean.
This gives the same bound for the centered family.

We will also need inverse estimates. If $A_W$ is the unnormalized centered
Gram matrix of the finite family, its inverse is the nonconstant
principal block of the augmented inverse. Compression of the inverse of
\eqref{h2:eq:rare-operators} gives
\begin{equation}\label{h2:eq:rare-inverse}
(A_W^{-1})_{kk}\le1+\vartheta\bigl((B^{-1})_{jj}-1\bigr)
\end{equation}
for a retained nonconstant coordinate in old row $j$. Here we used
that the inverse of a compression of a positive operator is at most
the corresponding compression of its inverse, and that
$\ip{M_De_k}{e_k}=\vartheta$ for every Fourier character.

As $Y$ increases to all Fourier modes, the value of $Z$ for the finite family
increases to
\begin{align}
z_\vartheta
&=M\bigl[1+\vartheta((B^{-1})_{00}-1)\bigr]-1\notag\\
&=\frac{\vartheta}{1-\epsilon}
 \bigl\{Z+\epsilon^2-\vartheta\epsilon(Z+\epsilon)\bigr\}.
\label{h2:eq:raregain}
\end{align}
Indeed, \eqref{h2:eq:constant-inverse} expresses the finite value of $Z$ in
terms of the constant entry of the inverse Gram. These entries
increase to the corresponding entry of $\mathcal G^{-1}$ by the
variational formula for inverse quadratic forms. In particular,
every strict lower bound below $z_\vartheta$ holds for a sufficiently large
finite Fourier set.

\subsection{Partial sums for independent copies}

The preceding operation decreases both the loss of mass and
the quantity $Z$ in \eqref{h2:eq:scalar-gain}. We take sufficiently
many independent copies to keep $Z$ bounded away from zero.
The centered parts of different copies are orthogonal, while
their means belong to the same one-dimensional space of
constant functions. The next lemma controls the partial sums in this situation. It distinguishes the sum of the contributions from all copies from the contribution of the one copy where the partial sum stops.

\begin{lemma}\label{h2:lem:copies}
Let $\mathcal H=V_0\oplus\bigoplus_iH_i$ be an orthogonal Hilbert
sum, with finitely many $H_i$. Suppose that each $H_i$ has an ordered
finite basis $(x_{ik})_k$ with basis constant at most $\sqrt Q$.
Let $T_i:H_i\to V_0$, and suppose
$$
\left\|\sum_iT_iT_i^*\right\|\le Z,\quad
\norm{T_iT_i^*}\le\eta,\quad Q>1+Z.
$$
Concatenate the families $(x_{ik}+T_ix_{ik})_k$ after $V_0$.
Every coordinate projection retaining $V_0$ and an initial segment
of this list has squared norm at most
\begin{equation}\label{h2:eq:copies-bound}
Q+\frac{Q^2\eta}{Q-1-Z}.
\end{equation}
The same bound holds on the span of the concatenated families alone.
\end{lemma}

\begin{proof}
At a partial sum ending inside family $i$, let $P$ be its centered
coordinate projection on $H_i$. The preceding centered spaces are
retained in full and split off as identity operators. On the remaining
spaces, the partial sum has the form
$$
M=\begin{pmatrix}A&-U\\0&P\end{pmatrix}.
$$
Here $A$ maps the old component and all later centered components
to $V_0$, and $U=T_i(I-P)$. Therefore
$$
\norm{A}^2\le1+Z,\quad
\norm{U}^2\le Q\eta,\quad \norm{P}^2\le Q.
$$
For the bound on $U$, we used that a nontrivial projection on a
Hilbert space has the same norm as its complementary projection.
The endpoint cases $P=0,I$ satisfy the same bound directly.

For vectors $x,y$ in the two domain summands, these bounds give
\begin{align*}
\norm{M(x,y)}^2
&\le\bigl(\sqrt{1+Z}\,\norm{x}+\sqrt{Q\eta}\,\norm{y}\bigr)^2
 +Q\norm{y}^2\\
&\le Q\norm{x}^2+
 \left(Q+\frac{Q^2\eta}{Q-1-Z}\right)\norm{y}^2.
\end{align*}
The second inequality is Cauchy--Schwarz with weights $1+Z$ and
$Q-1-Z$, whose sum is $Q$. This proves \eqref{h2:eq:copies-bound}.
Finally, the span of the concatenated families is invariant under
these coordinate projections, so restriction to this span cannot
increase their norms.
\end{proof}

\subsection{Independent copies and the induction}\label{h2:ss:scalar-induction}

 Fix $0<d\le1$, and put $t=d/4$ and $a_0=t/100$.
 We begin with the circle weight and single character
$$
w_0(x)= 1-a_0 + a_0 \cos(2\pi x),\quad
\chi_1(x)=e^{2\pi ix}.
$$
The augmented Gram is
$B_0=\left(\begin{smallmatrix}1-a_0&a_0/2\\a_0/2&1-a_0\end{smallmatrix}\right)$
,
so a direct calculation gives
\begin{equation}\label{h2:eq:scalar-initial}
\epsilon_0= a_0,\quad Z_0= \frac{a_0^2}{4-8a_0+3a_0^2},\quad
(B_0^{-1})_{11}-1 \le2a _0.
\end{equation}
 Since $a_0\le1/400$, we also have
$a_0^2/4\le Z_0\le a_0^2/3$ and $B_0\ge a_0^2I$.
 Its centered family has one vector and hence basis constant one.
We use $Q_0=1+t$ as an upper bound for its squared basis constant,
leaving room for the increases at later steps.

At stage $n$, $Q_n$ will bound the squared basis constant of the
centered family. We maintain
$$
Z_0/2\le Z_n\le Z_0,\quad
(B_n^{-1})_{kk}-1\le2\epsilon_n\quad(k\ne0),
\quad 1+t \le Q_n< 1+2t.
 $$
We also retain the description of the family as a concatenation of
independent copies, each with the quantity in \eqref{h2:eq:scalar-gain} at most $\eta_n$. Initially,
there is only one copy and $\eta_0=Z_0$. Below we suppress the
index $n$ on $\epsilon,Z,B$.

Put $C_0=\max_j((B^{-1})_{jj}-1)$, including the constant index.
Choose $\vartheta>0$ so small that
\begin{equation}\label{h2:eq:choose-p}
\vartheta\le\frac{1}{100},\quad
\vartheta\le\frac{\epsilon}{Z+\epsilon},\quad
\vartheta C_0\le\frac{\epsilon}{100},\quad
z_\vartheta\le Z_0\,2^{-n-3}.
\end{equation}
Choose a finite Fourier set in \Cref{h2:ss:rare} whose quantity $z$ in \eqref{h2:eq:scalar-gain} satisfies
$$
(1-\epsilon/4)z_\vartheta\le z\le z_\vartheta.
$$
We now take $L=\lfloor Z/z\rfloor$ copies of that finite system,
each in a separate set of variables. We multiply the weights and
list the nonconstant characters of the first copy, followed by those
of the second copy, and so on. After subtracting their means,
functions from different copies are orthogonal. Their centered
Gram matrix is therefore a direct sum, and the quantities in
\eqref{h2:eq:scalar-gain} add. The new parameters are
\begin{equation}\label{h2:eq:new-parameters}
1-\epsilon'=M^L,\quad Z'=Lz,\quad \eta_{n+1}=z.
\end{equation}
The decrease from $Z$ to $Z'$ is less than $z$. Summing the last bound in
\eqref{h2:eq:choose-p} shows that the total decrease in $Z$ is at most
$Z_0/4$. In fact, the construction maintains $Z_n\ge3Z_0/4$.

We next check that the mass loss decreases to zero. From
\eqref{h2:eq:raregain}, the condition $\vartheta(Z+\epsilon)\le\epsilon$,
and the choice of $z$, we obtain
\begin{equation}\label{h2:eq:Lp}
L\vartheta\le\frac{1-\epsilon}{1-\epsilon/4}
\le1-\frac{3\epsilon}{4},\quad
L\vartheta\ge\frac{(1-\epsilon)Z}{Z+\epsilon^2}-\vartheta>\frac{1}{10}.
\end{equation}
For the last strict inequality, use $Z\ge3Z_0/4\ge3a_0^2/16$,
 $\epsilon\le a_0$, and $\vartheta\le1/100$ to get
$$
\frac{(1-\epsilon)Z}{Z+\epsilon^2}-\vartheta
\ge (1-a_0)\frac{3}{19} -\frac{1}{100} >\frac{1}{10}.
$$
Bernoulli's inequality and
$1-(1-a)^L\ge La/(1+La)$ now imply
\begin{equation}\label{h2:eq:epsrec}
\frac{\epsilon}{12}\le\epsilon'
=1-(1-\vartheta\epsilon)^L
\le\epsilon-\frac{3\epsilon^2}{4}.
\end{equation}
The upper bound forces $\epsilon_n\to0$ and gives
 $\sum_n\epsilon_n^2\le4a_0/3$.

The new nonconstant inverse diagonal entries are bounded, by
\eqref{h2:eq:rare-inverse} and the normalization in
\eqref{h2:eq:new-parameters}, by
$$
\frac{M(1+\vartheta C_0)}{1-\epsilon'}.
$$
Their excess over one is at most
$$
\frac{\epsilon'+\epsilon/100}{9/10}<2\epsilon',
$$
where we used the lower bound in \eqref{h2:eq:epsrec}. The inverse
diagonal condition is therefore preserved.

It remains to bound partial sums. The old family is a concatenation
of independent copies, each with centered squared basis constant
at most $Q_n$. Apply \Cref{h2:lem:copies} with $V_0$ the constants,
total contribution $Z_n$, and individual contribution at most
$\eta_n$. Its hypothesis $Q_n>1+Z_n$ holds because $Q_n\ge1+t$
 and $Z_n\le Z_0<t/2$. The old augmented squared basis constant is at
most $Q_n+(5/t)\eta_n$ ; indeed,
 $$
\frac{Q_n^2}{Q_n-1-Z_n}
\le\frac{(1+2t)^2}{t/2}<\frac{5}{t}
\quad(0<t\le1/4).
$$
 For a nonconstant old row, \Cref{h2:lem:split} adds at most
$$
\frac{1}{4}(1-B_{jj})\bigl((B^{-1})_{jj}-1\bigr)
\le\frac{\epsilon^2}{2}.
$$
For the constant row, the endpoint squared bound is $1+Z$, and
the whole row has squared bound at most
$$
1+Z+\frac{\epsilon(Z+\epsilon)}{4(1-\epsilon)} \le1+a_0^2 < 1+t.
 $$
After quotienting by the constant and taking independent copies,
we may therefore set
\begin{equation}\label{h2:eq:qrec}
Q_{n+1}=Q_n+ \frac{5}{t} \eta_n+\frac{\epsilon_n^2}{2}.
\end{equation}
Since $\sum_n\eta_n<3Z_0/2$, this yields
\begin{equation}\label{h2:eq:qbound}
Q_n \le1 + t + \frac{15Z_0}{2t}+\frac{2a_0}{3}\le1+t+\frac{t}{4000}+\frac{t}{150} < 1+2t.
 \end{equation}
All the inductive conditions are now verified.

\subsection{Gram matrix estimates and polynomial weights}

We complete the proof of \Cref{h2:lem:scalar} by checking the remaining
matrix estimates and replacing the measurable weights by polynomials.
Let $h=\norm{B^{-1}}$. By \eqref{h2:eq:rare-operators} and
the inequality for the inverse of a compression, the inverse
of the augmented Gram matrix of the finite system from
\Cref{h2:ss:rare} has norm at most $h$.
The inverse of its unnormalized centered Gram matrix is a
principal block of this inverse. Consequently, the inverse
of its normalized centered Gram matrix is bounded above by $MhI$.

After taking the $L$ independent copies, let $D_0$ be their normalized
centered inverse and $a_0$ their mean vector. Then
$$
D_0\le MhI,\quad a_0^*D_0a_0=Z'\le Z_0.
$$
The new augmented inverse is
\begin{equation}\label{h2:eq:inverse-new}
\frac{1}{1-\epsilon'}
\begin{pmatrix}
1+Z'&-a_0^*D_0\\-D_0a_0&D_0
\end{pmatrix}.
\end{equation}
The Cauchy--Schwarz inequality in the $D_0$ inner product bounds its
quadratic form at $(x,y)$ by the quadratic form at $(|x|,\norm{y})$
of the two-by-two matrix with diagonal entries $1+Z',Mh$ and
off-diagonal entry $\sqrt{MhZ'}$, divided by $1-\epsilon'$.
The largest eigenvalue of a positive matrix is at most its trace,
so
$$
h'\le\frac{h+1+Z_0}{1-\epsilon'}.
$$
Since $(1-t)/t^2$ is decreasing on $(0,1)$, the upper bound in
\eqref{h2:eq:epsrec} gives
$$
\frac{1-\epsilon'}{(\epsilon')^2}-\frac{1}{\epsilon^2}
\ge\frac{1/2+3\epsilon/16}{\epsilon(1-3\epsilon/4)^2}
\ge\frac{1}{2\epsilon}>1+Z_0.
$$
Hence
$h\le\epsilon^{-2}$ implies $h'\le(\epsilon')^{-2}$.
Initially $B_0\ge a_0^2I$, so $B_n\ge\epsilon_n^2I$ for every $n$.
Also, \eqref{h2:eq:constant-inverse} gives
$(B_n^{-1})_{00}=(1+Z_n)/(1-\epsilon_n)<2$.

The augmented basis constant is less than
\begin{equation}\label{h2:eq:augmented-bound}
 \sqrt{1+2t} (1+\sqrt{Z_0}) \le(1+t)(1+a_0) < 1+d.
 \end{equation}
Indeed, in coordinates given by the constant and the centered
functions, an augmented partial sum acts as
$$
(a,x)\longmapsto\bigl(a-T(I-P)x,Px\bigr),
$$
where $\norm{T}=\sqrt Z$ and $P$ is the centered partial sum.
Its norm is at most $\norm{\diag(1,P)}+\sqrt Z\norm{I-P}$,
which gives \eqref{h2:eq:augmented-bound}, including the endpoint cases.

Finally, take sufficiently accurate product Fej\'er means of each
finite weight. Convolution with these positive kernels preserves the
mass, the upper bound one, and the positive lower bound. For a fixed
character list, the new Gram matrix is $B\circ\Phi$, where $\circ$
denotes entrywise multiplication and $\Phi$ is a positive semidefinite
matrix with diagonal entries one. Thus the lower bound
$B\ge\epsilon^2I$ and the diagonal projection inequalities
\eqref{h2:eq:gram-prefix} are preserved.

The inverse diagonal estimates are also preserved. To see this,
take the entrywise product of the two positive block matrices
$$
\begin{pmatrix}B&I\\I&B^{-1}\end{pmatrix},\quad
\begin{pmatrix}\Phi&\Phi\\\Phi&\Phi\end{pmatrix}.
$$
Their entrywise product is positive and has off-diagonal block $I$. Its Schur
complement gives
\begin{equation}\label{h2:eq:schur-smoothing}
(B\circ\Phi)^{-1}\le B^{-1}\circ\Phi.
\end{equation}
In particular, no inverse diagonal entry increases. The finite Gram
matrices converge to the original ones as the Fej\'er order increases.
For each fixed family, its centered Gram and $Z$ depend continuously
on these entries. We choose the Fej\'er order separately for each
family. By continuity and the strict inequality in \eqref{h2:eq:qbound},
the centered basis constant remains less than $1+d$ for
sufficiently high Fej\'er order. Likewise, the bound
$Z_n\ge3Z_0/4$ ensures that $Z$ remains at least $c_d:=Z_0/2$.

We will also need to prescribe a small value of $Z$ while making
the loss of mass arbitrarily small.

\begin{corollary}\label{h2:cor:tune}
For each fixed $0<d\le1$ and $0<z<c_d$, there are positive trigonometric polynomial
weights $w\le1$, invariant under simultaneous inversion, and
ordered nonconstant character lists with $Z=z$, augmented and centered basis
 constants less than $1+d$, and arbitrarily
small mass loss $1-\int w$.
\end{corollary}

\begin{proof}
For a system from \Cref{h2:lem:scalar}, interpolate with Haar measure:
$$
w_\theta=1-\theta(1-w),\quad0\le\theta\le1.
$$
The augmented Gram is $(1-\theta)I+\theta B$. Its basis constant
is less than $1+d$ by \eqref{h2:eq:gram-prefix}, and quotienting by the
first constant gives the same bound for the centered system. The quantity $Z$
depends continuously on the finite Gram matrix. It is zero at
$\theta=0$ and at least $c_d$ at $\theta=1$, so it equals $z$
at some intermediate value. The mass loss is at most that of $w$.
\end{proof}

\section{Approximation of a prescribed vector}\label{h2:sec:selective}

The preceding construction approximates the constant function.
We now need to approximate a prescribed unit vector $e$ in a
finite character span $V$. The new character span should approximate
$e$ much more effectively than the directions in $V$ orthogonal
to $e$. This distinction will keep the new span nearly orthogonal
to the characters already chosen in the basis.

We again add variables and change the weight. Averaging the new
characters over these variables gives functions in $V$. We first
write the projection onto their span in terms of these averages.
This identifies the estimate that the new weight must satisfy.

\subsection{An identity for projection norms}\label{h2:ss:information}

Let $\rho$ be a probability measure on $\R$ times a finite torus,
and write $s$ for these variables. We add torus variables $y$
and a positive weight $W(s,y)$ whose integral in $y$ is independent
of $s$:
$$
\int W(s,y)\,dy=M>0.
$$
The probability measure for the enlarged space is
$\rho(ds)W(s,y)\,dy/M$. Its marginal in $s$ is still $\rho$.
For a finite character family $U_1,\ldots,U_d$, let
$T\colon\C^d\to L^2(\rho)$ take a coefficient vector to the
corresponding sum of conditional averages:
$$
Tc=\sum_kc_k\E[U_k | s],
$$
and let $C$ be the Gram matrix of the centered functions
$U_k-\E[U_k |  s]$ in this probability measure. We assume $C>0$.
The full normalized Gram of $(U_k)$ is $C+T^*T$.

For a function $v$ depending only on $s$, its squared projection norm onto
$\spanop\{U_k\}$ is
$$
\ip{T(C+T^*T)^{-1}T^*v}{v}.
$$
Set
$J=TC^{-1}T^*$. Writing $S=TC^{-1/2}$ gives
\begin{equation}\label{h2:eq:info}
T(C+T^*T)^{-1}T^*
=S(I+S^*S)^{-1}S^*
=J(I+J)^{-1}.
\end{equation}
For the last identity, use
$(I+SS^*)S=S(I+S^*S)$. Thus, \eqref{h2:eq:info} gives the squared
projection norm of any function of the variables $s$. In particular,
if $F$ is the span of the new characters, then
$$
\dist(v,F)^2=\norm{v}_\rho^2-
\ip{J(I+J)^{-1}v}{v}_\rho
\quad(v\in L^2(\rho)).
$$
For conditionally independent families, $C$ is a direct sum. It follows
that the operator $J=TC^{-1}T^*$ for their union is the sum of
the corresponding operators for the individual families.

\begin{proposition}\label{h2:lem:selective}
Let $V$ have a finite ordered basis of distinct characters with basis
constant $K_V$ in $L^2(\rho)$, and let $e\in V$ have norm one.
 Fix $K_s>1$. There is a sequence of positive trigonometric polynomial weights
$W\le1$ with constant conditional mass $M\to1$, and ordered lists
of characters with nonzero frequency in the new variables, such that
their conditional means lie in $V$ and
\begin{equation}\label{h2:eq:selective-J}
J=cP_{V\cap e^\perp}+jP_e,\quad
j\to0,\quad \frac{c}{j}\to0,\quad\frac{1-M}{j}\to0.
\end{equation}
 If $\epsilon$ is the auxiliary scalar mass loss, their centered basis
constants satisfy
\begin{equation}\label{h2:eq:selective-bound}
K_{\rm sel}^2\le1+
\left(\sqrt{K_s^2-1}+\sqrt{(1+2\epsilon)K_V^2-1}\right)^2.
\end{equation} If $\rho$ is invariant under
simultaneous inversion and $e(-s)=\overline{e(s)}$, then the
weights may also be chosen invariant under simultaneous inversion.
\end{proposition}

Fix $\sigma>0$. For each system in the proposition, take
$n=\lceil\sigma/j\rceil$ copies using separate new variables
and the same original variables $s$. Equip these variables with
the product of the conditional densities, so that the copies
are independent conditional on $s$. Then
\begin{equation}\label{h2:eq:selective-repetition}
nj\longrightarrow\sigma,\quad
nc\longrightarrow0,\quad n(1-M)\longrightarrow0.
\end{equation}
The centered spaces of different copies are orthogonal, so
their operators $J$ add. The operator for the combined family
is therefore $nJ$, and $nJ\to\sigma P_e$ in operator norm.
If $F_n$ denotes the combined character span, then
\eqref{h2:eq:info} gives
$$
\norm{P_{F_n}e}^2\longrightarrow\frac{\sigma}{1+\sigma},
\qquad
\norm{P_{F_n}v}^2\longrightarrow0
\quad(v\in V\cap e^\perp).
$$
Moreover, the total loss of mass is
$1-M^n\le n(1-M)\to0$.
These are the properties used in \Cref{h2:sec:extension}.
The effect of repetition is illustrated in
\Cref{fig:selective-repetition}.

\begin{figure}[htbp]
\color{black}
\centering
\begin{tikzpicture}[x=1cm,y=1cm,font=\small,line cap=round,
 every node/.style={text=black,inner sep=1pt}]
\node at (2.9,4.0) {One family};
\node at (11.55,4.0) {$n=\lceil\sigma/j\rceil$ independent copies};
\foreach \x in {0.4,8.9}{
 \draw[draw=black,line width=0.6pt,-{Stealth[length=4pt,width=3pt]}]
 (\x,0.55) -- (\x+5.1,0.55);
 \draw[draw=black,line width=0.6pt,-{Stealth[length=4pt,width=3pt]}]
 (\x,0.55) -- (\x,3.7);
}
\path[draw=black,fill=black!18,line width=0.6pt] (1.15,0.55) rectangle (2.35,1.27);
\path[draw=black,fill=black!4,line width=0.6pt] (3.65,0.55) rectangle (4.85,0.60);
\node at (1.75,1.55) {$j$};
\node at (4.25,0.88) {$c$};
\node at (1.75,0.15) {$e$};
\node at (4.25,0.15) {$V\cap e^\perp$};
\path[draw=black,fill=black!18,line width=0.6pt] (9.65,0.55) rectangle (10.85,3.43);
\path[draw=black,fill=black!4,line width=0.6pt] (12.15,0.55) rectangle (13.35,0.75);
\draw[draw=black,line width=0.6pt,densely dashed] (8.9,3.2) -- (11.25,3.2);
\node[anchor=east] at (8.72,3.2) {$\sigma$};
\node at (10.25,3.68) {$nj\to\sigma$};
\node at (12.75,1.03) {$nc\to0$};
\node at (10.25,0.15) {$e$};
\node at (12.75,0.15) {$V\cap e^\perp$};
\draw[draw=black,line width=0.6pt,-{Stealth[length=4pt,width=3pt]}]
 (5.65,2.15) -- (8.2,2.15);
\node at (6.925,2.55) {Repeat};
\node at (2.9,-0.55) {$c/j\to0$};
\node at (11.4,-0.55) {$n(1-M)\to0$};
\end{tikzpicture}
\caption{Effect of repetition on $J$. The bars show
the eigenvalues of $J$ on $\spanop\{e\}$ and $V\cap e^\perp$,
and those of the sum $nJ$ for conditionally independent copies. Repetition raises the
coefficient on $e$ to a quantity tending to $\sigma$, while the
coefficient on the orthogonal directions and the total loss of
weight tend to zero. Hence $nJ\to\sigma P_e$. The bar heights are
schematic; squared projection norms are given by $J(I+J)^{-1}$,
as in \eqref{h2:eq:info}.}
\label{fig:selective-repetition}
\end{figure}
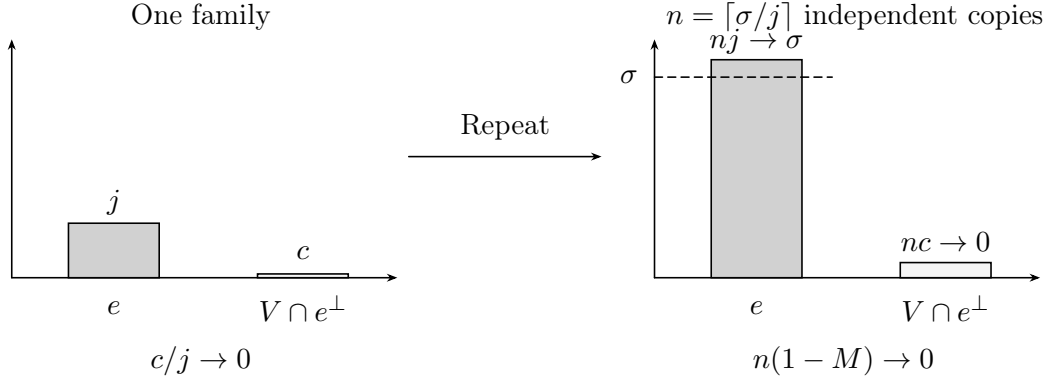

\subsection{Adding a constant to the scalar weight}

 Choose $0<d\le1$ with $1+d<K_s$ and take scalar systems
from \Cref{h2:lem:scalar} for this fixed $d$, writing $w_{\rm aux}$
for their weights, $\epsilon$ for their mass losses, and $G,b$ for
the blocks in \eqref{h2:eq:scalar-gram}. Define
\begin{equation}\label{h2:eq:response}
\Phi(t)=(1-t)^2b^*[tI+(1-t)G]^{-1}b,
\quad0\le t\le1.
\end{equation}
This is the squared norm, before normalizing the measure, of the
projection of the constant onto the nonconstant character span
when $w_{\rm aux}$ is replaced by
$t+(1-t)w_{\rm aux}$. The modified weight is bounded below by $t$. The function $\Phi$ is decreasing: after diagonalizing $G$,
it is a sum of nonnegative multiples of
$(1-t)^2/(t+(1-t)\lambda)$ with $0<\lambda\le1$.

There is a constant $c_0>0$, depending on the fixed scalar bound, such that, for all sufficiently small
$\epsilon$ in our sequence,
\begin{align}
\Phi(\epsilon^2)&\ge c_0/4,&
\Phi(\sqrt\epsilon)&\le\sqrt\epsilon,\label{h2:eq:response-endpoints}\\
\Phi(Rt)&\ge\frac{\Phi(t)}{4R}
&&\text{if }R\ge1\text{ and }Rt\le1/2.
\label{h2:eq:response-ratio}
\end{align}
Indeed, \eqref{h2:eq:scalar-capture} gives
$\Phi(0)=(1-\epsilon)Z/(1+Z)\ge c_0$.
Since $G\ge\epsilon^2I$, inserting $t=\epsilon^2$ in
\eqref{h2:eq:response} decreases this by at most a factor four for
small $\epsilon$. Also, the nonconstant Fourier coefficients of
$w_{\rm aux}$ are the negatives of those of $1-w_{\rm aux}$.
Bessel's inequality for Haar measure and $0\le1-w_{\rm aux}\le1$ give
$$
\norm{b}^2\le\int(1-w_{\rm aux})^2\le\epsilon.
$$
This proves the upper bound at $\sqrt\epsilon$. Finally,
$$
RtI+(1-Rt)G\le R[tI+(1-t)G]
$$
and $(1-Rt)^2/(1-t)^2\ge1/4$ imply
\eqref{h2:eq:response-ratio}.

\subsection{Construction of an auxiliary weight}\label{h2:ss:seed}

We next construct a weight that distinguishes $e$ from the
other directions in $V$. Its Gram matrix will have a small
eigenvalue whose eigenvector is close to $e$. Adding several
characters with the same conditional average produces this
matrix. We write $\norm{e}_\infty$ for the supremum over all
the variables $s$; it is finite because $e$ is a finite character
polynomial. We set
$$
\kappa=\frac{1}{4\norm{e}_\infty},\quad
\theta=\kappa^2,\quad a=1-\theta.
$$
Since $\rho$ is a probability measure and $\norm{e}_\rho=1$,
we have $\norm{e}_\infty\ge1$. Hence $0<\theta\le1/16$ and
$a\ge15/16$.
For $L$ additional circle variables $z_1,\ldots,z_L$, define
\begin{equation}\label{h2:eq:seed}
V_0(s,z)=\prod_{\ell=1}^L
\frac{1}{2}\bigl[1+2\kappa\operatorname{Re}
 (e(s)e^{-2\pi iz_\ell})\bigr].
\end{equation}
Every factor lies between $1/4$ and $3/4$. Thus $V_0$ is strictly
positive, bounded above by one, and has conditional mass $2^{-L}$.
Multiplying it by $m2^L$ gives a weight $V_1$ of any mass
$0<m\le2^{-L}$, still with $0<V_1\le1$.

Write $\zeta_\ell=e^{2\pi iz_\ell}$. Under the probability measure with density $V_1/m$, direct integration gives
\begin{equation}\label{h2:eq:seed-moments}
\E[\zeta_\ell |  s]=\kappa e(s),\quad
\E[\zeta_\ell\overline{\zeta_k} |  s]
=\kappa^2|e(s)|^2\quad(\ell\ne k).
\end{equation}
For the second identity, the variables $z_\ell$ are conditionally
independent. Consider the row family consisting of the given basis
of $V$ followed by $\zeta_1,\ldots,\zeta_L$. Under the reference
measure $\rho\otimes dz$, the characters $\zeta_\ell$ are orthonormal and
orthogonal to $V$. In reference orthonormal coordinates, its unnormalized
Gram $B_s$ has eigenvalue $m$ on $V\cap e^\perp$ and
eigenvalue $ma$ on the differences $\zeta_\ell-\zeta_k$. On the remaining two
dimensions, spanned by $e$ and $L^{-1/2}\sum_\ell\zeta_\ell$, it is
\begin{equation}\label{h2:eq:seed-matrix}
mH_L=m\begin{pmatrix}1&s_L\\s_L&a+s_L^2\end{pmatrix},
\quad s_L^2=L\theta.
\end{equation}
These entries follow from \eqref{h2:eq:seed-moments} and $\norm{e}_\rho=1$.
Moreover, $0<B_s\le I$, because the weight $V_1$ is a positive
contraction relative to the reference measure.

Let $\gamma<\Gamma$ be the eigenvalues of $H_L$. Its determinant
is $a$, and its trace is $1+a+L\theta$. Hence, as $L\to\infty$,
$$
\gamma\Gamma=a,\quad \Gamma\asymp L,\quad
R:=\gamma^{-1}\asymp L.
$$
The squared component of $e$ in the $\gamma$ eigenspace is
\begin{equation}\label{h2:eq:spectral-mass}
\omega=\frac{\Gamma-1}{\Gamma-\gamma}\longrightarrow1.
\end{equation}
For example, this follows by writing
$1=\ip{H_Le}{e}=\omega\gamma+(1-\omega)\Gamma$.
Thus the projection of $e$ onto the eigenspace corresponding to
$m\gamma$ has squared norm tending to one, whereas $B_s=mI$ on
$V\cap e^\perp$.

\subsection{Choosing the parameters}

We now choose $L$ and $m$ so that $\Phi(m\gamma)$ tends to
zero but is much larger than both $\Phi(m)$ and $\epsilon$.
The comparison with $\Phi(m)$ separates the direction $e$
from $V\cap e^\perp$. The comparison with $\epsilon$ will
make the loss of weight small enough after repetition.
We keep $e$, and hence $\theta>0$, fixed and let the scalar
mass loss $\epsilon$ tend to zero. We choose
$$
L=\left\lfloor\frac{\log(1/\epsilon)}{4\log2}\right\rfloor.
$$
For small $\epsilon$, \eqref{h2:eq:spectral-mass} gives $\omega\ge3/4$.
Consider the grid $t_k=\epsilon^2R^k$, ending at its first value
at least $\sqrt\epsilon$. Every grid point is at most
$R\sqrt\epsilon=o(\epsilon^{1/4})\le2^{-L}$.

Let $k_A$ and $k_B$ be the first indices where $\Phi(t_k)$ is at most
$\epsilon^{1/4}$ and $\epsilon^{1/2}$, respectively. They exist by
\eqref{h2:eq:response-endpoints}. By \eqref{h2:eq:response-ratio},
for sufficiently small $\epsilon$ we have $k_A<k_B$ and
\begin{equation}\label{h2:eq:crossings}
\frac{\Phi(t_{k_A})}{\Phi(t_{k_B})}
>\frac{\epsilon^{-1/4}}{4R}.
\end{equation}
There are at most $3\log(1/\epsilon)/(2\log R)+1$ grid steps.
It follows that some $k_A<k\le k_B$ satisfies
$$
\frac{\Phi(t_k)}{\Phi(t_{k-1})}\le R^{-1/16}.
$$
Indeed, otherwise the logarithm of the ratio in
\eqref{h2:eq:crossings} would be less than
$3\log(1/\epsilon)/32+O(\log R)$, whereas
\eqref{h2:eq:crossings} makes it greater than
$\log(1/\epsilon)/4-\log(4R)$.

Set $m=t_k$ and $u=\Phi(m\gamma)=\Phi(t_{k-1})$. Then
\begin{equation}\label{h2:eq:window}
\sqrt\epsilon<u\le\epsilon^{1/4},\quad
\frac{\Phi(m)}{u}\to0,\quad \frac{\epsilon}{u}\to0.
\end{equation}
The bound $m\le2^{-L}$ permits the weight $V_1$ of mass $m$ from
\Cref{h2:ss:seed}. Thus the weight is available at the chosen
$m$, and \eqref{h2:eq:window} gives the two comparisons we need.

\subsection{The combined weight and its means}

Let $d=1-w_{\rm aux}$. We combine the two weights by setting
\begin{equation}\label{h2:eq:combined-kernel}
W=w_{\rm aux}+dV_1.
\end{equation}
The auxiliary scalar variables are independent of the variables $z_\ell$.
The weight is a positive trigonometric polynomial, $0<W\le1$, and
its conditional mass relative to the original variables $s$ is
\begin{equation}\label{h2:eq:combined-mass}
M=1-(1-m)\epsilon.
\end{equation}
For each nonconstant auxiliary character $\chi_k$, retain every
product $\phi_j\chi_k$ and $\zeta_\ell\chi_k$, where $(\phi_j)$
is the given character basis of $V$. Order the list by auxiliary
column $k$, and within each column by the row order
$(\phi_1,\ldots,\phi_d,\zeta_1,\ldots,\zeta_L)$.

Put $\mu_k=\int\chi_kw_{\rm aux}$. These moments are real by the
evenness in the scalar construction. Integrating
\eqref{h2:eq:combined-kernel} gives
\begin{align}
\E_W[\phi_j\chi_k |  s]
&=\frac{(1-m)\mu_k}{M}\phi_j(s),\label{h2:eq:carrier-mean}\\
\E_W[\zeta_\ell\chi_k |  s]
&=-\frac{m\kappa\mu_k}{M}e(s).
\label{h2:eq:seed-mean}
\end{align}
The minus sign in the second formula comes from
$\int\chi_kd=-\mu_k$. In particular, the means of \emph{all}
retained characters lie in $V$, including the rows indexed by $\zeta_\ell$.

Identify the coefficient space with the tensor product of the
reference row space and the auxiliary coefficient space. After
orthonormalizing only the row coordinates, the unnormalized Gram
of the new characters is
$$
B_s\otimes I+(I-B_s)\otimes G.
$$
The cross-moment map from a vector of $V$ is given in the same
coordinates by $(I-B_s)$ and the auxiliary moment vector $b$.
At an eigenvalue $\beta$ of $B_s$, the Gram is therefore
$\beta I+(1-\beta)G$, with cross-moment vector $(1-\beta)b$.
By \eqref{h2:eq:response}, the operator giving squared projection norms before normalization on $V$ is
$$
P_V\Phi(B_s)P_V.
$$
Here, $\Phi(B_s)$ is obtained by applying $\Phi$ to the
eigenvalues of the positive matrix $B_s$. The compression $P_V\Phi(B_s)P_V$ acts as $\Phi(m)$
on $V\cap e^\perp$ and as
\begin{equation}\label{h2:eq:v-capture}
v=\omega \Phi(m\gamma)+(1-\omega)\Phi(m\Gamma)
\end{equation}
on $e$. Since $\Phi$ is decreasing and $\omega\ge3/4$,
$3u/4\le v\le u$.

Normalizing the measure divides this operator by
$M$. The centered Gram is positive definite: under a strictly
positive conditional density, a nonzero combination of distinct
characters with nonzero new-variable frequencies cannot be a function of $s$ alone. Since the entire range of the mean map lies in $V$,
\eqref{h2:eq:info} now gives
$$
J=cP_{V\cap e^\perp}+jP_e,
\quad c=\frac{\Phi(m)}{M-\Phi(m)},\quad j=\frac{v}{M-v}.
$$
The denominators are positive, either by positive definiteness of
the centered Gram or directly from $\Phi(m),v\to0$ and $M\to1$.
All the limits in \eqref{h2:eq:selective-J} follow from
\eqref{h2:eq:window} and \eqref{h2:eq:combined-mass}.

\subsection{Partial sum estimates in the prescribed order}

It remains to control partial sums in the column order specified above.
 Temporarily include the auxiliary constant column with all rows.
At each eigenvalue $0<\beta\le1$ of $B_s$, the augmented auxiliary
Gram is $\beta I+(1-\beta)B$. By
\eqref{h2:eq:gram-prefix}, initial projections retaining only complete columns have norm
at most $K_s$. Inverse convexity gives
 $$
 (\beta I+(1-\beta)B)^{-1}
 \le\beta I+(1-\beta)B^{-1}.
 $$
 For a nonconstant auxiliary column $k$, retain the sharper
nonconstant diagonal bound in \eqref{h2:eq:scalar-extra}.
 If $S_V$ is the reference Gram of the original row basis,
 $\Omega=\operatorname{diag}(S_V,I_L)$, $H$ is the full augmented
 Gram, and $I_k$ embeds the row coefficient space in column $k$, then
 \begin{equation}\label{h2:eq:column-estimates}
 I_k^*HI_k\le\Omega,\quad
 I_k^*H^{-1}I_k \le(1+2\epsilon) \Omega^{-1}.
\end{equation}
The first inequality also follows from $W\le1$.

For a cut inside this column, let $P$ retain all preceding complete
columns and let $R$ retain the chosen initial rows of column $k$,
 annihilating all other coordinates. These are coordinate projections,
so $PR=RP=0$ and $P+R$ is a projection. We have
$$
 \norm{P}\le K_s,\quad
 \norm{R}\le\sqrt{1+2\epsilon}\,K_V.
$$
Indeed, coefficient extraction into the $\Omega$ metric has norm
at most $\sqrt{1+2\epsilon}$ by \eqref{h2:eq:column-estimates},
 the row prefix has norm at most $K_V$, and synthesis has norm at
most one. The characters $\zeta_1,\ldots,\zeta_L$ are orthonormal and
orthogonal to $V$ with respect to the reference measure.
Thus, including them does not increase the bound $K_V\ge1$
for the initial projections in the row order.

By \eqref{h2:eq:projection-skew}, every nonzero projection satisfies $\norm{Q}^2=1+\norm{Q-Q^*}^2$. Applying this identity to $P+R$ and using
the triangle inequality for skew-adjoint parts yields
 \begin{equation}\label{h2:eq:selective-prefix}
 \norm{P+R}^2
 \le1 + \left(\sqrt{K_s^2-1}
 +\sqrt{(1+2\epsilon)K_V^2-1}\right)^2.
\end{equation}
 If $P=0$ or $R=0$, the same estimate follows directly
from the bound for the other projection.

Delete the rows indexed by $\zeta_1,\ldots,\zeta_L$
from the auxiliary constant column. The remaining coordinate span
is invariant under the corresponding partial sum projections,
so restriction does not increase their norms. Its initial block
is precisely $V$. Each partial sum projection associated with the
new family fixes $V$ pointwise and therefore induces an operator
of no larger norm on the quotient by $V$.
Since all conditional means lie in $V$, conditional expectation
is the orthogonal projection of each new character onto $V$.
The induced operators are therefore the partial sum projections
of the centered family.

 This proves \eqref{h2:eq:selective-bound}. Cuts inside the
 auxiliary constant column are not needed, so the weaker estimate
for $(B^{-1})_{00}$ causes no loss.

Finally, suppose that $e(-s)=\overline{e(s)}$. Each factor in
\eqref{h2:eq:seed} is unchanged by $(s,z)\mapsto(-s,-z)$, so the
same is true of $V_1$. The scalar weight $w_{\rm aux}$ is even
by construction. Formula \eqref{h2:eq:combined-kernel} therefore
gives a weight invariant under simultaneous inversion. This
completes the proof of \Cref{h2:lem:selective}.

\section{One stage of the construction}\label{h2:sec:extension}

We now combine the finite systems to extend an arbitrary finite
character basis on the line. The extension reduces the distance to
one prescribed character, with a basis bound arbitrarily close to one
 for the new family. The loss of weight is measured against a fixed reference
density, so the same quantity can be summed over the infinite
construction.

 Fix $0<\varepsilon\le1$ and put
\begin{equation}\label{h2:eq:extension-parameters}
u=\varepsilon/100,\quad K_*=1+10u,\quad
c_*=\frac{u}{1+2u},\quad \vartheta_\varepsilon=1-c_*/4<1.
\end{equation}

 \begin{proposition}\label{h2:prop:extension}
Let $w>0$ almost everywhere be even and integrable on $\R$, and let a finite
ordered list of distinct characters span $E$ with basis constant $K$
in $L^2(w)$. Let $f$ be a character outside $E$, and let $w_0$ be
a fixed probability density on $\R$. For every $\delta,\tau>0$,
there are an even positive trigonometric polynomial $0<W\le1$ on $\R$
and a finite appended list of distinct new characters such that,
with $w'=wW$ and $E'$ the enlarged span,
\begin{equation}\label{h2:eq:extension-conclusions}
\int_\R w_0(1-W)<\delta,\quad
\dist_{w'}(f,E')^2\le\frac{1}{2}\dist_w(f,E)^2.
\end{equation}
The basis constant of the concatenated list in $L^2(w')$ is at most
 $\max\{K,K_*\}+\tau$. If the old frequencies and the frequency
of $f$ are rational, the new frequencies can also be rational.
Moreover, given $B>0$, they can have distance greater than $B$
from zero, from all old frequencies, and from each other.
\end{proposition}

We first prove the proposition with the factor $1/2$ replaced by
$\vartheta_\varepsilon$. Repeating the construction finitely many
times then gives the stated estimate. The construction follows
the outline in \Cref{h2:sec:weighted}.

\subsection{An auxiliary character family for the residual}\label{h2:ss:carriers}

Normalize $\rho=w/\int w$ and set
$$
r=f-P_E^\rho f,\quad F=\norm{r}_\rho>0,\quad e_r=r/F.
$$
Write $r=\sum_{j=1}^d c_j\phi_j$ with distinct characters
$\phi_j$, and let $S$ be their Gram matrix in $L^2(\rho)$.
We choose
\begin{equation}\label{h2:eq:tag-z}
0<z< c_u, \quad z\norm{S}<u.
\end{equation}
For each $j$, take a scalar system from \Cref{h2:cor:tune} with
 $d=u$ and $Z=z$, using separate auxiliary variables for different $j$.
Denote its nonconstant characters by $(\psi_{jk})_k$.
Equip the product space with $\rho$ times these normalized scalar
measures. We consider the characters
\begin{equation}\label{h2:eq:intermediate-characters}
(\phi_j\psi_{jk})_{j,k},
\end{equation}
ordered by $j$ and then by the scalar character order. The
centered part of $\phi_j\psi_{jk}$ is
$\phi_j(\psi_{jk}-\mathbb E\psi_{jk})$. Since $|\phi_j|=1$,
its inner products within one $j$ are those of the scalar
centered characters. For different $j$, these inner products
vanish by independence and zero means. Thus, their centered
Gram matrix $C_\oplus$ is a direct sum.
If $T$ is their conditional mean map, \eqref{h2:eq:scalar-gain} gives
$$
TC_\oplus^{-1}T^*=z\sum_{j=1}^dP_{\phi_j},
\quad
\norm{TC_\oplus^{-1/2}}^2=z\norm{S}< u.
 $$
Here every $\phi_j$ has norm one, and the nonzero eigenvalues of
$\sum_jP_{\phi_j}$ are those of $S$. The full Gram matrix of this family is
therefore
\begin{equation}\label{h2:eq:carrier-gram}
H=C_\oplus+T^*T,\quad C_\oplus\le H \le(1+u)C _\oplus.
\end{equation}
In the $C_\oplus$ norm, every partial sum has norm less than $1+u$.
The norm comparison in \eqref{h2:eq:carrier-gram} therefore
gives a basis constant $K_V$ with $K_V^2\le(1+u)^3$. Although the
choice of $z$ depends on the old Gram matrix $S$, this final
bound does not. In particular, it remains valid when the old
characters are almost linearly dependent.

Choose in the $j$-th scalar system a polynomial $g_j$ of mean one and minimum
variance. If $a_j$ and $C_j$ are its scalar mean vector and
centered Gram matrix, its coefficient vector is
$C_j^{-1}a_j/z$. Indeed, its mean is one and its variance is
$a_j^*C_j^{-1}a_j/z^2=1/z$. This attains the minimum in
\eqref{h2:eq:scalar-gain}. We set
\begin{equation}\label{h2:eq:carrier-h}
h=\sum_{j=1}^d c_j\phi_jg_j,\quad
\E[h |  t]=r,\quad
\norm{h}^2=F^2+\frac{1}{z}\sum_{j=1}^d|c_j|^2=:D.
\end{equation}
The variance terms add because their conditional means are zero
and they use independent auxiliary variables. In particular,
\begin{equation}\label{h2:eq:alpha}
e=h/\sqrt D,\quad \E[e |  t]=\alpha e_r,\quad
\alpha=F/\sqrt D>0.
\end{equation}
The polynomial $e$ is a unit vector in the span of \eqref{h2:eq:intermediate-characters} and is
orthogonal to $E$. Although its character expansion depends on the
chosen scalar systems, the number $\alpha$ depends only on $r$ and
$z$. We can therefore make the scalar mass losses arbitrarily small
without changing $\alpha$.

The symmetry needed for a real generator is preserved here.
Because $w$ is even, all inner products of characters in $L^2(\rho)$
are real. The coefficients $c_j$ of $r=f-P_E^\rho f$ are therefore
real, so $r(-t)=\overline{r(t)}$. The even scalar weights also
have real means and real Gram matrices. The mnimizing polynomial
$g_j$ can consequently be chosen with real character coefficients.
It follows from \eqref{h2:eq:carrier-h} that
$h(-t,-y)=\overline{h(t,y)}$, and the same identity holds for $e$.

\subsection{Independent copies of the auxiliary system}\label{h2:ss:replicas}

Set $\sigma=u$ and choose
\begin{equation}\label{h2:eq:N-replicas}
N= \lceil\alpha^{-2}\rceil.
\end{equation}
Take $N$ independent copies of all the scalar systems in \Cref{h2:ss:carriers}.
Their normalized polynomials can be written as
\begin{equation}\label{h2:eq:replica-decomposition}
e_i=\alpha e_r+\xi_i,\quad
\xi_i\perp L^2(\rho),\quad
\xi_i\perp\xi_j\ (i\ne j),\quad
\norm{\xi_i}^2=1-\alpha^2.
\end{equation}
Thus the different copies have a common component in the direction
$e_r$ and mutually orthogonal remaining components.

Consider the positive operator
\begin{equation}\label{h2:eq:replica-A}
A=\sigma\sum_{i=1}^NP_{e_i},\quad \beta=N\sigma\alpha^2.
\end{equation}
Every $e_i$ is orthogonal to $E$, so $AP_E=0$. By
\eqref{h2:eq:replica-decomposition}, the Gram matrix of $(e_i)$
has diagonal entries $1$ and off-diagonal entries $\alpha^2$.
Its largest eigenvalue is $1+(N-1)\alpha^2$, so
\begin{equation}\label{h2:eq:A-norm}
\norm{A}=\sigma\bigl(1+(N-1)\alpha^2\bigr)
=\beta+ u (1-\alpha^2)< 2u,
\end{equation}
where we used $u\le\beta<u(1+\alpha^2)$. Since $A$ is positive,
$$
A(I+A)^{-1}\ge\frac{A}{1+\norm{A}}.
$$
Pairing with $e_r$ and using $\ip{Ae_r}{e_r}=\beta\ge u$ gives
\begin{equation}\label{h2:eq:tagcapture}
\ip{A(I+A)^{-1}e_r}{e_r}
\ge\frac{\beta}{1+\norm{A}}\ge \frac{u}{1+2u}=c_*.
\end{equation}

Equation \eqref{h2:eq:tagcapture} gives a positive lower bound
for $\ip{A(I+A)^{-1}e_r}{e_r}$ depending only on $\varepsilon$,
even when $\alpha$ is very small. This is why we use independent
copies. The number $N$
may depend on the residual, but the operator norm in
\eqref{h2:eq:A-norm} does not.

\subsection{Replacing the polynomials by new characters}

Let $V_i$ be the character span of the $i$-th copy of \eqref{h2:eq:intermediate-characters}. Its
basis constant satisfies $K_V^2\le(1+u)^3$. Apply \Cref{h2:lem:selective} to
$V_i$ and $e_i$ with $K_s=1+u$, using the product of $\rho$ and all the scalar measures for the variables already present.
Write $c_i,j_i,M_i$ for its parameters. Repeat its new character
family
$$
n_i=\lceil \sigma/j_i\rceil
$$
times in conditionally independent new variables, and concatenate
all these lists. The parts with conditional mean zero are orthogonal between
different copies. Their operators $J$ therefore add, as explained
after \eqref{h2:eq:info}, giving
\begin{equation}\label{h2:eq:J-limit}
J=\sum_{i=1}^N n_i
\bigl(c_iP_{V_i\cap e_i^\perp}+j_iP_{e_i}\bigr)
\longrightarrow A
\end{equation}
in operator norm as, for each $i$, we take successive systems
from the sequence in \Cref{h2:lem:selective}.

For each $i$, this is the repetition in
\eqref{h2:eq:selective-repetition}. There are only finitely many
$i$, and the total loss of weight satisfies
\begin{equation}\label{h2:eq:engine-cost}
\sum_{i=1}^N n_i(1-M_i)
\le\sum_{i=1}^N(\sigma+j_i)\frac{1-M_i}{j_i}\longrightarrow0.
\end{equation}
The number $N$ is fixed before these limits are taken.

 Choose the auxiliary mass losses so small that $2\epsilon\le u$.
Since $u\le1/100$, we have
$$
K_s^2-1\le3u,\quad
(1+2\epsilon)K_V^2-1\le(1+u)^4-1\le5u.
$$
 By \eqref{h2:eq:selective-bound}, the centered squared basis
constant of each new family is less than
$1+(\sqrt3+\sqrt5)^2u<1+16u$.
Orthogonality of their centered spaces gives the same bound for
the concatenation. By \eqref{h2:eq:J-limit}, we may choose the finite systems
so that $\norm{J}<3u$.
If $C$ is the full centered Gram and $T$ the full mean map, then
\begin{equation}\label{h2:eq:new-norm-comparison}
C\le C+T^*T \le(1+3u) C.
\end{equation}
Thus the actual new character family has basis constant satisfying
 \begin{equation}\label{h2:eq:new-constant}
 K_{\rm new}^2 < (1+16u)(1+3u)
=1+19u+48u^2\le1+20u<K_*^2.
 \end{equation}
We append the new characters from \Cref{h2:lem:selective}. The polynomials $e_i$ and the characters
in \eqref{h2:eq:intermediate-characters} served only to construct them.

\subsection{Combining the old and new spans}

Let $F_{\rm new}$ be the span of the new characters in the normalized
product measure. For unit vectors in $E$ and $F_{\rm new}$, let
$\rho_*$ be the supremum of the absolute inner product. By
\eqref{h2:eq:info},
\begin{equation}\label{h2:eq:angle}
\rho_*^2=\norm{P_EJ(I+J)^{-1}P_E}\longrightarrow0.
\end{equation}
The limit is zero because $J\to A$ and $AP_E=0$.
If $\rho_*<1$, then for $u\in E$ and $v\in F_{\rm new}$,
\begin{equation}\label{h2:eq:angle-comparison}
(1-\rho_*)(\norm{u}^2+\norm{v}^2)
\le\norm{u+v}^2
\le(1+\rho_*)(\norm{u}^2+\norm{v}^2).
\end{equation}
In the orthogonal sum norm, listing the old characters before the
new ones gives a basis constant at most $\max\{K,K_*\}$. Therefore
\eqref{h2:eq:angle-comparison} bounds its basis constant by
$$
\sqrt{\frac{1+\rho_*}{1-\rho_*}}\max\{K, K_* \}.
$$
We choose the finite systems so that this is less than
 $\max\{K,K_*\}+\tau/2$.

The approximation estimate follows from the same operator $J$. By \eqref{h2:eq:tagcapture} and \eqref{h2:eq:J-limit},
we may choose the finite systems so that the squared norm
of the orthogonal projection of $r$ onto $F_{\rm new}$
is greater than $(c_*/2)F^2$. Adjoining $E$ can only
improve approximation, and $f=P_E^\rho f+r$. Hence, in the normalized
product measure,
\begin{equation}\label{h2:eq:product-error}
\dist(f,E+F_{\rm new})^2< (1-c_*/2) F^2.
\end{equation}
The old marginal is still $\rho$, so the residual used here is
exactly the one fixed in \Cref{h2:ss:carriers}.

\subsection{The order of the choices and the loss of weight}

We check that the preceding choices give both the required
approximation and an arbitrarily small loss of weight. We first
fix $r$ and $z$. These determine $\alpha$ in
\eqref{h2:eq:alpha} and the number $N$ in
\eqref{h2:eq:N-replicas}. We then choose the scalar systems from
\Cref{h2:cor:tune} with sufficiently small total loss of mass.
Their variances remain $1/z$, so $\alpha$ and $N$ are unchanged.
This fixes every $V_i$ and $e_i$. Finally, we choose the systems
from \Cref{h2:lem:selective} and repeat them as above to obtain
\eqref{h2:eq:J-limit}, \eqref{h2:eq:product-error}, and
\eqref{h2:eq:engine-cost}. Thus, the spaces $V_i$, the vectors $e_i$,
and the number of copies $N$ are fixed before we choose the systems
used in \eqref{h2:eq:J-limit}.

Let $\mathcal W(t,y)$ denote the product of all the scalar weights
and all the subsequent conditional weights, with $y$ collecting
every auxiliary variable. It is a positive trigonometric polynomial
with $\mathcal W\le1$. The scalar weights are even and each
$e_i$ satisfies $e_i(-t,-y)=\overline{e_i(t,y)}$. The last
assertion of \Cref{h2:lem:selective} therefore makes $\mathcal W$
invariant under simultaneous inversion of all its variables. Each subsequent weight has constant
conditional mass relative to the product of the scalar spaces, and each scalar weight is
independent of $t$. Therefore
\begin{equation}\label{h2:eq:total-mass}
\int\mathcal W(t,y)\,dy=\mathfrak m
\end{equation}
is independent of $t$. More explicitly, if $m_j$ is the mass
of the $j$-th scalar weight in one copy, then
$$
\mathfrak m=\left(\prod_{j=1}^d m_j\right)^N
\prod_{i=1}^N M_i^{n_i},\quad
1-\mathfrak m\le N\sum_{j=1}^d(1-m_j)
 +\sum_{i=1}^N n_i(1-M_i).
$$
The first sum can be made arbitrarily small after $N$ is fixed,
and the second tends to zero by \eqref{h2:eq:engine-cost}.
In particular,
$$
\int_\R\int w_0(t)(1-\mathcal W(t,y))\,dy\,dt
=1-\mathfrak m
$$
can be made less than $\delta/2$.

The normalized product measure used above is
$\rho(dt)\mathcal W(t,y)\,dy/\mathfrak m$. Returning to the unnormalized
measure $w(t)\mathcal W(t,y)\,dt\,dy$ multiplies squared norms by
$(\int w)\mathfrak m$. Thus \eqref{h2:eq:product-error} becomes
\begin{equation}\label{h2:eq:raw-product-error}
\dist_{w\mathcal W}(f,E+F_{\rm new})^2
< (1-c_*/2) \mathfrak m\,\dist_w(f,E)^2
 \le(1-c_*/2) \dist_w(f,E)^2.
\end{equation}
This leaves a strict margin for the transfer to the line.

\subsection{Choosing the real frequencies}\label{h2:ss:transfer}

We now replace the finitely many auxiliary variables by functions
of $t$. Write each character as $e_\nu(t)e^{2\pi i k\cdot y}$,
where $k\in\Z^D$. We first check that the appended characters
have distinct nonzero vectors $k$. This will allow us to separate their frequencies on the line
while approximating the inner products of the product-space
construction arbitrarily closely.

Within one application of \Cref{h2:lem:selective}, the new
characters are $\phi_j\chi_k$ and $\zeta_\ell\chi_k$.
Here the $\phi_j$ belong to one copy of
\eqref{h2:eq:intermediate-characters}, so they already have distinct nonzero frequency vectors in their scalar variables. The $\zeta_\ell$ use separate variables, and the $\chi_k$ are distinct nonconstantt characters in a further set of variables. Different applications use disjoint sets of these last variables. Thus every new vector $k$ is nonzero, and the difference between any two new vectors is also nonzero.

Collect these vectors and their differences, together with every nonzero auxiliary frequency appearing when we expand the Gram entries, cross-moments with $f$, and reference loss integral. This is a finite subset of $\Z^D\setminus\{0\}$. Choose $\omega\in\Z^D$ whose inner product with each of these vectors is nonzero. For example, $\omega=(1,M,\ldots,M^{D-1})$ works for all but finitely many positive integers $M$, since each excluded equality is a nonzero polynomial equation in $M$. We substitute
\begin{equation}\label{h2:eq:transfer}
y=A\omega t,
\end{equation}
where $A$ tends to infinity through positive integers.
For $a\in L^1(\R)$ and any nonzero auxiliary frequency under
consideration, the Riemann--Lebesgue lemma gives
$$
\int_\R a(t)e^{2\pi i(\nu+A k\cdot\omega)t}\,dt
\longrightarrow0.
$$
The terms with $k=0$ are unchanged. Consequently, every required
moment converges to its product-space value.

At the same time, the absolute values of the new frequencies
$\nu+A k\cdot\omega$, their distances from the old list, and
their mutual distances tend to infinity. If the old frequencies and the target frequency are rational, then all the frequencies $\nu$ used in the finite construction are rational. Taking $A$ and $\omega$ integral therefore preserves rationality.

Set $W(t)=\mathcal W(t,A\omega t)$. We have $0<W\le1$, and
simultaneous inversion gives $W(-t)=W(t)$. The limiting Gram
matrix of the old and new characters is positive definite.
Its inverse, coordinate projection norms, and approximation error
depend continuously on the finitely many moments. We therefore
choose $A$ large enough that the reference loss is less than
$\delta$, the basis constant is at most $\max\{K,K_*\}+\tau$,
and the squared approximation error is at most
 $\vartheta_\varepsilon\dist_w(f,E)^2$. The strict margins in
\eqref{h2:eq:raw-product-error} and the preceding projection
estimate permit all three choices simultaneously. Increasing $A$
further gives the prescribed frequency separations.

Finally, choose an integer $m$ with $\vartheta_\varepsilon^m\le1/2$
and repeat this construction for the same target $m$ times,
using the parameters $\delta/m$ and $\tau/m$ in place of
$\delta$ and $\tau$ at each step.
If the target enters the span earlier, leave the remaining steps
unchanged. The product of the weights is positive, even, and bounded
by one. Its loss against $w_0$ is at most the sum of the individual
losses, by $1-\prod_iW_i\le\sum_i(1-W_i)$.
The partial sum bound is at most $\max\{K,K_*\}+\tau$,
and the squared distance decreases by a factor at least two.
Each step separates its new frequencies from all earlier ones.
 This proves \Cref{h2:prop:extension}.

\section{Assembling the basis}\label{h2:sec:assembly}

We now iterate \Cref{h2:prop:extension}. The finite projection bound
will pass to the limiting weight. To obtain completeness, we use a
countable total family of target characters and reduce each of their
approximation errors infinitely often.

\subsection{The stages and the positive limiting weight}

 It suffices to consider $0<\varepsilon\le1$. Use $K_*=1+\varepsilon/10$ as in
\eqref{h2:eq:extension-parameters}. Start with the even Gaussian probability density $w_0(t)=\pi^{-1/2}e^{-t^2}$ and the
single character $1$. Choose positive sequences $(\delta_n)$ and
$(\tau_n)$ such that
\begin{equation}\label{h2:eq:budgets}
\sum_{n=1}^\infty\delta_n<\frac{1}{2},\quad
\sum_{n=1}^\infty\tau_n< \varepsilon/4.
 \end{equation}
Enumerate the rational-frequency characters as $(f_j)_{j\ge1}$.
Choose the character to be approximated according to the order
$$
f_1;\quad f_1,f_2;\quad f_1,f_2,f_3;\quad\ldots.
$$
Thus each character is selected infinitely often.

Suppose that the current ordered list spans $E_{n-1}$ under the
weight $w_{n-1}$. If the selected character already belongs to this
span, set $W_n=1$ and leave the list unchanged. Otherwise apply
\Cref{h2:prop:extension} with parameters $\delta_n,\tau_n$ and the
fixed reference density $w_0$, choosing the new frequencies at
distance greater than $n$ from zero, the old frequencies, and each
other. Append its new characters, denote
the enlarged span by $E_n$, and set
\begin{equation}\label{h2:eq:weights}
w_n=w_{n-1}W_n=w_0\prod_{k=1}^nW_k.
\end{equation}
The finite list at stage $n$ has basis constant at most
\begin{equation}\label{h2:eq:finite-bound}
 K_* +\sum_{k=1}^n\tau_k \le K_\infty:=K_*+\sum_{k=1}^\infty\tau_k < 1+\varepsilon.
 \end{equation}
No earlier character is changed by the extension. All frequencies
are rational, and all weights are even. Outside the characters
chosen in the first $n$ stages, every pair of frequencies has
distance greater than $n$. This also proves the zero-density
assertion in \Cref{rem:translation-sets}: after removing finitely
many frequencies, an interval of length $R$ contains at most
$1+R/n$ of the remaining ones. First let $R\to\infty$ and then
$n\to\infty$.

The weights decrease pointwise. Moreover,
$$
\sum_{n=1}^\infty\int_\R w_0(1-W_n)
<\sum_{n=1}^\infty\delta_n<\infty.
$$
Tonelli's theorem gives $\sum_n(1-W_n(t))<\infty$ for almost
every $t$, since $w_0>0$ everywhere. At such a point all factors
are positive, and eventually $W_n(t)\ge1/2$. Using
$-\log x\le2(1-x)$ for $1/2\le x\le1$, the infinite product is
strictly positive. Hence
\begin{equation}\label{h2:eq:limit-weight}
w_\infty=w_0\prod_{n=1}^\infty W_n>0
\quad\text{almost everywhere}.
\end{equation}
Also, $0<w_\infty\le w_n\le w_0$ and dominated convergence gives
$w_n\to w_\infty$ in $L^1$. The elementary inequality
$1-\prod_{k=1}^nW_k\le\sum_{k=1}^n(1-W_k)$ yields the useful
mass bound
\begin{equation}\label{h2:eq:limit-mass}
\int_\R w_\infty\ge1-\sum_{n=1}^\infty\delta_n>\frac{1}{2}.
\end{equation}

\subsection{Passing the projection bounds to the limit}

For each fixed finite combination of constructed characters, its
squared norm under $w_n$ converges to its squared norm under
$w_\infty$, because characters are bounded and $w_n\to w_\infty$
in $L^1$. Therefore every finite coordinate inequality in
\eqref{h2:eq:finite-bound} passes to the limit. More explicitly, if
the constructed order is $(e_{\mu_j})$ and $m\le l$, then
\begin{equation}\label{h2:eq:limit-prefix}
\left\|\sum_{j=1}^m a_je_{\mu_j}\right\|_{w_\infty}
 \le K_\infty \left\|\sum_{j=1}^l a_je_{\mu_j}\right\|_{w_\infty}.
\end{equation}
To prove this, choose a stage containing the first $l$ characters,
apply the bound at every later stage, and pass to the limit.

Characters with distinct real frequencies remain linearly independent under
$w_\infty$, since the weight is positive almost everywhere.
Indeed, a finite exponential polynomial that vanishes in this norm
vanishes almost everywhere on $\R$, hence everywhere by continuity;
the usual linear independence of distinct exponentials then makes
all its coefficients zero. Thus the coordinate maps on the finite
span are well defined and satisfy \eqref{h2:eq:limit-prefix}.

\subsection{Completeness and convergence of the partial sums}

For a fixed rational-frequency character $f$, set
$$
D_n(f)=\dist_{w_n}(f,E_n)^2.
$$
These errors decrease at every stage, since $w_n\le w_{n-1}$ and
$E_{n-1}\subseteq E_n$. Whenever $f$ is selected, its squared approximation error
is either zero or decreases by a factor of at least two.
Since $f$ is selected infinitely often, $D_n(f)\to0$.
Let $M$ be the closed span of all constructed characters in
$L^2(w_\infty)$. From $w_\infty\le w_n$, we get
$$
\dist_{w_\infty}(f,M)^2\le D_n(f)\longrightarrow0.
$$
Every rational-frequency character therefore belongs to $M$.

These characters are total in $L^2(w_\infty)$. To see this, suppose
$h$ is orthogonal to all of them. The Cauchy--Schwarz inequality gives $h w_\infty\in L^1(\R)$. Its Fourier transform is continuous and vanishes at every rational number, hence everywhere. Uniqueness of the Fourier tranform on $L^1$ implies $h w_\infty=0$ almost
everywhere. By positivity of the weight, $h=0$. Thus $M=L^2(w_\infty)$. In particular, the constructed character list is infinite.

We now use completeness to show that the sequence of partial sums converges in norm for every function in $L^2(w_\infty)$. By \eqref{h2:eq:limit-prefix}, the partial sum maps on the finite span extend to operators $P_m$ on $L^2(w_\infty)$ with $\norm{P_m}\le K_\infty$. On a finite linear combination of the characters, $P_m$ eventually equals the identity. Given $f$ and a finite combination $h$, for all sufficiently large $m$ we have
$$
\norm{f-P_mf}_{w_\infty}
\le(1+ K_\infty )\norm{f-h}_{w_\infty}.
$$
Density therefore proves $P_mf\to f$ for every $f$.
Differences of consecutive partial sums are bounded coordinate
projections. Applying one of them to a convergent expansion recovers
its coefficient, proving uniqueness. This proves
\Cref{h2:thm:weighted}.

\subsection{Returning to translates}

 \begin{proof}[Proof of \Cref{thm:translates} and \Cref{thm:translate-properties} for $p=2$]
 Take $w_\infty$ and $(\mu_n)$ from \Cref{h2:thm:weighted}, and define
$g\in L^2(\R)$ by $\widehat g=\sqrt{w_\infty}$. Multiplication
by $\sqrt{w_\infty}$ is unitary from $L^2(w_\infty)$ onto
$L^2(\R)$: its inverse is division by $\sqrt{w_\infty}$, which
is defined almost everywhere by \eqref{h2:eq:limit-weight}.
Consequently, $(e_{\mu_n}\sqrt{w_\infty})$ is a Schauder basis
of $L^2(\R)$ with the same basis constant. Inverse Fourier transform
sends its $n$-th vector to $g(\,\cdot+\mu_n)$. Taking
$\lambda_n=-\mu_n$ gives the required basis.

The weight $w_\infty$ is even and bounded by the even Gaussian
$w_0$. Hence $\sqrt{w_\infty}$ is real, even, and integrable, so
its inverse Fourier transform $g$ is real and even. It is nonzero,
since Plancherel's theorem and \eqref{h2:eq:limit-mass} give
$\norm{g}_2^2=\int w_\infty>1/2$.

 For every $z\in\C$, the integral
$$
g(z)=\int_\R\sqrt{w_\infty(t)}e^{2\pi izt}\,dt
$$
converges absolutely and locally uniformly in $z$, as do the
integrals obtained by differentiating any number of times.
Thus $g$ extends to an entire function. For every integer $k\ge0$, the function $(2\pi it)^k\sqrt{w_\infty(t)}$ belongs to both $L^1$ and $L^2$. The Riemann--Lebesgue lemma and Plancherel's theorem therefore give $g^{(k)}\in C_0(\R)\cap L^2(\R)$.

All translates of $g$ are real. If a real function is expanded in this complex basis, conjugating its expansion and using uniqueness shows that every coefficient is real. Thus the same sequence is a basis over the real scalar field as well.
\end{proof}

\begin{remark}
The basis constant one cannot occur for a basis of translates in $L^2(\R)$.
Suppose a Schauder basis of a Hilbert space has basis constant one.
Every nonzero initial projection is contractive, and
\eqref{h2:eq:projection-skew} makes it orthogonal. Each later basis vector lies in the kernel of every earlier initial projection, so the basis is orthogonal. If the vectors are translates of one
function, they all have the same nonzero norm; after normalizing
the generator, they form an orthonormal basis.

Assume therefore that $(T_{\lambda_n}g)$ is an orthonormal basis of $L^2(\R)$ and $\norm{g}_2=1$, where $T_\lambda g(x)=g(x-\lambda)$. Strong continuity of translations gives $\langle g,T_tg\rangle\to1$ as $t\to0$. Orthogonality therefore forces the parameters to be uniformly separated by some $d>0$. Choose a bounded interval $I$ of positive length less than $d$. The intervals $I-\lambda_n$ are pairwise disjoint up to null
sets, and hence
\begin{equation}\label{h2:eq:local-trace}
 \sum_n\norm{\mathbf1_I T_{\lambda_n}g}_2^2
 =\sum_n\int_{I-\lambda_n}|g(x)|^2\,dx\le1.
\end{equation}
Choose two orthonormal functions $h_1,h_2$ supported in $I$.
Parseval’s identity for this basis of translates and Bessel's inequality
for this pair imply
$$
 2=\sum_n\sum_{j=1}^2
 |\langle h_j,T_{\lambda_n}g\rangle|^2
 \le\sum_n\norm{\mathbf1_I T_{\lambda_n}g}_2^2\le1,
$$
a contradiction. The argument applies over either scalar field
and proves that the infimum one is not attained.

\end{remark}

\section*{Acknowledgments}
L.L.~is grateful to the Azrieli Foundation for the award of an Azrieli Fellowship and acknowledges the support of this research by ISF Grant No.~854/25.


\begin{thebibliography}{99}

\bibitem{AO}
A.~Atzmon and A.~Olevskii,
\emph{Completeness of integer translates in function spaces on $\R$},
J. Approx. Theory \textbf{87} (1996), no.~3, 291--327.
\href{https://doi.org/10.1006/jath.1996.0106}{doi:10.1006/jath.1996.0106}.

\bibitem{BVV}
\'A.~Baricz, J.~Vesti, and M.~Vuorinen,
\emph{On Kaluza's sign criterion for reciprocal power series},
Ann. Univ. Mariae Curie-Sk{\l}odowska Sect. A
\textbf{65} (2011), no.~2, 1--16.

\bibitem{BHW}
J.~J.~Benedetto, C.~Heil, and D.~F.~Walnut,
\emph{Differentiation and the Balian-Low theorem},
J. Fourier Anal. Appl. \textbf{1} (1995), no.~4, 355--402.

\bibitem{BC}
M.~Berasategui and D.~Carando,
\emph{Unconditional Schauder frames of translates in $L_p(\R^d)$},
Israel J. Math. \textbf{238} (2020), no.~2, 687--713.
\href{https://doi.org/10.1007/s11856-020-2041-9}{doi:10.1007/s11856-020-2041-9}.

\bibitem{BOU}
J.~Bruna, A.~Olevskii, and A.~Ulanovskii,
\emph{Completeness in $L^1(\R)$ of discrete translates},
Rev. Mat. Iberoam. \textbf{22} (2006), no.~1, 1--16.
\href{https://doi.org/10.4171/RMI/447}{doi:10.4171/RMI/447}.

\bibitem{Dau}
I.~Daubechies,
\emph{Orthonormal bases of compactly supported wavelets},
Comm. Pure Appl. Math. \textbf{41} (1988), no.~7, 909--996.

\bibitem{DH}
B.~Deng and C.~Heil,
\emph{Density of Gabor Schauder bases},
in \emph{Wavelet Applications in Signal and Image Processing VIII},
Proc. SPIE \textbf{4119} (2000), 153--164.
\href{https://doi.org/10.1117/12.408600}{doi:10.1117/12.408600}.

\bibitem{DS}
B.~Dharra and S.~Sivananthan,
\emph{Schauder frames and completeness of translates in the Orlicz space},
Banach J. Math. Anal. \textbf{17} (2023), no.~3, Paper No.~52.
\href{https://doi.org/10.1007/s43037-023-00275-0}{doi:10.1007/s43037-023-00275-0}.

\bibitem{DO}
L.~E.~Dor and E.~Odell,
\emph{Monotone bases in $L_p$},
Pacific J. Math. \textbf{60} (1975), no.~2, 51--61.

\bibitem{FOSZ}
D.~Freeman, E.~Odell, Th.~Schlumprecht, and A.~Zs\'ak,
\emph{Unconditional structures of translates for $L_p(\R^d)$},
Israel J. Math. \textbf{203} (2014), no.~1, 189--209.

\bibitem{FPT}
D.~Freeman, A.~M.~Powell, and M.~A.~Taylor,
\emph{A Schauder basis for $L_2$ consisting of non-negative functions},
Math. Ann. \textbf{381} (2021), no.~1-2, 181--208.

\bibitem{JS}
W.~B.~Johnson and G.~Schechtman,
\emph{A Schauder basis for $L_1(0,\infty)$ consisting of non-negative functions},
Illinois J. Math. \textbf{59} (2015), no.~2, 337--344.

\bibitem{LO}
N.~Lev and A.~Olevskii,
\emph{Wiener's `closure of translates' problem and Piatetski-Shapiro's
uniqueness phenomenon},
Ann. of Math. (2) \textbf{174} (2011), no.~1, 519--541.
\href{https://doi.org/10.4007/annals.2011.174.1.15}{doi:10.4007/annals.2011.174.1.15}.

\bibitem{LT-no}
N.~Lev and A.~Tselishchev,
\emph{There are no unconditional Schauder frames of translates in
$L^p(\R)$, $1\le p\le2$},
Adv. Math. \textbf{460} (2025), Paper No.~110036, 11~pp.
\href{https://doi.org/10.1016/j.aim.2024.110036}{doi:10.1016/j.aim.2024.110036}.

\bibitem{LT2}
N.~Lev and A.~Tselishchev,
\emph{Schauder frames of discrete translates in $L^2(\R)$},
J. Funct. Anal. \textbf{290} (2026), no.~7, Paper No.~111318, 15~pp.
\href{https://doi.org/10.1016/j.jfa.2025.111318}{doi:10.1016/j.jfa.2025.111318}.

\bibitem{LTp}
N.~Lev and A.~Tselishchev,
\emph{Schauder frames of discrete translates in $L^p(\R)$},
Rev. Mat. Iberoam. \textbf{42} (2026), no.~2, 573--590.
\href{https://doi.org/10.4171/RMI/1602}{doi:10.4171/RMI/1602}.

\bibitem{LT-sparse}
N.~Lev and A.~Tselishchev,
\emph{Completeness of sparse, almost integer, and finite local complexity
sequences of translates in $L^p(\R)$},
J. Lond. Math. Soc. \textbf{113} (2026), no.~5, Paper No.~e70562.
\href{https://doi.org/10.1112/jlms.70562}{doi:10.1112/jlms.70562}.

\bibitem{LTclassic}
J.~Lindenstrauss and L.~Tzafriri,
\emph{Classical Banach Spaces I and II: Sequence Spaces and Function Spaces},
Classics in Mathematics, Springer-Verlag, Berlin, 1996.

\bibitem{LW}
Y.~Liu and Y.~Wang,
\emph{The uniformity of non-uniform Gabor bases},
Adv. Comput. Math. \textbf{18} (2003), no.~2-4, 345--355.

\bibitem{Mallat}
S.~G.~Mallat,
\emph{Multiresolution approximations and wavelet orthonormal bases of $L^2(\R)$},
Trans. Amer. Math. Soc. \textbf{315} (1989), no.~1, 69--87.

\bibitem{NV}
N.~Nikolski and A.~Volberg,
\emph{On the sign distributions of Hilbert space frames},
Anal. Math. Phys. \textbf{9} (2019), no.~3, 1115--1132.
\href{https://doi.org/10.1007/s13324-019-00304-y}{doi:10.1007/s13324-019-00304-y}.

\bibitem{NO07}
S.~Nitzan-Hahamov and A.~Olevskii,
\emph{Sparse exponential systems: completeness with estimates},
Israel J. Math. \textbf{158} (2007), 205--215.
\href{https://doi.org/10.1007/s11856-007-0010-1}{doi:10.1007/s11856-007-0010-1}.

\bibitem{NO}
S.~Nitzan and A.~Olevskii,
\emph{Quasi-frames of translates},
C. R. Math. Acad. Sci. Paris \textbf{347} (2009), no.~13-14, 739--742.
\href{https://doi.org/10.1016/j.crma.2009.04.001}{doi:10.1016/j.crma.2009.04.001}.

\bibitem{OSSZ}
E.~Odell, B.~Sari, Th.~Schlumprecht, and B.~Zheng,
\emph{Systems formed by translates of one element in $L_p(\R)$},
Trans. Amer. Math. Soc. \textbf{363} (2011), no.~12, 6505--6529.

\bibitem{Ol97}
A.~Olevskii,
\emph{Completeness in $L^2(\R)$ of almost integer translates},
C. R. Acad. Sci. Paris S\'er. I Math. \textbf{324} (1997), no.~9, 987--991.
\href{https://doi.org/10.1016/S0764-4442(97)87873-1}{doi:10.1016/S0764-4442(97)87873-1}.

\bibitem{OU}
A.~Olevskii and A.~Ulanovskii,
\emph{Almost integer translates. Do nice generators exist?},
J. Fourier Anal. Appl. \textbf{10} (2004), no.~1, 93--104.
\href{https://doi.org/10.1007/s00041-004-8006-2}{doi:10.1007/s00041-004-8006-2}.

\bibitem{OZ}
T.~E.~Olson and R.~A.~Zalik,
\emph{Nonexistence of a Riesz basis of translates},
in \emph{Approximation Theory (Memphis, TN, 1991)},
Lecture Notes in Pure and Appl. Math., vol.~138,
Marcel Dekker, New York, 1992, 401--408.

\bibitem{PS}
A.~M.~Powell and A.~H.~Spaeth,
\emph{Nonnegativity constraints for structured complete systems},
Trans. Amer. Math. Soc. \textbf{368} (2016), no.~8, 5783--5806.

\bibitem{Yu}
P.-T.~Yu,
\emph{Existence of unconditional Schauder frames of translates in modulation spaces},
J. Geom. Anal. \textbf{35} (2025), no.~8, Paper No.~239.
\href{https://doi.org/10.1007/s12220-025-02070-9}{doi:10.1007/s12220-025-02070-9}.

\end{thebibliography}
\end{document}